\documentclass[a4paper,12pt,reqno]{amsart}
\usepackage{amssymb,amsmath,array,amscd,amsthm,hhline}

\usepackage[mathscr]{euscript}
\usepackage{stmaryrd}
\usepackage{ulem}

\usepackage{tikz-cd}

\usepackage{mathrsfs}

\def\({\left(}
\def\){\right)}

\def\ltimes#1{\mathop{{}_#1\times}}
\def\rtimes#1{{}_#1{\times}\,}

\def\can{\mathop{\mathbf{can}}\nolimits}
\def\n{\mathfrak n}

\def\lm{\lambda}

\def\C{\mathbb C}

\def\R{\mathcal R}

\def\Gr{\mathop{\rm Gr}\nolimits}
\def\dlim{\mathop{\rm lim}\limits_{\longrightarrow}}

\def\<{\langle}
\def\>{\rangle}

\renewcommand\emptyset{\varnothing}
\renewcommand\phi{\varphi}

\def\im{\mathop{\rm im}}

\def\pr{\mathop{\rm pr}\nolimits}
\def\id{\mathrm{id}}
\renewcommand\epsilon{\varepsilon}

\def\A{S}
\def\B{T}
\def\C{F}
\def\suchthat{\mathbin{\rm |}}

\def\ito{\stackrel\sim\to}
\def\b{t}

\def\p{\mathbf p}
\def\u{\mathbf u}
\def\Mat{\mathop{\mathrm{Mat}}\nolimits}
\def\SU{\mathop{\rm SU}\nolimits}
\def\rr{\rrbracket}

\newtheorem{theorem}{Theorem}
\newtheorem{proposition}[theorem]{Proposition}
\newtheorem{lemma}[theorem]{Lemma}
\newtheorem{corollary}[theorem]{Corollary}
\newtheorem{remark}[theorem]{Remark}

\def\Eu{\mathop{\rm Eu}\nolimits}

\def\pt{{\rm pt}}
\def\k{\Bbbk}

\def\le{\leqslant}
\def\ge{\geqslant}

\renewcommand{\labelenumi}{{\rm\theenumi}}
\renewcommand{\theenumi}{{\rm(\arabic{enumi})}}

\def\={\equiv}

\def\BS{\mathrm{BS}}

\def\lq{\mathop{\backslash}}
\def\gld{\mathop{\mathrm{gld}}}

\def\g{\mathbf g}
\def\h{\mathbf h}

\author{Vladimir Shchigolev}
\title{Twisted actions on cohomologies and bimodules}

\begin{document}

\begin{abstract} We introduce a twisted action of the equivariant cohomology of the singleton $H_R^\bullet(\pt,\k)$
on the equivarinat cohomology $H_L^\bullet(X,\k)$ of an $L$-space $X$.
Considering this actions as a right action, $H_L^\bullet(X,\k)$ becomes a bimodule togeather
with the canonical left action of $H_L^\bullet(\pt,\k)$. Using this bimodule structure,
we prove an equivariant version of the K\"unneth isomorphism.

We apply this result to the computation of the equivariant cohomologies of Bott-Samelson varieties
and to a geometric construction of the bimodule morphisms between these cohomologies.

\end{abstract}

\maketitle

%Keywords: %\MSC[2008] 55N91

\bigskip
\bigskip
\medskip
\section{Introduction} Let $L$ be a topological group acting continuously on a topological space $X$.
For any commutative ring $\k$, the $L$-equivariant (sheaf) cohomology $H^\bullet_L(X,\k)$ is naturally
a left $H^\bullet_L(\pt,\k)$-module, where $\pt$ is the singleton trivially acted upon by $L$.
This module structure can explicitly be described as follows:

\vspace{-15pt}

$$
cm=a_X^\star(c)\cup m,
$$

\vspace{2pt}

\noindent
where $c\in H^\bullet_L(\pt,\k)$, $m\in H^\bullet_L(X,\k)$, $a_X:X\to\pt$ is the constant map,
$a_X^\star$ is the equivariant pull-back and $\cup$ denotes the cup product.

The main idea of this paper is to define the structure
of a right $H^\bullet_R(\pt,\k)$-module on $H^\bullet_L(X,\k)$ possibly for $R\ne L$.
Of course, if $L=R$ and we do not have any additional information about $X$,
then we can do it, simply setting
$
md=m \cup a_X^\star(d)
$.
%This case is indeed a very special case of our construction of right modules, see Remark~\ref{remark:2}.
This construction is however very restrictive.
Therefore, we define the right $H^\bullet_R(\pt,\k)$-module structure on $H^\bullet_L(X,\k)$
with the help of an $L$-equivariant map $A:X\to G/R$, which we call the {\it twisting map},
see Section~\ref{Twisted action}. This definition, see~(\ref{eq:i}), assumes that the groups $L$ and $R$
are contained in a larger group $G$, which acts on the total space $E$ of a universal principal $L$-bundle.
Therefore, we define the equivariant cohomology $H^\bullet_L(X,\k)$ not just with the help
of any universal principal $L$-bundle $E\to L\lq E$ but only with the help of such a bundle
that bears a continuous action of $G$. This construction is possible at least in the case where $G$ is a compact Lie group
and $L$ and $R$ are its closed subgroups. We believe that twisting maps are omnipresent,
see, for example, Section~\ref{Basic_examples} and~\ref{Bott-Samelson varieties}.

The main result of the paper is the proof of the following equivariant form of the K\"unneth isomorphism.
Let $X$ and $Y$ be topological spaces and $L,R,P,Q$ be closed subgroups of a compact Lie group $G$
such that $R\subset P$ and $Q\subset P$,
the groups $L$ and $R$ act commutatively and continuously on $X$ on the left and on the right respectively and
$Q$ acts continuously on $Y$. Then we consider the space
$$
X\mathop{\times}\limits_R P\mathop{\times}\limits_QY=X\times P\times Y/\sim,
$$
where $\sim$ is the equivalence relation defined by $(x,p,y)\sim(xr,r^{-1}pq,q^{-1}y)$ for any $r\in R$ and $q\in Q$.
Let also $\alpha:X\to G$ be a continuous map such that $\alpha(lx)=l\alpha(x)$ and $\alpha(xr)=\alpha(x)r$
for any $l\in L$, $x\in X$ and $r\in R$. This map induces the morphism of left $L$-spaces $A:X/R\to G/R$ by
$A(xR)=\alpha(x)R$. Considering $A$ as a twisting map, we get the structure of a right $H_R^\bullet(\pt,\k)$-module
on $H_L^\bullet(X/R,\pt)$ and thus the structure of a right $H_P^\bullet(\pt,\k)$-module
through the natural morphism $H_P^\bullet(\pt,\k)\to H_R^\bullet(\pt,\k)$.
Similarly, the natural morphism $H_P^\bullet(\pt,\k)\to H_Q^\bullet(\pt,\k)$ defines the structure
of a left $H_P^\bullet(\pt,\k)$-module on $H^\bullet_Q(Y,\k)$ by modifying the canonical left module structure.
Then there exists a map
$$
\begin{tikzcd}
H^\bullet_L(X/R,\k)\otimes_{H_P^\bullet(\pt,\k)}H^\bullet_Q(Y,\k)\arrow{r}{\theta}&H_L^\bullet\(X\mathop{\times}\limits_R P\mathop{\times}\limits_Q Y,\k\)
\end{tikzcd}
$$
that is an isomorphism (of bimodules if a twisted right action on $H^\bullet_Q(Y,\k)$ is defined)
under certain restrictions~\ref{restr:-1}--\ref{restr:6}. Note that the importance of this isomorphism
is that it is given by an explicitely constructed map $\theta$, whereas the mere fact
that the cohomologies in both sides are isomorphic follows quite trivially from spectral sequences.
This exact description pays off, when we consider the cohomologies of Bott-Samelson varieties in Section~\ref{Bott-Samelson varieties}
and morphisms between them in Section~\ref{Morphisms}.
%More exactly, using the exact construction of $\theta$,
%we can coordinatize these cohomologies (Theorem~\ref{theorem:3}), prove the concatenation properties
%of these coordinatizations (Theorem~\ref{theorem:4}), and compute the coordinate forms of the one-color morphisms
%in Section~\ref{One-color morphisms} and their horizontal extensions in Section~\ref{Horizontal_extensions}.

We give two basic examples of the applications of the previously developed theory in Section~\ref{Basic_examples}.
The first example concerns the cohomology of the flag variety $G/B$. We reprove
the classical isomorphism
$$
H^\bullet_T(G/B,\k)\cong \R\otimes_{\R^W}\R,
$$
where $B$ is a Borel subgroup of a semisimple complex algebraic group $G$
containing a maximal torus $T$,
$W$ is the Weyl group, $\R=H^\bullet_T(\pt,\k)$ and $\R^W$ is the subrings
of $W$-invariants. Whereas this isomorphism is usually used to introduce the $\R$-$\R$-bimodule 
structure on $T$-equivariant cohomologies of spaces $X$
having $T$-equivariant continuous morphisms $X\to G/B$ (for example, the flag and Bott-Samelson varieties),
we show that it itself is a consequence of the theory developed in
Sections~~\ref{Notation and basic constructions}--\ref{The isomorphism}.
In this paper, we cite this isomorphism only as an example because
our theory of twisted actions allows us to introduce the right actions independently.
The second example concerns the standard bimodules (with twisted right actions)
and we reprove the usual multiplication property of such bimodules.

Our main example concerns Bott-Samelson varieties $\BS(s)$ in Section~\ref{Bott-Samelson varieties}.
%Let $s=(s_1,\ldots,s_n)$ be a sequence of reflections and
Here we use the original definition of these varieties by Bott and Samelson~\cite{Bott_Samelson}. %,
%which includes the case of not necessarily simple reflections appearing in the sequence $s$.
Note that $\BS(s)$ is a left $K$-space for the maximal compact torus $K$.
Thus we can consider the $K$-equivariant cohomology $H_K^\bullet(\BS(s),\k)$ for any commutative ring $\k$.
In Theorem~\ref{theorem:3}, we establish the isomorphism of $\R$-$\R$-bimodules
$$
\theta_s:\R\otimes_{\R^{s_1}}\R\otimes_{\R^{s_2}}\otimes\cdots\otimes_{\R^{s_n}}\R\ito H_K^\bullet(\BS(s),\k),
$$
where $s=(s_1,\ldots,s_n)$ and $\R^{s_i}$ denotes the subring of $s_i$-invariants of the $K$-equivariant
cohomology of the point $\R$. Compared with other constructions of similar isomorphisms (for example,~\cite[Theoren 1.6]{WW}),
our construction has the following advantages:

\smallskip

\begin{itemize}\itemsep=5pt
\item the reflections $s_i$ are not necessarily simple;
\item the proof is localization free; 
\item $\theta_s$ is constructed explicitly as a quotient of the composition
      of a pull-back and a K\"unneth isomorphism;
\item the ring of coefficients $\k$ can be any commutative ring of finite global dimension.
\end{itemize}

The isomorphisms $\theta_s$ satisfy the quite expected concatenation property (Theorem~\ref{theorem:4})
and their restrictions to the $K$-fixed points are computed (Theorem~\ref{theorem:5}).
%As all our constructions are localization-free, it is not important for us if the restriction to 
%the points fixed by $K$
%direct sum of all possible restrictions 
%for a given sequence $s$ 
%is a monomorphism.

Finally, in Section~\ref{Morphisms}, we consider the morphisms of $\R$-$\R$-bimodules
$H_K^\bullet(\BS(s),\k)\to H_K^\bullet(\BS(t),\k)$. It is well known~\cite{Libedinsky}
that all these morphisms are generated (as linear combinations) by morphisms described by planar diagrams~\cite{EW}.
We will assume here that such diagrams do not have horizontal tangent lines and that no two vertices have
the same $y$-coordinate (Soergel graphs in the terminology of~\cite{EW}). We can cut any such diagram
by horizontal lines to strips containing only one vertex. A typical picture looks like this

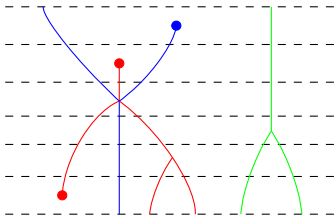
\vspace{10pt}

\begin{center}
\begin{tikzpicture}
\draw[dashed] (-1.5,1.25) -- (2.9,1.25);
\draw[dashed] (-1.5,-1.5) -- (2.9,-1.5);
\draw[dashed] (-1.5,0.75) -- (2.9,0.75);
\draw[dashed] (-1.5,0.25) -- (2.9,0.25);
\draw[dashed] (-1.5,-0.2) -- (2.9,-0.2);
\draw[dashed] (-1.5,-0.575) -- (2.9,-0.575);
\draw[dashed] (-1.5,-1) -- (2.9,-1);
%\draw[red] (0,0) -- (1,-1.5);
\draw[red] (0,0) -- (0,0.5);
%\draw[red] (0.5,-0.75) -- (0.25,-1.5);
%\draw[red] (0.7,-0.75) -- (0.4,-1.5);
\draw [red] plot [smooth,tension=1.2] coordinates {(0.7,-0.75) (0.5,-1.125) (0.4,-1.5)};
\draw[red,fill=red] (0,0.5) circle(0.06);
\draw[blue] (0,0) -- (0,-1.5);
%\draw[red] (0,0) -- (-0.7,-0.6) -- (-0.75,-1.25);
\draw [red] plot [smooth,tension=1.2] coordinates {(0,0) (-0.5,-0.5) (-0.75,-1.25)};
\draw[red,fill=red] (-0.75,-1.25) circle(0.06);
\draw [red] plot [smooth,tension=1.2] coordinates {(0,0) (0.7,-0.75) (1,-1.5)};
\draw [blue] plot [smooth,tension=1.2] coordinates {(0,0) (0.5,0.5) (0.75,1)};
\draw[blue,fill=blue] (0.75,1) circle(0.06);
%\draw [green] (1.6,-1.5) -- (2,-0.4);
\draw [green] plot [smooth,tension=1.2] coordinates {(1.6,-1.5) (1.73,-0.95) (2,-0.4)};
%\draw [green] (2.4,-1.5) -- (2,-0.4);
\draw [green] plot [smooth,tension=1.2] coordinates {(2.4,-1.5) (2.27,-0.95) (2,-0.4)};
\draw [green] (2,-0.4) -- (2,1.25);
\draw [blue] plot [smooth,tension=1.2] coordinates {(0,0) (-0.7,0.75) (-1,1.25)};
\end{tikzpicture}
\end{center}
We find a geometric description for any such horizontal strip
as a pull-back or a push-forward or the composition of both.
The first two cases are considered and computed in coordinates in Sections~\ref{One-color morphisms}
and~\ref{Horizontal_extensions} and the compositions are considered in
Section~\ref{Two-color morphisms} and are extended horizontally in Section~\ref{Horizontal extensions of two-color morphisms}.
One can see that the compositions are necessary only
for strips containing two-color vertices. Such strips can be further cut to the upper and the lower
parts, which are given by pull-backs and push-forwards, respectively. For example, the diagram above
receives an additional cut. The corresponding strip turns into the following diagram:

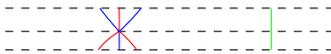
\vspace{7pt}

\begin{center}
\begin{tikzpicture}[y=35pt]
%\draw[dashed] (-1.5,1.25) -- (2.9,1.25);
%\draw[dashed] (-1.5,-1.5) -- (2.9,-1.5);
%\draw[dashed] (-1.5,0.75) -- (2.9,0.75);
\draw[dashed] (-1.5,0.25) -- (2.9,0.25);
\draw[dashed] (-1.5,-0.2) -- (2.9,-0.2);
\draw[dashed] (-1.5,0) -- (2.9,0);
%\draw[dashed] (-1.5,-0.575) -- (2.9,-0.575);
%\draw[dashed] (-1.5,-1) -- (2.9,-1);
%\draw[red] (0,0) -- (1,-1.5);
\draw[red] (0,0) -- (0,0.25);
%\draw[red] (0.5,-0.75) -- (0.25,-1.5);
%\draw[red] (0.7,-0.75) -- (0.4,-1.5);
\draw [red] plot [smooth,tension=1] coordinates {(0,0) (-0.165,-0.1) (-0.275,-0.2)};
\draw [blue] plot [smooth,tension=1] coordinates {(0,0) (0.155,0.125) (0.287,0.25)};
\draw [blue] plot [smooth,tension=1] coordinates {(0,0) (-0.135,0.125) (-0.255,0.25)};
%\draw [red] plot [smooth,tension=1.2] coordinates {(0.7,-0.75) (0.5,-1.125) (0.4,-1.5)};
\draw [red] plot [smooth,tension=1] coordinates {(0,0) (0.13,-0.1) (0.23,-0.2)};
%\draw[red,fill=red] (0,0.5) circle(0.06);
\draw[blue] (0,0) -- (0,-0.2);
%\draw[blue] (0,0) -- (-0.277,-0.2);

%\draw [red] plot [smooth,tension=1.2] coordinates {(0,0) (-0.5,-0.5) (-0.75,-1.25)};
%\draw[red,fill=red] (-0.75,-1.25) circle(0.06);
%\draw [red] plot [smooth,tension=1.2] coordinates {(0,0) (0.7,-0.75) (1,-1.5)};
%\draw [blue] plot [smooth,tension=1.2] coordinates {(0,0) (0.5,0.5) (0.75,1)};
%\draw[blue,fill=blue] (0.75,1) circle(0.06);
%\draw [green] (1.6,-1.5) -- (2,-0.4);
%\draw [green] plot [smooth,tension=1.2] coordinates {(1.6,-1.5) (1.73,-0.95) (2,-0.4)};
%\draw [green] (2.4,-1.5) -- (2,-0.4);
%\draw [green] plot [smooth,tension=1.2] coordinates {(2.4,-1.5) (2.27,-0.95) (2,-0.4)};
\draw [green] (2,-0.2) -- (2,0.25);
%\draw [blue] plot [smooth,tension=1.2] coordinates {(0,0) (-0.7,0.75) (-1,1.25)};
\end{tikzpicture}
\end{center}

\vspace{2pt}

\noindent
This picture shows that we need to consider varieties more general than Bott-Sammelson varieties,
see Section~\ref{Generalized Bott-Samelson varieties}. We also prove that the
normalization criterion for the two-color morphisms holds (Lemma~\ref{lemma:13}), which allows us
to identify them with the morphism $f_{s,t}$ defined by Libedinsky in~\cite[Lemme 4.7]{Libedinsky}.

If we compose these elements differently as in Section~\ref{The Jones-Wenzl projector},
we obtain the Jones-Wenzl projectors, which we also represent by two-color vertices
of the same valency as the vertices representing the morphisms $f_{s,t}$.
Some relations can be proved using the direct construction of the morphisms.
The examples are given in Sections~\ref{Horizontal_extensions}, \ref{The Jones-Wenzl projector},
and~\ref{Two-color dot contraction}. We conjecture that all other relations defining
the diagrammatic category for the given group can also be proved similarly.
It also would be interesting to find a geometric description for morphisms represented
by strips containing more than one vertex, at least in the case when all vertices
of this strip are represented either by pull-backs or push-forwards.

\section{Notation and basic constructions}\label{Notation and basic constructions}
\subsection{Set theory} We denote sets by capital Latin letters and their elements by small Latin letters.
We use the usual signs for the union, the intersection and the Cartesian product\footnote{We also use the following notation for the power $X^n=\underbrace{X\times\cdots\times X}_{n\text{ times}}$.}.
%The Cartesian power is denoted by $X^n=$
%The difference of sets is denoted
%by the minus symbol: $A-B$ is the set of all elements in $A$ but not $B$.
To avoid confusion with quotients, we denote the difference of sets by $-$.
For a sequence of indices $1\le i_1<i_2<\cdots<i_k\le n$ and sets $X_1,\ldots,X_n$,
we denote by $\pr_{i_1,i_2,\ldots,i_k}$ the projection
$X_1\times\cdots\times X_n\to X_{i_1}\times\cdots\times X_{i_k}$ to the corresponding coordinates.
%(however, we deviate slightly from this notation in the proof of Theorem~\ref{theorem:4}).

Let $X$, $Y$, and $S$ be sets and $f:X\to S$ and $g:Y\to S$ be maps. We denote
$$
X\times_{f=g} Y=\{(x,y)\in X\times Y\suchthat f(x)=g(y)\}.
$$
This set is called the {\it fiber product} of $X$ and $Y$ with respect to $f$ and $g$.

We will often use the singleton $\pt$ whose unique element will be denoted by $pt$.

\subsection{Sequences} %: $s=(s_1,\ldots,s_n)$.
The length $n$ of a finite sequence $s$ is denoted by $|s|$ and its $i$th element is denoted by $s_i$.
Thus $s=(s_1,\ldots,s_n)$. The case $|s|=0$ corresponds to the empty sequence, which we denote by $\emptyset$.
If however $|s|>0$, then we get the truncated sequence $s'=(s_1,\ldots,s_{n-1})$.
If $t=(t_1,\ldots,t_m)$ is another sequence, we consider their {\it concatenation}
$$
st=(s_1,\ldots,s_n,t_1,\ldots,t_m).
$$
Obviously, $|st|=|s|+|t|$.

\subsection{Modules} Let $M$ and $M'$ be right modules over the rings $R$ and $R'$, respectively.
Suppose that there is an isomorphism of rings $\iota:R\ito R'$. Then a map $\mu:M\to M'$ is called an
{\it isomorphism of modules} if it is an isomorphism of the underlying abelian groups and
\begin{equation}\label{eq:0}
\mu(mr)=\mu(m)\iota(r)
\end{equation}
for any $m\in M$ and $r\in R$. We assume the similar definitions for left modules and bimodules.

The tensor product of a right $R$-module $M$ by a left $R$-module $N$ is denoted as usual
by $M\otimes_R N$. If $R$ is commutative we can also consider the $n$th {\it tensor power}
$N\otimes_R\cdots\otimes_RN$ with $n$ factors, which we denote by $N^{\otimes_Rn}$.

\subsection{Group actions and quotients}\label{Group_actions_and_quotients} In this paper, we consider left as well as right actions of groups on sets.
Therefore, we will be careful to distinguish between the left and the right quotients. More exactly, let a group $G$ act on a set $X$
on the left. We will denote this fact by $G\circlearrowright X$.
Then we denote by
$$
G\lq X=\{Gx\suchthat x\in X\}
$$
the {\it quotient} space. The map $X\to G\lq X$ that sends each $x$ to $Gx$ is called the {\it quotient} map
(for the action of $G$). We also consider the set of $G$-fixed points
$$
{}^{G\!}X=\{x\in X\suchthat \forall g\in G: gx=x\}
$$

Let $G$ also act on the left on $Y$ and $f:X\to Y$ be a map. We say that $f$ is {\it $G$-equivariant} if $f(gx)=gf(x)$
for any $x\in X$ and $g\in G$. In this case, we denote by $G\lq f$ the map from $G\lq X$ to $G\lq Y$ given by
$$
(G\lq f)(Gx)=Gf(x).
$$

Similarly, if $G$ acts on the right on $X$, we denote this fact by $X\circlearrowleft G$, set
$$
X/G=\{xG\suchthat x\in X\},
$$
use the similar notion of the quotient map and $G$-equivariance, and denote similarly
the set of $G$-fixed points and the maps between the quotients.

Moreover, if groups $L$ and $R$ act on a set $X$ on the left and on the right, respectively, then we say that
the actions of $L$ and $R$ {\it commute} if $l(xr)=(lx)r$ for any $l\in L$, $r\in R$, and $x\in X$.
In this case, $L$ acts on $X/R$ on the left and $R$ acts on $L\lq X$ on the right by $l(xR)=(lx)R$ and $(Lx)r=L(xr)$,
respectively.
If additionally $Y$ is another space endowed with commuting left and right actions of $L$ and $R$, respectively,
then we say that a map $f:X\to Y$ is {\it $R$-$L$-equivariant} if it is $L$-equivariant and $R$-equivariant simultaneously.

Suppose now that $G$ acts on the left both on $X$ and $E$. Then $G$ acts on the left on the product $X\times E$ diagonally:
$g(x,e)=(gx,ge)$. Then we set
$$
X\,{}_G\!\times\,E=G\lq(X\times E).
$$
%Similarly, if $G$ acts on the right on both $X$ and $Y$, we set
%$$
%X\times_G Y=(X\times Y)/G.
%$$
We call the map $X\,{}_G\!\times\,E\to G\lq E$ that maps any orbit $G(x,e)$ to $Ge$ the {\it canonical projection}.
We denote canonical projections by $\can$.
If $X'$ and $E'$ are another sets endowed with left $G$-actions and
$\alpha:X\to X'$ and $\beta:E\to E'$
are $G$-equivariant maps, then we denote by $\alpha{}\,{}_G\!\times \beta$ the map from
$X\,{}_G\!\times\,E$ to $X'\,{}_G\!\times\,E'$ given by $G(x,e)\mapsto G(\alpha(x),\beta(e))$.

\subsection{Quotient products}\label{Quotient_products}
Let $X_1,\ldots,X_n$ be topological spaces and $G_1,\ldots,G_{n-1}$ be
topological groups such that each $G_i$ acts on the right on $X_i$ and on the left on $X_{i+1}$.
Suppose additionally that the actions of $G_{i-1}$ and $G_i$ on $X_i$ commute for each $i=2,\ldots,n-1$.
Then we consider the following equivalence relation $\sim$ on $X_1\times\cdots \times X_n$:
$$
(x_1,\ldots,x_n)\sim(x_1g_1,g_1^{-1}x_2g_2,\ldots,g_{n-1}x_n)
$$
for any $g_1\in G_1$, \ldots, $g_{n-1}\in G_{n-1}$.
We denote
$$
X_1\mathop{\times}\limits_{\;\,G_1}X_2\mathop{\times}\limits_{\;\,G_2}\cdots\!\!\!\!\!\mathop{\times}\limits_{\;\;\;\;\;\;G_{n-1}}X_n=X_1\times X_2\times\cdots\times X_n/\sim.
$$
If $n=0$ this space is just the singleton and if $n=1$ this space is $X_1$.
The equivalence class containing $(x_1,x_2,\dots,x_n)$ is denoted by $[x_1:x_2:\cdots:x_n]$.
If $n>0$ and $G_0$ is a group acting on the left on $X_1$ such that in the case $n>1$
this action and the action of $G_1$ on $X_1$ commute, then we get the following left action of $G_1$ on
$X_1\mathop{\times}\limits_{\;\,G_1}X_2\mathop{\times}\limits_{\;\,G_2}\cdots\!\!\!\!\!\mathop{\times}\limits_{\;\;\;\;\;\;G_{n-1}}X_n$:
\begin{equation}\label{eq:left}
g_0[x_1:x_2:\cdots:x_n]=[g_0x_1:x_2:\cdots:x_n].
\end{equation}
Similarly, for $n>0$ and a group $G_n$ acting on $X_n$ on the right
such that in the case $n>1$ this action and the action of $G_{n-1}$ on $X_n$ commute, %and satisfying the similar commutativity condition,
we get the following right action:
\begin{equation}\label{eq:right}
[x_1:x_2:\cdots:x_n]g_n=[x_1:x_2:\cdots:x_ng_n].
\end{equation}

\noindent
This  construction is associative, that is,
$$
X_1\!\mathop{\times}\limits_{\;\,G_1}\!\!\cdots\!\!\!\!\!\!\!\mathop{\times}\limits_{\;\;\;\;\;\;G_{n-1}}X_n\cong
X_1\!\!\mathop{\times}\limits_{\;\,G_1}\!\cdots\!\!\!\!\!\!\!\mathop{\times}\limits_{\;\;\;\;\;\;\;G_{m-2}} X_{m-1}\!\!\!\!\!\mathop{\times}\limits_{\;\;\;\;\;\;\;G_{m-1}}(X_m\!\mathop{\times}\limits_{\;\,G_m}\cdots\!\!\!\!\!\!\mathop{\times}\limits_{\;\;\;\;\;\;G_{k-1}}X_k)\mathop{\times}\limits_{\;\,G_k}X_{k+1}\!\!\!\!\!\mathop{\times}\limits_{\;\;\;\;\;G_{k+1}}\!\!\cdots\!\!\!\!\!\!\mathop{\times}\limits_{\;\;\;\;\;\;G_{n-1}}X_n.
$$
for $1\le m\le k\le n$. We will make this identification throughout the paper. Hence
$[x_1:\cdots:x_n]=[x_1:\cdots:x_{m-1}:[x_m:\cdots:x_k]:x_{k+1}:\cdots:x_n]$.

%Note also that the cases $n=0$ and $n=1$ are also possible

\subsection{More quotient products}\label{More quotient products} Keeping the notation and assumptions of the previous section, consider
the quotient
$$
X_1\mathop{\times}\limits_{\;\,G_1}X_2\mathop{\times}\limits_{\;\,G_2}\cdots\!\!\!\!\!\mathop{\times}\limits_{\;\;\;\;\;\;G_{n-1}}X_n/G_n=X_1\times X_2\times\cdots\times X_n/\approx,
$$
where the equivalence relation $\approx$ is given by $(x_1,\ldots,x_n)\approx(x_1g_1,g_1^{-1}x_2g_2,\ldots,g_{n-1}x_ng_n)$
for any $g_1\in G_1$, \ldots, $g_{n-1}\in G_{n-1}$, $g_n\in G_n$.
The equivalence class containing $(x_1,x_2,\dots,x_n)$ is denoted by $[x_1:x_2:\cdots:x_n\rr$.
In view of the isomorphism
$$
X_1\mathop{\times}\limits_{\;\,G_1}X_2\mathop{\times}\limits_{\;\,G_2}\cdots\!\!\!\!\!\mathop{\times}\limits_{\;\;\;\;\;\;G_{n-1}}X_n/G_n\cong X_1\mathop{\times}\limits_{\;\,G_1}X_2\mathop{\times}\limits_{\;\,G_2}\cdots\!\!\!\!\!\mathop{\times}\limits_{\;\;\;\;\;\;G_{n-1}}X_n\mathop{\times}\limits_{\;\;\;G_n}\pt,
$$
we will assume $[x_1:x_2:\cdots:x_n\rr=[x_1:x_2:\cdots:x_n:pt]$, where $G_n$ acts on $\pt$ on the left trivially.
This identification, allows us to introduce the left action of $G_0$ on the above space
for any $n$ (including the case $n=0$). Note that
$$
X_1\mathop{\times}\limits_{\;\,G_1}X_2\mathop{\times}\limits_{\;\,G_2}\cdots\!\!\!\!\!\mathop{\times}\limits_{\;\;\;\;\;\;G_{n-1}}X_n/G_n\cong
(X_1\mathop{\times}\limits_{\;\,G_1}X_2\mathop{\times}\limits_{\;\,G_2}\cdots\!\!\!\!\!\mathop{\times}\limits_{\;\;\;\;\;\;G_{n-1}}X_n)/G_n
$$
and
$$
X_1\mathop{\times}\limits_{\;\,G_1}X_2\mathop{\times}\limits_{\;\,G_2}\cdots\!\!\!\!\!\mathop{\times}\limits_{\;\;\;\;\;\;G_{n-1}}X_n/G_n\cong
X_1\mathop{\times}\limits_{\;\,G_1}X_2\mathop{\times}\limits_{\;\,G_2}\cdots\!\!\!\!\!\mathop{\times}\limits_{\;\;\;\;\;\;G_{n-1}}(X_n/G_n)
$$
if $n>0$.

Finally note that
$$
X_1\mathop{\times}\limits_{\;\,G_1}\cdots\!\!\!\!\!\mathop{\times}\limits_{\;\;\;\;\;\;G_{n-1}}X_n/G_n\cong X_1\mathop{\times}\limits_{\;\,G_1}\cdots\!\!\!\!\!\mathop{\times}\limits_{\;\;\;\;\;\;G_{i-1}}X_i\mathop{\times}\limits_{\;\,G_i}G_i\mathop{\times}\limits_{\;\,G_i}X_{i+1}\!\!\!\!\mathop{\times}\limits_{\;\;\;\;G_{i+1}}\cdots\!\!\!\!\!\mathop{\times}\limits_{\;\;\;\;\;\;G_{n-1}}X_n/G_n
$$
where we assume that $[x_1:\cdots:x_n\rr=[x_1:\cdots:x_i:1:x_{i+1}:\cdots:x_n\rr$.
We will make this identification, when appropriate.

\subsection{Topology}\label{Topology} A {\it topological group} is a group $G$ that is also a topological space
such that the multiplication and taking the inverse are continuous maps.
We say that $G$ acts continuously on a topological space $X$ if $G$ acts on $X$ set-theoretically so that
the multiplication $G\times X\to X$ (resp. $X\times G\to X$) is continuous.
In this case, we also say that $X$ is a left (resp. right) $G$-{\it space}.

We are going to use in the sequel the following topological result.

\begin{proposition}\label{proposition:1}
Let $X$ and $Y$ be topological spaces and $f:X\to Y$ be a continuous bijection.
Suppose that there exists an open covering $Y=\bigcup_{i\in I}U_i$ and continuous functions
$g_i:U_i\to X$ for each $i\in I$ such that $fg_i(y)=y$ for any $i\in I$ and $y\in U_i$.
Then $f$ is a homeomorphism.
\end{proposition}

Let $E$ and $B$ be $G$-spaces and $p:E\to B$ be a continuous $G$-equivariant map.
We say that $p$ is a {\it principal $G$-bundle} if for each point $b\in B$, there exist an open neighbourhood $U\subset B$
and a $G$-equivariant homeomorphism $h:G\times U\ito p^{-1}(U)$ for the left $G$-action or
$h:U\times G\ito p^{-1}(U)$ for the right $G$-action such that the respective diagram
$$
\begin{tikzcd}
G\times U\arrow{rr}{h}[swap]{\sim}\arrow{dr}[swap]{\pr_2}&&p^{-1}(U)\arrow{dl}{p}\\
&U&
\end{tikzcd}
\qquad\quad
\begin{tikzcd}
U\times G\arrow{rr}{h}[swap]{\sim}\arrow{dr}[swap]{\pr_1}&&p^{-1}(U)\arrow{dl}{p}\\
&U&
\end{tikzcd} $$
is commutative. % for the left and the right action, respectively.
The $G$-equivariance of $h$
%is with respect to the following action of $G$ on $U\times G$:
%$(u,g_1)g_2$
actually means that for any $u\in U$ and $g_1,g_2\in G$, there holds
$$
h(g_1g_2,u)=g_1h(u,g_2)
$$
for the left $G$-action and
$$
h(u,g_1g_2)=h(u,g_1)g_2
$$
for the right $G$-action.
By abuse of notation, we have written $p$ instead of restriction of $p$ to $p^{-1}(U)$ in the diagrams above.
We will apply the similar abbreviation in the other diagrams in this paper.

We have the following well-known result, see, for example,~\cite[Section 1.4]{Jantzen}.
\begin{proposition}\label{proposition:pb:1}
Let $X$ and $E$ be left $G$-spaces. Suppose that the quotient map $E\to G\lq E$ is a principal $G$-bundle.
Then the canonical projection $X\times_G E\to G\lq E$ is a fibre bundle with fibre $X$.
\end{proposition}

We also can make new principal bundles from the already existing ones.
\begin{lemma}\label{lemma:pb:1}
Let $X$ and $Y$ be right and left $G$-spaces, respectively.
Let, moreover, $H$ act continuously on $Y$ on the right so that this action and the left action of $G$ commute.
Suppose that the quotient maps $X\to X/G$ and $Y\to Y/H$ are principal $G$- and $H$-bundles, respectively.
Then the quotient map
\begin{equation}\label{eq:pb:1}
X\mathop{\times}\limits_G Y\to X\mathop{\times}\limits_G Y/H
\end{equation}
for the right action~(\ref{eq:right}) is a principal $H$-bundle.
\end{lemma}
\begin{proof}
%We will use the identification $X\mathop{\times}\limits_G Y/H\cong X\mathop{\times}\limits_G(Y/H)$
%as in Section~\ref{More quotient products}. It is given by $[x:y]H\sim[x:yH]$.
%We will assume this identification in this proof.
We denote map~(\ref{eq:pb:1}) by $f$. Thus $f([x:y])=[x:y\rr$. We also denote the quotient maps $X\to X/G$
and $Y\to Y/H$ by $p$ and $q$, respectively.

%
%In the formulation of this lemma and in its prove, we identify both spaces $(X\mathop{\times}\limits_G Y)/H$
%and $X\mathop{\times}\limits_G(Y/H)$ by $[x:y]H\sim[x:yH]$ and denote them both by $X\mathop{\times}\limits_G Y/H$.
%We will denote the quotient maps $X\to X/G$ and $Y\to Y/H$ by $p$ and $q$ respectively and
%denote the quotient map~(\ref{eq:pb:1}) by $f$.

Let $[x_0:y_0\rr$ be an arbitrary point of $X\mathop{\times}\limits_G Y/H$.
By the hypothesis of the lemma, there exist an open neighborhood $U$ of $x_0G$
and a $G$-equivariant homeomorphism $\sigma$ such that the diagram
$$
\begin{tikzcd}
U\times G\arrow{dr}[swap]{\pr_1}\arrow{rr}{\sigma}[swap]{\sim}&&p^{-1}(U)\arrow{dl}{p}\\
&U&
\end{tikzcd}
$$
is commutative.
%and $\sigma(u,gg')=\sigma(u,g)g'$ for any $u\in U$ and $g,g'\in G$.
Let $\g=\pr_2\sigma^{-1}$. This is a continuous map from $p^{-1}(U)$ to $G$ such that
$\g(xg)=\g(x)g$ for any $x\in p^{-1}(U)$ and $g\in G$.

There exists an open neighborhood $V$ of $\g(x_0)y_0H$
and a $H$-equivariant homeomorphism $\tau$ such that the diagram
$$
\begin{tikzcd}
V\times H\arrow{dr}[swap]{\pr_1}\arrow{rr}{\tau}[swap]{\sim}&&q^{-1}(V)\arrow{dl}{q}\\
&V&
\end{tikzcd}
$$
is commutative.
Let $\h=\pr_2\tau^{-1}$.
This is a continuous map from $q^{-1}(V)$ to $H$ such that
$\h(yh)=\h(y)h$ for any $y\in q^{-1}(V)$ and $h\in H$.

Let $\lm:X\times Y/H\to X/G$ be the map given by $\lm([x:y\rr)=xG$ and $\mu:\lm^{-1}(U)\to Y/H$
be the map given by $\mu([x:y\rr)=\g(x)yH$. Obviously, both these maps are well-defined and continuous.
We set $W=\mu^{-1}(V)$. By our choice of $U$ and $V$, we get $[x_0:y_0\rr\in W$.
We define the function $\xi:W\times H\to X\mathop{\times}\limits_G Y$
by
$$
\([x:y\rr,h\)\mapsto [x:y\h(\g(x)y)^{-1}h].
$$
One can easily check that this map is well-defined, continuous, and $H$-equivarinat. Composing it with $f$, we get
%$$
%\([x:y\rr,h\)\stackrel{\xi}{\mapsto}[x:y\h(\g(x)y)^{-1}h]\stackrel{f}{\mapsto}[x:y\h(\g(x)y)^{-1}hH\rr=[x:y\rr.
%$$
$$
\begin{tikzcd}[column sep=15pt]
\([x:y\rr,h\)\arrow[mapsto]{r}{\xi}&{[x:y\h(\g(x)y)^{-1}h]}\arrow[mapsto]{r}{f}&{[x:y\h(\g(x)y)^{-1}hH\rr=[x:y\rr}.
\end{tikzcd}
$$

This calculation proves that $\xi$ is actually a map to $f^{-1}(W)$ and that the following diagram is commutative
$$
\begin{tikzcd}
W\times H\arrow{rr}{\xi}\arrow{dr}[swap]{\pr_1}&&f^{-1}(W)\arrow{dl}{f}\\
&W&
\end{tikzcd}
$$
It remains to prove that $\xi$ is a homeomorphism.
This is true, as the inverse map is given by
$$
\xi^{-1}([x:y])=\big([x:y\rr,\h(\g(x)y)\big).
$$
\end{proof}

\subsection{Equivariant cohomology}\label{Equivariant cohomology} Let $G$ be a topological group. A principal $G$-bundle $E\to B$
is called {\it universal} if the total space $E$ is contractible.
%ссылка на Милнора
We will tacitly assume that such a bundle exists for any topological group considered here.
We define the $G$-equivariant cohomology of a left $G$-space $X$ with coefficients in a commutative rigns $\k$ by
$$
H^\bullet_G(X,\k)=H^\bullet(X\,{}_G\!\times\,E,\k).
$$
%Remember that as usual all cohomologies are cohomologies with coefficients in a fixed commutative ring $\k$.

Considering the canonical projection $\can:X\,{}_G\!\times\,E\to G\lq E\cong\pt\,{}_G\!\times\,E$, we obtain the map
$\can^*:H^\bullet_G(\pt,\k)\to H^\bullet_G(X,\k)$. This map makes $H^\bullet_G(X,\k)$ into a
left $H^\bullet_G(\pt,\k)$-module by
$$
am=\can^*(a)\cup m,
$$
where $a\in H^\bullet_G(\pt,\k)$, $m\in H^\bullet_G(X,\k)$ and $\cup$ denotes the cup product. This module does not depend on the choice
of the universal bundle $E\to B$. This fact will be discussed in Corollary~\ref{corollary:1}.
%, see a similar argument in Lemma~\ref{lemma:2}.
We will call this left action of $H^\bullet_G(\pt,\k)$ {\it canonical}.

The similar right action of $H^\bullet_G(\pt,\k)$ on $H^\bullet_G(X,\k)$ given by
$ma=m\cup \can^*(a)$ will also be called {\it canonical}.
We will use it in Section~\ref{Equivariant_cohomology_of_the_flag_variety}.

If $X'$ is another left $G$-space and $f:X\to X'$ is a continuous $G$-equivariant map,
then we have the continuous map $f\,{}_G\!\times\,\id:X\,{}_G\!\times\,E\to X'\,{}_G\!\times\,E$.
This map induces the pull-back
$$
(f\,{}_G\!\times\,\id)^*:H^\bullet_G(X',\k)\to H^\bullet_G(X,\k)
$$
between cohomologies. We use the notation $f^\star=(f\,{}_G\!\times\,\id)^*$ to distinguish between the ordinary
and the equivariant pull-backs. For example $a_X^\star=\can^*$, where $a_X:X\to\pt$ is the constant map.

\subsection{Twisted action}\label{Twisted action} Let $G$ be a topological group and $L$ and $R$ be its subgroups.
We assume that $L$ and $R$ act on $G$ on the left and on the right, respectively, by multiplication.
As both actions commute, the group $L$ acts continuously on the quotient space $G/R$ on the left.
%(that is, $G/R$ is a left $L$-space).

Let $X$ be another left $L$-space and $A:X\to G/R$ be an $L$-equivariant continuous map.
We will call $A$ a {\it twisting map}.
For any left $G$-space $E$, the map $A$ allows us to define the map
$$
v_A:X\,{}_L\!\!\times\,E\to R\lq E
$$
by
\begin{equation}\label{eq:1}
%\begin{tikzcd}
%L(x,e)\arrow[mapsto]{r}{v}& A(x)^{-1}e.
%\end{tikzcd}
L(x,e)\mapsto A(x)^{-1}e.
\end{equation}
Note that $A(x)^{-1}\in R\lq G$.
The reader can easily check that $v_A$ is
well-defined and continuous. We call this map the {\it projection twisted by} $A$.

Let us suppose that the quotient maps $E\to L\lq E$ and $E\to R\lq E$ are universal principal $L$- and $R$-bundles,
respectively. Then we get the map between cohomologies $v_A^*:H^\bullet_R(\pt,\k)\to H^\bullet_L(X,\k)$, which induces the
following right $H^\bullet_R(\pt,\k)$-action on $H^\bullet_L(X,\k)$:
\begin{equation}\label{eq:i}
mb=m\cup v_A^*(b).
\end{equation}
%for $m\in H^\bullet_L(X,\k)$ and $b\in H^\bullet_R(\pt,\k)$. % and $\cup$ denotes the cup product.
We call this action the {\it action twisted by} $A$.

A priori, this construction of a right module depends on the choice of $E$.
%the right $H^\bullet_R(\pt,\k)$-module $H^\bullet_L(X,\k)$ depends on the choice of the $G$-space $E$.
However we get the following result.

\begin{lemma}\label{lemma:2}
Up to isomorphism, the right $H^\bullet_R(\pt,\k)$-module $H^\bullet_L(X,\k)$ is independent of the choice of $E$.
\end{lemma}
\begin{proof}
%All such modules (together with their rings) will be isomorphic, as is clear from the following construction.
Let $\widetilde E$ be another left $G$-space such that $\widetilde E\to L\lq\widetilde E$ and
$\tilde E\to R\lq\tilde E$ are universal principal
$L$- and $R$-bundles, respectively. We use the following notation only in this proof:
$$
\widetilde H_L^\bullet(X,\k)=H^\bullet(X{}_L\!\times\widetilde E,\k), \quad \widetilde H_R^\bullet(\pt,\k)=H^\bullet(R\lq\widetilde E,\k).
$$
We also have the map $\tilde v_A:X{}_L\!\times\widetilde E\to R\lq\widetilde E$ similar to $v_A$ also given by~(\ref{eq:1}).
This map defines the structure of a right $\widetilde H^\bullet_R(\pt,\k)$-module on $\widetilde H^\bullet_L(X,\k)$
by
\begin{equation}\label{eq:ii}
\tilde m\tilde b=\tilde m\cup \tilde v_A^*(\tilde b).
\end{equation}

Let us make $L$ and $R$ act on the products $X\times E\times\widetilde E$ and $E\times\widetilde E$ respectively diagonally: 
$l(x,e,\tilde e)=(lx,le,l\tilde e)$ and $r(e,\tilde e)=(re,r\tilde e)$. Then we get the following commutative diagram:
\begin{equation}\label{eq:a}
\begin{tikzcd}
{}&L\lq(X\times E\times\widetilde E)\arrow{ddd}{q_A}\arrow{rd}{L\lq\pr_{1,3}}\arrow{ld}[swap]{L\lq\pr_{1,2}}&\\
X{}_L\!\times E\arrow{d}[swap]{v_A}&&X{}_L\!\times\widetilde E\arrow{d}{\tilde v_A}\\
R\lq E&&R\lq\widetilde E\\
&E{}_R\!\times\widetilde E\arrow{ru}[swap]{R\lq\pr_2}\arrow{lu}{R\lq\pr_1}&
\end{tikzcd}
\end{equation}
where $q_A$ is the map given by
%$
%\begin{tikzcd}
%L(x,e,\tilde e)\arrow[mapsto]{r}{q}& A(x)^{-1}(e,\tilde e).
%\end{tikzcd}
%$
$L(x,e,\tilde e)\mapsto A(x)^{-1}(e,\tilde e)$.
Hence we get the following diagram for cohomologies:
\begin{equation}\label{eq:b}
\begin{tikzcd}
{}&H^\bullet\big(L\lq(X\times E\times\widetilde E),\k\big)&\\
H^\bullet_L(X,\k)\arrow{ru}{\sim}\arrow[dashed]{rr}[near start]{\sim}[near end]{\mu}&&\widetilde H^\bullet_L(X,\k)\arrow[swap]{lu}{\sim}\\
H^\bullet_R(\pt,\k)\arrow[swap]{rd}{\sim}\arrow{u}{v_A^*}\arrow[dashed]{rr}[near start]{\sim}[near end]{\iota}&&\widetilde H^\bullet_R(\pt,\k)\arrow{ld}{\sim}\arrow[swap]{u}{\tilde v_A^*}\\
&H^\bullet(E{}_R\!\times\widetilde E,\k)\arrow{uuu}{q_A^*}&
\end{tikzcd}
\end{equation}
If we forget about the dashed arrows, then we get a commutative diagram.
Thereofre the central rectangle (with the dashed arrows) is also commutative, because all slant arrows are isomorphisms.
%This fact ammonts to the above claim about the isomorphism of modules.
As all maps preserve the additive structure on cohomologies, it remains to check condition~(\ref{eq:0}):
for $m\in H_L^\bullet(X,\k)$ and $b\in H_R^\bullet(\pt,\k)$, we get
%$$
%\psi(\mu\sigma)=\psi(\mu\cup v_A^*(\sigma))=\psi(\mu)\cup \psi v_A^*(\sigma)=\psi(\mu)\cup \tilde v_A^*\phi(\sigma)=\psi(\mu)\phi(\sigma),
%$$
$$
\mu(mb)=\mu(m\cup v_A^*(b))=\mu(m)\cup \mu v_A^*(b)=\mu(m)\cup \tilde v_A^*\iota(b)=\mu(m)\iota(b),
$$
where we use the multiplication rules~(\ref{eq:i}) and~(\ref{eq:ii}) and the fact that $\mu$ preserves the cup product.
\end{proof}

\begin{corollary}[of Lemma~\ref{lemma:2}]\label{corollary:1}
Up to isomorphism, the $H_L^\bullet(\pt,\k)$-$H^\bullet_R(\pt,\k)$-bimodule $H^\bullet_L(X,\k)$ is independent of the choice of $E$.
\end{corollary}
\begin{proof} It suffices to prove the result similar to Lemma~\ref{lemma:2} for the canonical action.
We can do it if in~(\ref{eq:a}) we replace: $R$ with $L$; $v_A$ and $\tilde v_A$ with the canonical projections
$X\,{}_L\!\!\times\,E\to L\lq E$ and $X\,{}_L\!\!\times\,\tilde E\to L\lq \tilde E$, respectively;
$q_A$ with $L\lq\pr_{2,3}$.
\end{proof}

\begin{remark}\label{remark:1}\rm The above stipulation that both $E\to L\lq E$ and $E\to R\lq E$ are universal principal $L$- and $R$-bundles
usually is the result of the fact that $G$ is a Lie group, $L$ and $R$ are its closed subgroups and the quotient map
$E\to G\lq E$ is a universal principal $G$-bundle. We consider this case in Section~\ref{The_case_of_a_compact_Lie_group}.
\end{remark}

\begin{remark}\label{remark:2}\rm If $R=L$ and $A:X\to G/L$ is defined by $A(x)=L$ for any $x\in X$, then $v_A$
is the canonical projection. In this case, the corresponding right action~(\ref{eq:i})
of $H^\bullet_L(\pt,\k)$ on $H_L^\bullet(X,\k)$ is canonical.
\end{remark}

\subsection{Pull-back as a bimodule homomorphism} Keeping the notation of the previous section, let additionally
$Y$ be another $G$-space and $f:Y\to X$ be an $L$-equivariant continuous map.
Then the composition $Af:Y\to G/R$ is also $L$-equivariant.

\begin{lemma}\label{lemma:pull-back}
The equivariant pull-back $f^\star:H^\bullet_L(X,\k)\to H^\bullet_L(Y,\k)$ is a homomorphism of bimodules,
where the left actions are canonical and the right actions on the domain and codomain
are twisted by $A$ and $Af$, respectively.
\end{lemma}
\begin{proof}
As for the left actions, the result is well-known, we will prove only
the claim about the right actions. It follows directly from definition that
\begin{equation}\label{eq:d}
v_{Af}=v_A(f\,{}_L\!\!\times\,\id).
\end{equation}
Hence for any $m\in H^\bullet_L(X,\k)$ and $b\in H_R^\bullet(\pt,\k)$, we get
%$$
%f^\star(m)b=f^\star(m)\cup v_{Af}^*(b)=f^\star(m)\cup f^\star v_A^*(b)=f^\star(m\cup v_A^*(b))=f^\star(mb).
%$$
$$
f^\star(mb)=f^\star(m\cup v_A^*(b))=f^\star(m)\cup f^\star v_A^*(b)=f^\star(m)\cup v_{Af}^*(b)=f^\star(m)b.
$$
\end{proof}

\section{The tensor product morphism}\label{The_tensor_product_morphism}

In this section, we fix the main objects of this paper: the topological spaces $X$ and $Y$, the groups $G$,$L$,$R$,$P$,$Q$, and the map $\alpha:X\to G$.
\subsection{The embedding of Borel constructions}\label{The_embedding} Let again $L,R,P,Q$ be subgroups of a topological group $G$
such that $R\subset P$ and $Q\subset P$.
Let $X$ and $Y$ be topological spaces such that
{\renewcommand{\labelenumi}{{\rm\theenumi}}
\renewcommand{\theenumi}{{\rm(\roman{enumi})}}
\begin{enumerate}
\itemsep=3pt
\item\label{p:1} $L$ acts on the left on $X$;
\item\label{p:2} $R$ acts on the right on $X$;
\item\label{p:3} $Q$ acts on the left on $Y$.
\end{enumerate}}
\noindent
Assuming that $R$ and $Q$ act on $P$ on the left and on the right respectively
via multiplication, we can write these data as follows:
$$
L\circlearrowright X\circlearrowleft R\circlearrowright P\circlearrowleft Q\circlearrowright Y.
$$
We also assume that %the associativity law
{\renewcommand{\labelenumi}{{\rm\theenumi}}
\renewcommand{\theenumi}{{\rm(\roman{enumi})}}
\begin{enumerate}\setcounter{enumi}{3}
\item\label{p:4} the actions of $L$ and $R$ on $X$ commute. %$(lx)r=l(xr)$ for any $l\in L$, $x\in X$ and $r\in R$.
\end{enumerate}}
\noindent
Let additionally
{\renewcommand{\labelenumi}{{\rm\theenumi}}
\renewcommand{\theenumi}{{\rm(\roman{enumi})}}
\begin{enumerate}\setcounter{enumi}{4}
\item\label{p:5} $\alpha:X\to G$ be a continuous $R$-$L$-eqivariant map.
% map such that $\alpha(lx)=l\alpha(x)$ and $\alpha(xr)=\alpha(x)r$
%for any $l\in L$, $x\in X$ and $r\in R$.
\end{enumerate}}
\noindent
In that case, we get the morphism of left $L$-spaces $A:X/R\to G/R$ given by
\begin{equation}\label{eq:tw5:6}
A(xR)=\alpha(x)R.
\end{equation}
Thus we are in the situation of Section~\ref{Twisted action} (with $X$ replaced by $X/R$)
and we have the map $v_A:(X/R)\,{}_L\!\!\times\,E\to R\lq E$ given by~(\ref{eq:1}).

Let $E$ be a left $G$-space. In view of~\ref{p:4}, we can define the left action of $L$ on
$X\mathop{\times}\limits_R P\mathop{\times}\limits_QY$ by~(\ref{eq:left}), that is,
$
l[x:p:y]=[lx:p:y]
$.
We consider the map
%\begin{equation}\label{eq:2}
$$
\phi:\(X\mathop{\times}\limits_R P\mathop{\times}\limits_QY\)\rtimes{L}E\to((X/R)\rtimes{L}E)\times(Y\rtimes{Q}E)
$$
%\end{equation}
defined by
\begin{equation}\label{eq:5}
%\begin{tikzcd}
%L([x:R:y],e)\arrow[mapsto]{r}{\phi}&\big(L(xR,e),R(y,\alpha(x)^{-1}e)\big).
%\end{tikzcd}
L([x:p:y],e)\mapsto\big(L(xR,e),Q(y,p^{-1}\alpha(x)^{-1}e)\big).
\end{equation}
The reader can easily check that this map is well-defined and continuous.
%Besides this map we have the following three maps:
%\begin{itemize}\itemsep=3pt
%\item The map $v_A:(X/R)\,{}_L\!\!\times\,E\to R\lq E$ given by~(\ref{eq:1}).
%%$\begin{tikzcd}[cramped, sep=small]
%%L(xR,e)\arrow[mapsto]{r}{v}& R\alpha(x)^{-1}e
%%\end{tikzcd}$. This is with $x$ replaced by $xR$.
%\item The canonical projection $u:(X/R)\,{}_L\!\!\times\,E\to L\lq E$.
%\item The canonical projection $w:Y\,{}_Q\!\!\times\,E\to Q\lq E$.
%\end{itemize}

We get the following diagram:
%\begin{equation}\label{eq:c}
%\begin{tikzcd}
%\(X\mathop{\times}\limits_R P\mathop{\times}\limits_Q Y\)\;{}_L\!\times E\arrow{r}{\phi}&((X/R){}_L\times\!E)\times(Y\,{}_Q\!\!\times\,E)\arrow{r}{\pr_1}\arrow{d}[swap]{\pr_2}&(X/R)\,{}_L\!\!\times\,E\arrow{d}{v_A}\\
%&Y{}_Q\times\!E\arrow{d}[swap]{\can}&R\lq E\arrow{d}{\pi_R}\\
%&Q\lq E\arrow{r}{\pi_Q}&P\lq E
%\end{tikzcd}
%\end{equation}
%
\begin{equation}\label{eq:c}
\begin{tikzcd}
\(X\mathop{\times}\limits_R P\mathop{\times}\limits_Q Y\)\rtimes{L} E\arrow{r}{\phi}&((X/R)\rtimes{L}E)\times(Y\rtimes{Q}E)\arrow{r}{\pr_1}\arrow{d}[swap]{\pr_2}&(X/R)\rtimes{L}E\arrow{d}{v_A}\\
&Y\rtimes{Q}E\arrow{d}[swap]{\can}&R\lq E\arrow{d}{\pi_R}\\
&Q\lq E\arrow{r}{\pi_Q}&P\lq E
\end{tikzcd}
\end{equation}
where $\pi_R$ and $\pi_Q$ are the natural quotient maps. In general, the rectangle on the right is not commutative.
But we get the following result.

\begin{lemma}\label{lemma:c}
The two maximal paths of diagram~(\ref{eq:c}) commute, that is,
$$\pi_R\,v_A\pr_1\phi=\pi_Q \can \pr_2\phi.$$
\end{lemma}
\begin{proof}
The result follows directly from~(\ref{eq:5}).
\end{proof}

Now we are going to study the image of $\phi$.

\begin{lemma}\label{lemma:1}
Suppose that $P$ acts freely on $E$. Then $\phi$ is an embedding
whose image coincides with %the fibre product
$
((X/R)\,{}_L\!\!\times\,E)\times_{\pi_R\,v_A\pr_1=\pi_Q\can\pr_2}(Y\,{}_Q\!\!\times\,E)
$.
\end{lemma}

\begin{proof}
Let $L([x:p:y],e)$ and $L([x':p':y'],e')$ be two orbits mapped to the same pair by $\phi$.
As $L(xR,e)=L(x'R,e')$, there exists $l\in L$ such that $lxR=x'R$ and $le=e'$.
From the first equality, it follows that there exists $r\in R$
such that $lxr=x'$. We get
\begin{equation}\label{eq:tw2:1}
\begin{array}{c}
\hspace{-140pt}L([x:p:y],e)=L([lx:p:y],le)=L([lxr:r^{-1}p:y],e')\\[3pt]
\hspace{160pt}
\begin{tikzcd}
=L([x':r^{-1}p:y],e')\arrow[mapsto]{r}{\pr_2\phi}&Q(y,p^{-1}r\alpha(x')^{-1}e').
\end{tikzcd}
\end{array}
\end{equation}
As $L([x':p':y'],e')$ is mapped by $\pr_2\phi$ to the same orbit, we get
$$
Q(y,p^{-1}r\alpha(x')^{-1}e')=Q(y',(p')^{-1}\alpha(x')^{-1}e').
$$
Therefore, there exists $q\in Q$ such that $qy=y'$
and $qp^{-1}r\alpha(x')^{-1}e'=(p')^{-1}\alpha(x')^{-1}e'$. From the last equality and
the freeness of the action of $P$ on $E$, we get $qp^{-1}r=(p')^{-1}$.
Hence $r^{-1}p=p'q$ and, applying~(\ref{eq:tw2:1}), we get
\begin{multline*}
L([x:p:y],e)=L([x':r^{-1}p:y],e')=L([x':p'q:y],e')\\
=L([x':p':qy],e')=L([x':p':y'],e')
\end{multline*}
as required.

Let us prove now the claim about the image. It follows from Lemma~\ref{lemma:c}
that $\im\phi$ is contained in the fibre product. Let
$a=(L(xR,e),Q(y,e'))$ be its arbitrary point.
Hence $P\alpha(x)^{-1}e=Pe'$. Thus there exists an element $p\in P$ such that $pe'=\alpha(x)^{-1}e$. We get
$\phi(L([x:p:y],e))=a$.
\end{proof}

\begin{theorem}\label{theorem:2}
Suppose that the quotient map $E\to P\lq E$ is a principal $P$-bundle. Then $\phi$ is a topological embedding.
\end{theorem}
\begin{proof} We denote this quotient map by $\pi$ and by $D$ the fibre product from the formulation of Lemma~\ref{lemma:1}.
Let $a$ be any point of $D$. The assumption of this lemma implies that there exists
an open neighborhood $U\subset P\lq E$ of the point $\pi_Rv_A\pr_1(a)=\pi_Q\can\pr_2(a)$
and a $P$-equivariant homeomorphism $h:P \times U\ito\pi^{-1}(U)$ such that the diagram
$$
\begin{tikzcd}
P\times U\arrow{rr}{h}[swap]{\sim}\arrow{rd}[swap]{\pr_2}&&\arrow{ld}{\pi}\pi^{-1}(U)\\
&U
\end{tikzcd}
$$
is commutative.
% and $h(u,p'p)=p'h(u,p)$ for any $u\in U$ and $p',p\in P$.
We denote $\p=\pr_1 h^{-1}$. Clearly, $\p(pe)=p\p(e)$ for any $p\in P$ and $e\in\pi^{-1}(U)$.

%We are going to define a continuous map
%$$
%\psi:(v^{-1}(U)\times w^{-1}(U))\cap D\to(X\mathop{\times}\limits_B Y)\;{}_T\!\times E
%$$
%inverting $\phi$.
We define the continuous function
$$
\xi:(\pi_Rv_A)^{-1}(U)\times(\pi_Q\can)^{-1}(U)\to (X\mathop{\times}\limits_R P\mathop{\times}\limits_QY)\,{}_L\!\!\times\,E
$$ as follows. Let $d=(L(xR,e),Q(y,e'))$ be a point of %$(\pi_Rv_A)^{-1}(U)\times(\pi_Qw)^{-1}(U)$.
the domain of $\xi$.
We set
$$
\xi(d)=L\big([x:\p(\alpha(x)^{-1}e)\p(e')^{-1}:y],e\big).
$$
One can easily check that this definition does not depend on the choice of $x,y,e,e'$
in the representation of $d$.

Now suppose that $d\in D$. Then we get $P\alpha(x)^{-1}e=Pe'$.
Let us write $\alpha(x)^{-1}e=pe'$ for the corresponding $p\in P$.
We get
\begin{multline*}
\phi\xi(d)=\Big(L(xR,e),Q\big(y ,\p(e')\p(\alpha(x)^{-1}e)^{-1}\alpha(x)^{-1}e\big)\Big)
=\big(L(xR,e),Q\big(y ,\p(e')\p(pe')^{-1}pe'\big)\big)\\
=\big(L(xR,e),Q\big(y ,\p(e')\p(e')^{-1}p^{-1}pe'\big)\big)
=\big(L(xR,e),Q\big(y ,e'\big)\big)=d.
\end{multline*}
This calculation proves that the restriction of $\xi$ to $(\pi_Rv_A)^{-1}(U)\times(\pi_Q\can)^{-1}(U)\cap D$
inverts $\phi$. As $a$ belongs to the last subset, all such subsets cover $D$.
Thus the required claim follows from Lemma~\ref{lemma:1} and Proposition~\ref{proposition:1}.
\end{proof}

\subsection{Canonical and twisted projections for the triple product}\label{Canonical and twisted projections for the triple product}
By~(\ref{eq:5}), we get the following commutative diagram:
\begin{equation}\label{eq:6}
\begin{tikzcd}
\(X\mathop{\times}\limits_R P\mathop{\times}\limits_Q Y\)\;{}_L\!\times E\arrow{r}{\phi}\arrow{d}[swap]{\can}&((X/R)\,{}_L\!\times E)\times(Y\,{}_Q\!\times\,E)\arrow{d}{\pr_1}\\
L\lq E&\arrow{l}{\can}((X/R)\,{}_L\!\times E)
\end{tikzcd}
\end{equation}
%where $\xi$ is the canonical projection.

Now let $M$ be a subgroup of $G$ and $B:Y\to G/M$ be a $Q$-equivariant continuous map.
%Then we have the $v_B:Y\times_Q E\to M\lq G$ defined by~(\ref{eq:1}).
We consider the map $\alpha B:X\mathop{\times}\limits_R P\mathop{\times}\limits_Q Y\to G/M$
given by
\begin{equation}\label{eq:tw:7}
[x:p:y]\mapsto \alpha(x)p B(y).
\end{equation}
It is easy to check that this map is well-defined, continuous and $L$-equivariant. By~(\ref{eq:5}), we get the following commutative diagram:
\begin{equation}\label{eq:7}
\begin{tikzcd}
\(X\mathop{\times}\limits_R P\mathop{\times}\limits_Q Y\)\;{}_L\!\times E\arrow{r}{\phi}\arrow{d}[swap]{v_{\alpha B}}&((X/R)\,{}_L\!\times E)\times(Y\,{}_Q\!\times\,E)\arrow{d}{\pr_2}\\
M\lq E&\arrow{l}{v_B}Y\,{}_Q\!\!\times\,E
\end{tikzcd}
\end{equation}

\subsection{The difference}In this section, we want to study the difference
$$
((X/R)\,{}_L\!\!\times\,E)\times(Y\,{}_Q\!\!\times\,E)-\im\phi.
$$
To this end, we will claim some additional properties of $E$.
%Let $U$ be a topological group acting continuously on $E$ on the right
%so that the following properties hold:
%{\renewcommand{\labelenumi}{{\rm\theenumi}}
%\renewcommand{\theenumi}{{\rm(\roman{enumi})}}
%\begin{enumerate}\setcounter{enumi}{5}
%\itemsep=3pt
%\item\label{p:6} the actions of $G$ and $U$ on $E$ commute.
%\item\label{p:7} $U$ acts transitevily on $E$.
%\item\label{p:8} for any point $m\in R\lq E$, the map $t_m:U\to R\lq E$ given by $t_m(u)=mu$
%                 is right invertible in an open neighbourhood of $m$.
%\end{enumerate}}
%\noindent
%The last condition looks artificial but it follows from~\ref{p:6}--\ref{p:7} in ``good'' cases, see Lemma...

%***
%
%We are going to use the following ad hoc terminology.
%
%\begin{definition}
%A topological space $M$ is {\it strongly homogeneous} if there exists a topological
%group $U$ acting continuously and transitively on $M$ on the right so that
%for any point $m\in M$ the map $U\to M$ such that $u\mapsto mu$
%is right invertible in an open vicinity of $m$.
%\end{definition}
%\noindent
%To specify the group $U$, we say that $M$ is strongly homogeneous with respect to $U$.
%Lemma ... below provides us with typical situations of strongly homogeneous spaces.

\begin{lemma}\label{lemma:3}
Suppose that $P$ acts freely on $E$.
Let $U$ be a compact Lie group acting continuously on $E$ on the right
so that this action and the left action of $G$ commute.
Suppose that the space $R\lq E$ is Hausdorff and is acted upon by $U$ transitively.
Then the projection to the first component
$$
\omega:((X/R)\,{}_L\!\!\times\,E)\times(Y\,{}_Q\!\!\times\,E)-\im\phi\to(X/R)\,{}_L\!\!\times\,E
$$
is a fibre bundle with fibre homeomorphic to
$$
F_{\bar e}=\{Q(y,e)\in Y\,{}_Q\!\!\times\,E\suchthat \bar e\ne Pe\}=Y\,{}_Q\!\!\times\,E-(\pi_Q\can)^{-1}(\bar e),
$$
where $\bar e$ is an arbitrary point of $P\lq E$.
\end{lemma}
\begin{proof} We will denote the action of $U$ by $\cdot$.
We make $U$ act on $Y\,{}_Q\!\times\,E$ on the right by the rule
$Q(y,e)\cdot u=Q(y,e\cdot u)$. Then we obviously get $F_{\bar e}\cdot u=F_{\bar e\cdot u}$. As $U$ acts transitively on $R\lq E$,
it also acts transitively on $P\lq E$.
Therefore, the spaces $F_{\bar e}$ are homeomorphic for different points $\bar e$.
We will rely on the notation of diagram~(\ref{eq:c}).

Let $a\in(X/R)\,{}_L\!\!\times\,E$ be an arbitrary point.
By Lemma~\ref{lemma:1}, we get
$$
\omega^{-1}(a)=\{a\}\times F_{\pi_Rv_A(a)}\cong F_{\pi_Rv_A(a)}\cong F_{\bar e}.
$$

We consider the map $t:U\to R\lq E$ defined by $t(u)=v_A(a)u$.
Let $V$ be the stabilizer of $v_A(a)$ in $U$. Note that $V$ is a Lie group
as it is closed in $U$. Let us consider the commutative diagram
$$
\begin{tikzcd}
{}&U\arrow{rd}{t}\arrow{ld}[swap]{\text{quotient map}}\\
V\lq U\arrow{rr}&&R\lq E
\end{tikzcd}
$$
The bottom map, which is given by $Vu\mapsto v_A(a)\cdot u$, is continouous and bijective.
It is a homeomorphism being a map from a compact space to a Hausdorff space.
As the natural projection $U\to V\lq U$ is a principal
$V$-bundle, it has a continous section in an open neighbourhood of any point of $V\lq U$.
From the diagram above, we get the same property for the map $t$.
In particular, there exist an open neighbourhood $W$ of $v_A(a)$ in $R\lq E$ and
a continuous section $s:W\to U$ of $t$. In other words, we have
\begin{equation}\label{eq:3}
w=t(s(w))=v_A(a)\cdot s(w)
\end{equation}
for any $w\in W$.

Let $H=v_A^{-1}(W)$. It is an open subset of $(X/R)\,{}_L\!\!\times\,E$ containing $a$. We construct the map
$\sigma:H\times\omega^{-1}(a)\to((X/R)\,{}_L\!\!\times\,E)\times(Y\,{}_Q\!\!\times\,E)$ by
$$
\Big(h,\big(a,Q(y,e)\big)\Big)\mapsto\Big(h,Q\big(y,e\cdot sv_A(h)\big)\Big).
$$
This map is obviously well-defined and continuous.
Let us prove that $\im\sigma\cap\im\phi=\emptyset$. By Lemma~\ref{lemma:1}, we have to prove the inequality
\begin{equation}\label{eq:4}
\pi_Rv_A(h)\ne\pi_Q(  Qe\cdot sv_A(h) )=Pe\cdot sv_A(h).
\end{equation}
%Due to the associativity condition~\ref{p:6},
As the $G$- and $U$-actions on $E$ commute, the map $\pi_R$ is $U$-equivariant. Therefore~(\ref{eq:4})
is equivalent to
$$
\pi_R\big(v_A(h)\cdot sv_A(h)^{-1}\big)\ne Pe.
$$
Applying~(\ref{eq:3}) for $w=v_A(h)$, to compute the left-hand side, we get the equivalent inequality
$$
\pi_Rv_A(a)\ne Pe,
$$
which holds by Lemma~\ref{lemma:1}, as $(a,Q(y,e))\notin\im\phi$. Thus we can consider $\sigma$ as a continuous map
from $H\times\omega^{-1}(a)$ to $\omega^{-1}(H)$. The diagram
$$
\begin{tikzcd}
H\times\omega^{-1}(a)\arrow{rr}{\sigma}\arrow{rd}[swap]{\pr_1}&&\omega^{-1}(H)\arrow{ld}{\omega}\\
&H
\end{tikzcd}
$$
is obviously commutative.

To prove that $\sigma$ is a homeomorphism, we notice that its inverse map $\omega^{-1}(H)\to H\times\omega^{-1}(a)$
is given by
$$
\big(h,  Q(y,e)  \big)\mapsto\Big(h,\big(a,Q(y,e\cdot sv_A(h)^{-1})\big)\Big).
$$
\end{proof}

Let us look more closely at the fibre $F_{\bar e}$. We will establish the vanishing theorem for its odd degree cohomologies
with compact support under some restrictions on the spaces entering into our initial construction.
We will need the following result.

\begin{lemma}\label{lemma:4}
Suppose that the quotient map $E\to P\lq E$ is a principal $P$-bundle.
Then the composition
$$
\begin{tikzcd}
Y\,{}_Q\!\times\,E\arrow{r}{\can}&Q\lq E\arrow{r}{\pi_Q}&P\lq E.
\end{tikzcd}
$$
is a fibre bundle with fibre $Y\,{}_Q\times P$.
\end{lemma}
\begin{proof}
Let us again denote the quotien map $E\to P\lq E$ by $\pi$.
Let $\bar e$ be an arbitrary point of $P\lq E$. By the hypothesis of the lemma,
there exists an open neibourhood $U\subset P\lq E$ of $\bar e$
and a $P$-equivariant homeomorphism $h$ such that the diagram
$$
\begin{tikzcd}
P\times U\arrow{rr}{h}[swap]{\sim}\arrow{dr}[swap]{\pr_2}&&\pi^{-1}(U)\arrow{dl}{\pi}\\
&U&
\end{tikzcd}
$$
is commutative. We define thae map $f:U\times(Y\,{}_Q\!\times\,P)\to Y\,{}_Q\!\!\times\,E$
%(\pi_Qw)^{-1}(U)$
by
$$
f(u,Q(y,p))=Q(y,h(p,u)).
$$
We have
$$
\pi_Q\can f(u,Q(y,p))=\pi_Q(Qh(p,u))=Ph(p,u)=\pi h(p,u)=\pr_2(p,u)=u.
$$
This calculation proves that the image of $f$ is contained in $(\pi_Q\can)^{-1}(U)$ and
the diagram
$$
\begin{tikzcd}
U\times(Y\,{}_Q\!\times\,P)\arrow{rr}{f}\arrow{dr}[swap]{\pr_1}&&(\pi_Q\can)^{-1}(U)\arrow{dl}{\pi_Q\can}\\
&U&
\end{tikzcd}
$$
is commutative.

To prove that $f$ is a homeomorphism, let us compute its inverse. We write the inverse map of $h$
as follows $h^{-1}(e)=(\p(e),\u(e))$. It is easy to check the that
$$
\u(pe)=\u(e),\qquad \p(pe)=p\p(e)
$$
for any $e\in\pi^{-1}(U)$ and $p\in P$. With this notation, we get
$$
f^{-1}(Q(y,e))=\big(\u(e),Q(y,\p(e))\big).
$$

\end{proof}

\subsection{The tensor product}\label{The_tensor_product} The above results can be applied to equivariant cohomologies.
Let us suppose that all quotient maps $E\to L\lq E$, $E\to P\lq E$, $E\to R\lq E$, $E\to Q\lq E$, $E\to M\lq E$
are universal principal bundles for the respective groups. We denote by $\Theta$
the following composition:
%$$
%\begin{tikzcd}[column sep=15pt]
%H^\bullet_L(X/R,\k)\otimes_\k H^\bullet_Q(Y,\k)\arrow{r}{\times}&
%H^\bullet\big(((X/R)\,{}_L\!\!\times\,E)\times(Y\,{}_Q\!\!\times\,E),\k\big)\arrow{r}{\phi}&
%H_L^\bullet\(X\mathop{\times}\limits_R P\mathop{\times}\limits_Q Y,\k\).
%\end{tikzcd}
%$$
%\begin{equation}\label{eq:9}
%\begin{array}{cc}%[column sep=15pt]
%\hspace{-90pt}H^\bullet_L(X/R,\k)\otimes_\k H^\bullet_Q(Y,\k)\stackrel\times\to
%H^\bullet\big(((X/R)\,{}_L\!\!\times\,E)\times(Y\,{}_Q\!\!\times\,E),\k\big)\\
%\hspace{300pt}\stackrel{\phi^*}\to H_L^\bullet\(X\mathop{\times}\limits_R P\mathop{\times}\limits_Q Y,\k\).
%\end{array}
%\end{equation}
\begin{equation}\label{eq:9}
\begin{tikzcd}[column sep=15pt]
H^\bullet_L(X/R,\k)\otimes_\k H^\bullet_Q(Y,\k)\arrow{r}{\times}\arrow{rd}[swap]{\Theta}&[30pt]
H^\bullet\big(((X/R)\,{}_L\!\!\times\,E)\times(Y\,{}_Q\!\!\times\,E),\k\big)\arrow{d}{\phi^*}\\
&[30pt]H_L^\bullet\(X\mathop{\times}\limits_R P\mathop{\times}\limits_Q Y,\k\).
\end{tikzcd}
\end{equation}

\noindent
The map $\times$ is the cross product and is given by $a\otimes b\mapsto \pr_1^*(a)\cup\pr_2^*(b)$.

According to Section~\ref{Twisted action}, the left factor $H^\bullet_L(X/R,\k)$ is the right $H_R^\bullet(\pt,\k)$-module
with the action twisted by $A$. On the other hand, the right factor $H^\bullet_Q(Y,\k)$ is the left
$H_Q^\bullet(\pt,\k)$-module with respect to the canonical action (Section~\ref{Equivariant cohomology}).
Moreover, the quotient maps $\pi_R$ and $\pi_Q$, see diagram~(\ref{eq:c}), yield the following maps:
$$
\pi_R^*:H_P^\bullet(\pt,\k)\to H_R^\bullet(\pt,\k),\quad \pi_Q^*:H_P^\bullet(\pt,\k)\to H_Q^\bullet(\pt,\k).
$$
They allow us to define the structure of a right $H_P^\bullet(\pt,\k)$-module on $H^\bullet_L(X/R,\k)$
and of a left $H_P^\bullet(\pt,\k)$-module on $H^\bullet_Q(Y,\k)$.

\begin{lemma}\label{lemma:5}
The map $\Theta$ factors through $H^\bullet_L(X/R,\k)\otimes_{H_P^\bullet(\pt,\k)}H^\bullet_Q(Y,\k)$.
\end{lemma}
\begin{proof}
Let $m\in H^\bullet_L(X/R,\k)$, $n\in H^\bullet_Q(Y,\k)$,
and $a\in H_P^\bullet(\pt,\k)$. By Lemma~\ref{lemma:c} and diagram~(\ref{eq:9}) defining $\Theta$, we get
\begin{multline*}
\Theta(m a\otimes n)=\Theta(m\pi_R^*( a)\otimes n)=\Theta(m\cup v_A^*\pi_R^*( a)\otimes n)
=\phi^*(\pr_1^*(m\cup v_A^*\pi_R^*( a))\cup\pr_2^*( n))\\
%=\phi^*\pr_1^*(m)\cup \phi^*\pr_1^*v_A^*\pi_R^*( a)\cup\phi^*\pr_2^* n=
=\phi^*\pr_1^*(m)\cup (\pi_R v_A \pr_1 \phi)^*( a)\cup\phi^*\pr_2^*( n)
=\phi^*\pr_1^*(m)\cup (\pi_Q \can \pr_2\phi)^*( a)\cup\phi^*\pr_2^*( n)\\
=\phi^*(\pr_1^*(m)\cup \pr_2^*  \can^*  \pi_Q^*( a)\cup\pr_2^*( n))
=\phi^*(\pr_1^*(m)\cup \pr_2^*(\can^*  \pi_Q^*( a)\cup n))\\
=\Theta(m\otimes \can^*\pi_Q^*( a)\cup n)=\Theta(m\otimes \pi_Q^*( a) n)=\Theta(m\otimes a n).
\end{multline*}
\end{proof}

\noindent
We denote the map induced by $\Theta$ as follows
$$
\theta:H^\bullet_L(X/R,\k)\otimes_{H_P^\bullet(\pt,\k)}H^\bullet_Q(Y,\k)\to H_L^\bullet\(X\mathop{\times}\limits_R P\mathop{\times}\limits_Q Y,\k\).
$$
%In the rest of this section, we will find sufficient conditions under which this map is an isomorphism.
We make the left-hand side into a ring by the classic rule
$$
(m\otimes n)(m'\otimes n')=(-1)^{ij}(m\cup m'\otimes n\cup n')
$$
for $n\in H^i_Q(Y,\k)$ and $m\in H^j_L(X/R,\k)$.

In the next lemma, we use the notation of Section~\ref{Canonical and twisted projections for the triple product}.

\begin{theorem}\label{lemma:6}
The map $\theta$ is a homomorphism of rings and of left $H^\bullet_L(\pt,\k)$-modules with respect to the canonical actions.
For any $Q$-equivariant continuous map $B:Y\to G/M$, the map $\theta$ is also a homomorphism
of right $H^\bullet_M(\pt,\k)$-modules with respect to the actions twisted by $B$ and $\alpha B$,
on the domain and the codomain of $\theta$,
respectively.
\end{theorem}
\begin{proof}
The fact that $\theta$ is a homomorphism of rings can be proved following 
the calculation before~\cite[Theorem 3.16]{Hatcher_AT}.

Let $a\in H^\bullet_L(\pt,\k)$, $m\in H^\bullet_L(X/R,\k)$, and $n\in H^\bullet_Q(Y,\k)$.
As diagram~(\ref{eq:6}) is commutative, we get
\begin{multline*}
\theta( a m\otimes n)=\Theta\big((\can^*( a)\cup m)\otimes n\big)=\phi^*\big(\pr_1^*(\can^*( a)\cup m)\cup\pr_2^*( n)\big)\\
=(\can\pr_1\phi)^*( a)\cup\phi^*(\pr_1^*( m)\cup\pr_2^*( n))=\can^*( a)\cup\Theta( m\otimes n)
= a\theta( m\otimes n).
\end{multline*}
Now let $ m\in H^\bullet_L(X/R,\k)$, $ n\in H^\bullet_Q(Y,\k)$, and $b\in H^\bullet_M(\pt,\k)$.
As diagram~(\ref{eq:7}) is commutative, we get
\begin{multline*}
\theta( m\otimes nb)=\Theta\big( m\otimes (n\cup v_B^*(b))\big)=\phi^*\big(\pr_1^*( m)\cup\pr_2^*( n\cup v_B^*(b))\big)\\
=\phi^*\big(\pr_1^*( m)\cup \pr_2^*( n)\big)\cup \big(v_B\pr_2\phi\big)^*(b)=\Theta( m\otimes n)\cup v_{\alpha B}^*(b)
=\theta( m\otimes n)b.
\end{multline*}
\end{proof}

\section{The case of a compact Lie group}\label{The_case_of_a_compact_Lie_group}
\noindent In this section, we consider the case where $G$ is a compact Lie group.

\subsection{Stiefel manifolds}\label{Stiefel_manifolds}
%For the rest of the paper, we will assume that . For the sake of the K\"unneth isomorphism
%we also suppose that .
As $G$ can be considered a closed subgroup of the unitary group $U(\n)$ for $\n$ big enough,
we will explain how to find a universal principal $U(\n)$-bundle.
For any natural number $N>\n$,
we consider the {\it Stiefel manifold\,}\footnote{This space is usually denoted by $W_{N,\n}$ or $V_\n({\mathbb C}^N)$ or
${\mathbb C} V_{N,\n}$. We also transpose matrices, as we want to have a left action of $U(\n)$.}
$$
%V_n({\mathbb C}^N)
E^N=\{A\in \Mat_{\n,N}(\mathbb C)\suchthat A\bar A^T=I_\n\},
$$
where $\Mat_{\n,N}({\mathbb C})$ is the space of $\n\times N$ complex matrices with respect to the metric topology and
$I_\n$ is the identity $\n\times \n$ matrix. The group $U(\n)$ %of unitary $\n\times \n$ matrices
acts on $E^N$ on the left by multiplication.
Similarly, $U(N)$ acts on $E^N$ on the right by multiplication. %We will denote this action by $\cdot$.
The last action is transitive and commutes with the left action of $U(\n)$.
The quotient space $\Gr^N=U(\n)\lq E^N$ is called the {\it Grassmanian} and the corresponding
quotient map $E^N\to\Gr^N$ is a principal $U(\n)$-bun\-dle. Note that the group $U(N)$ also acts on $\Gr^N$
by the right multiplication. For $N<N'$, we have the smooth embedding
$E^N\hookrightarrow E^{N'}$ by adding $N'-N$
zero columns to the right.

Taking the direct limits
$$
%E=\mathop{\dlim}_{N}E^N,\quad
E=\dlim E^N,\quad \Gr=\dlim\Gr^N,
$$
we get a universal principal $U(\n)$-bundle $E\to\Gr$.
These spaces satisfy the following properties~\cite[Proposition~9.1]{Borel}.
%, following, for example, from~\cite
\begin{proposition}\label{proposition:3}
\begin{enumerate}
\item\label{proposition:3:p:1} $E^N$ is simply-connected.\\[-6pt]
\item\label{proposition:3:p:2} $H^n(E^N,\k)=0$ for %$0<n<2(N-\n)+1$.\\[-6pt]
                                                   $0<n\le 2(N-\n)$.\\[-6pt]
\item\label{proposition:3:p:3} $H^n(E^N,\k)$ is %a
                               free
                               %$\k$-module
                               of finite rank
                               for any $n$.
\end{enumerate}
\end{proposition}

\begin{corollary}\label{corollary:sc}
For any connected closed subgroup $L$ of $G$, the space $L\lq E^N$ is simply connected.
\end{corollary}
\begin{proof}
As $E^N\to L\lq E^N$ is a fibre bundle with fibre $L$,
the result follows from Proposition~\ref{proposition:3}\ref{proposition:3:p:1}
and the long exact sequence of homotopy groups
$$
\{1\}=\pi_1(E^N)\to \pi_1(L\lq E^N)\to \pi_0(L)=\{1\}.
$$
\end{proof}

%We need the spaces $E^N$ to get the principal $K$-bundles $E^N\to E^N/K$.
%It is easy to note that this bundle for $N=\infty$ is the direct limit of the bundles for $N<\infty$.

\begin{proposition}\label{lemma:tw5:2} 
Let $L$ be a closed subgroup of $G$ and $X$ be a left $L$-space. Then the restriction map
\begin{equation}\label{eq:8}
H_L^n(X,\k)\to H^n(X\,{}_L\!\times E^N,\k)
\end{equation}
is an isomorphism for $n\le 2(N-\n)$.
%These isomorphisms are natural with respect to $X_1$,\ldots, $X_m$,
%the cup product, and the projections to the factors.
%In particular,
%we get the following commutative diagram:
%$$
%\begin{tikzcd}
%\displaystyle
%\mathop{{\bigotimes{}_\k}}\limits_{n_1+\cdots+n_m=n} H^{n_i}\(X_i\,{}_G\!\times E,\k\)\arrow{r}\arrow[equal]{d}[swap]{\wr}&\displaystyle H^n\(\prod_{i=1}^mX_i\,{}_G\!\times E,\k\)\arrow[equal]{d}{\wr}\\
%\displaystyle
%\mathop{{\bigotimes{}_\k}}\limits_{n_1+\cdots+n_m=n} H^{n_i}\(X_i\,{}_G\!\times E^N,\k\)\arrow{r}&\displaystyle H^n\(\prod_{i=1}^mX_i\,{}_G\!\times E^N,\k\)
%\end{tikzcd}
%$$
%\noindent
%where the horizontal arrow are given by the cross product.
\end{proposition}
%\begin{proof}
%The result follows if we consider the following direct products of fiber bundles:
%$$
%\begin{tikzcd}
%{}&\displaystyle\prod_{i=1}^m G\lq(X_i\times E\times E^N)\arrow{dr}{\pi^N}\arrow{dl}[swap]{\pi}&\\
%\displaystyle\prod_{i=1}^mX_i\,{}_G\!\times E&&\displaystyle\prod_{i=1}^mX_i\,{}_G\!\times E^N
%\end{tikzcd}
%$$
%where
%$$
%\pi=\prod_{i=1}^N G\lq\pr_{1,2},\qquad \pi^N=\prod_{i=1}^N G\lq\pr_{1,3}.
%$$
%%similar to~(\nef{eq:bt6:2.5}).
%The bundle $\pi$ has fiber $(E^N)^m$ and the bundle $\pi^N$ has fiber $E^m$. The last space
%is contractible and $H^n((E^N)^m)=0$ for $0<n<a_\k(E^N)$ by the K\"unneth formula.
%Taking cohomologies and apply Lemma~..., we get the commutative diagram
%$$
%\begin{tikzcd}
%{}&H^n\(\displaystyle\prod_{i=1}^m G\lq(X_i\times E\times E^N)\)&\\
%H^n\(\displaystyle\prod_{i=1}^mX_i\,{}_G\!\times E\)\arrow{ur}{\pi^*}[swap]{\sim}\arrow[dashed]{rr}{\sim}&&H^n\(\displaystyle\prod_{i=1}^mX_i\,{}_G\!\times E^N\)\arrow{ul}[swap]{(\pi^N)^*}{\sim}
%\end{tikzcd}
%$$
%The required isomorphism is the dashed arrow.
%The fact that these isomorphisms preserve cup product and projection to factors
%can be checked routinely.
%\end{proof}

\begin{corollary}\label{corollary:2}
Let $L$ be a closed subgroup of $G$, $X$ be a left $L$-space, and $\n<N'<N$. Then the restriction map
$$
H^n(X\,{}_L\!\times E^N,\k)\to H^n(X\,{}_L\!\times E^{N'},\k)
$$
is an isomorphism for $n\le 2(N'-\n)$.
\end{corollary}

\begin{remark}\label{remark:2.5}\rm
We will always mean this choice of $E$ in the remainder of the paper.
%applied to
%the theory developed in Section~\ref{The_tensor_product_morphism}.
If, however, we replace it by $E^N$, then we will write
$v_A^N$, $\phi^N$ and $F^N_{\bar e}$ instead
of $v_A$, $\phi$ and $F_{\bar e}$, respectively.
We will refer to these objects as {\it compact} versions 
of the respective infinite dimensional objects.
\end{remark}

\subsection{Push-forward}\label{push-forward} For any topological space $X$,
we consider the dualizeing complex $\omega_X=a_X^!\k$. %where $a_X:X\to\pt$ is the constant map.
This complex allows us to define the Borel-Moor
homology $H^{\rm BM}_i(X,\k)={\mathbb H}^{-i}(X,\omega_X)$. Any proper map $f:X\to Y$ induces
the {\it push-forward} map $f_*:H^{\rm BM}_i(X,\k)\to H^{\rm BM}_i(Y,\k)$.

Suppose additionally that $X$ and $Y$ are orientable of dimensions $n$ and $m$ respectively.
This means that there exist isomorphisms $\omega_X\cong\k[n]$ and $\omega_Y\cong\k[m]$,
which we fix. Then we get the following map between the cohomologies:
$$
\begin{tikzcd}
H^{n-i}(X,\k)\cong H^{\rm BM}_i(X,\k)\arrow{r}{f_*}&H^{\rm BM}_i(Y,\k)\cong H^{m-i}(X,\k),
\end{tikzcd}
$$
wich we also denote by $f_*$. This map is also called the {\it Gysin map}.

\subsection{Orientation of quotient spaces}\label{Orientation of quotient spaces} Remember that an {\it orientation} of a
finite-dimen\-sional real vector space $V$ is an equivalence class of its bases,
where two bases are assumed to be equivalent
iff the transition matrix between them
has a positive determinant.
Any basis belonging to the chosen equivalence class is called {\it oriented}. 
If an orientation is chosen for $V$, then $V$ is called {\it oriented}.

Suppose that $U$ is an oriented subspace of an oriented vector space $V$.
Then we will fix the following orientation of the quotient space $V/U$:
let $v_1,\ldots,v_n$ be an oriented basis of $V$ such that $v_1,\ldots,v_m$
is an oriented basis of $U$ for some $m\le n$. Then the orientation of $V/U$
is the one containing the basis $v_{m+1}+U,\ldots,v_n+U$.
The argument with block matrices proves the this definition is independent
of the choice of bases.

We can define orientation of the quotient of oriented vector bundles
if we define the orientation of each fibre using the rule described above.

\subsection{Compatibly oriented squares}  We will follow here the presentation
of this topic by Fulton~\cite{Fulton}. Consider a Cartesian square
$$
\begin{tikzcd}
X'\arrow{r}{g'}\arrow{d}[swap]{f'}&X\arrow{d}{f}\\
Y'\arrow{r}{g}&Y
\end{tikzcd}
$$
where $X,X',Y,Y'$ are compact smooth oriented manifolds and $f,f',g,g'$ are smooth maps.
Let us factor $f$ through a closed embedding $X\hookrightarrow Y\times{\mathbb R}^N$
followed by the projection to $Y$. We say that the above square is compatibly oriented if
$$
\nu_{X',Y'\times{\mathbb R}^N}\cong (g')^*\nu_{X,Y\times{\mathbb R}^N}
$$
as oriented vector bundles, where the quotients entering in the definitions of
the normal bundles $\nu_{X',Y'\times{\mathbb R}^N}$ and $\nu_{X,Y\times{\mathbb R}^N}$
are orientated as described in Section~\ref{Orientation of quotient spaces}.
If this isomorphism holds and $\dim X-\dim Y=\dim X'-\dim Y'$ then we get
\begin{equation}\label{eq:24}
f'_*(g')^*=g^*f_*.
\end{equation}
We will refer to this property as the {\it naturality of push-forward}.

\subsection{Smooth structure on the Borel constructions}\label{Smooth structure on the Borel constructions} %Let us come back to the notation of Section~\ref{Stiefel_manifolds}.
%Let $N>\n$.
%As we have the fibre bundle $L\lq E^N\to U(\n)\lq E^N=\Gr^N$ with smooth fibre
%$L\lq U(\n)$ and the Grassmanian $\Gr^N$ is also smooth, we get that
Let $L$ be a closed connected subgroup of $G$. By the quotient manifold theorem,
the space $L\lq E^N$ is a smooth manifold.
It is orientable being simply connected by Corollary~\ref{corollary:sc}.

Let $X$ be any smooth manifold acted upon smoothly by $L$ on the left. Applying the quotient manifold theorem
once again, we get that the space $X\;{}_L\! \times E^N$ is a smooth manifold.
Its smooth structure can be described as follows. Let $L(x,e)$ be an arbitrary point
of $X\;{}_L\! \times E^N$. There exists a coordinate chart $(U,\phi)$ of $L\lq E^N$
containing $Le$.
Shrinking $U$ if necessary, we may suppose that the restriction of the fibre bundle
$E^N\to L\lq E^N$ to $U$ is a trivial bundle.
Then the restriction of the fibre bundle $\can:X\;{}_L\! \times E^N\to L\lq E^N$ to $U$
is also trivial. Let $h:X\times U\ito\can^{-1}(U)$ be such a trivialization.
We get $h^{-1}(L(x,e))=(x',Le)$ for some $x'\in Lx\subset X$. There exists a coordinate
chart $(V,\psi)$ of $X$ containing $x'$. Now we get the coordinate chart
$(h(V\times U),(\psi\times\phi)h^{-1})$ containing the original point $L(x,e)$.
These charts are consistently oriented if $X$ is oriented, as $L$ acts on $X$
by orientation-preserving diffeomorphisms.

\subsection{Equivariant push-forward} Let $X$ and $Y$ be compact smooth orientable manifolds
acted upon smoothly on the left by a closed connected subgroup $L$ of $G$.
Let $f:Y\to X$ be a smooth $L$-equivariant map.
For any $N>\n$, we get
the smooth proper map $f\;{}_L\! \times \id:Y\;{}_L\! \times E^N\to X\;{}_L\! \times E^N$,
which we denote by $f^N$.
As both domain and codomain of this map are smooth orientable manifolds as described
in Section~\ref{Smooth structure on the Borel constructions}, we can consider the ordinary
(non-equivariant) push forward
$$
f^N_*:H^n(Y\;{}_L\! \times E^N,\k)\to H^{n+d}(X\;{}_L\! \times E^N,\k),
$$
where $d=\dim X-\dim Y$. So by Proposition~\ref{lemma:tw5:2}, we actually get the map
$$
f_\star:H^n_L(Y,\k)\to H_L^{n+d}(X,\k)
$$
as soon as $n\le\min\{2(N-\n),2(N-\n)-d\}$ defined
\begin{equation}\label{eq:17}
f_\star(m)=(i_X^*)^{-1}f_*^Ni_Y^*(m)
\end{equation}
for $m\in H^n_L(Y,\k)$, where $i_X:X\;{}_L\! \times E^N\hookrightarrow X\;{}_L\! \times E$
and $i_Y:Y\;{}_L\! \times E^N\hookrightarrow Y\;{}_L\! \times E$ are the natural embeddings.
This definition of $f_\star$ does not depend of $N$,
as by the naturality of push-forward, we have the following commutative diagram for any $N'>N$:
$$
\begin{tikzcd}[column sep=50pt]
H^n(Y\;{}_L\! \times E^{N'},\k)\arrow{r}{f^{N'}_*}\arrow{d}[swap]{\wr}& H^{n+d}(X\;{}_L\! \times E^{N'},\k)\arrow{d}{\wr}\\
H^n(Y\;{}_L\! \times E^N,\k)\arrow{r}{f^N_*}& H^{n+d}(X\;{}_L\! \times E^N,\k)
\end{tikzcd}
$$
where the vertical maps are restrictions, which are isomorphisms by Corollary~\ref{corollary:2}.

By the projection formula, the push-forward $f_\star$ is a morphism of left $H_L^\bullet(\pt,\k)$-modules
with respect to the canonical left action if the cohomology $H_L^\bullet(\pt,\k)$ vanishes in odd degrees.
Let $R$ be a closed subgroup of $G$ and $A:X\to G/R$ be an $L$-equivariant continuous
%smooth?
map (as in Section~\ref{Twisted action}).
Then the composition $Af:Y\to G/R$ is also $L$-equivariant.
We get the commutative diagrams\footnote{See Remark~\ref{remark:2.5} about the notation.}
\begin{equation}\label{eq:18}
\begin{tikzcd}
X\;{}_L\! \times E\arrow{r}{v_A}& R\lq E\\
X\;{}_L\! \times E^N\arrow{r}{v_A^N}\arrow[hook]{u}{i_X}& R\lq E^N\arrow[hook]{u}[swap]{i_R}
\end{tikzcd}\qquad
\begin{tikzcd}
Y\;{}_L\! \times E\arrow{r}{v_{Af}}& R\lq E\\
Y\;{}_L\! \times E^N\arrow{r}{v_{Af}^N}\arrow[hook]{u}{i_Y}& R\lq E^N\arrow[hook]{u}[swap]{i_R}
\end{tikzcd}
\end{equation}

Let $m\in H_L^n(Y,\k)$ and $b\in H_R^l(\pt,\k)$.
For $N$ big enough\footnote{This means that $n,n+l\le\min\{2(N-\n),2(N-\n)-d\}$.},
we get by~(\ref{eq:17}), the commutativity of diagrams~(\ref{eq:18}), the projection formula, 
and a compact version of~(\ref{eq:d}) that
\begin{multline*}
f_\star(mb)=f_\star(m\cup v_{Af}^*(b))=(i_X^*)^{-1}f_*^Ni_Y^*(m\cup v_{Af}^*(b))\\
=(i_X^*)^{-1}f_*^N\big(i_Y^*(m)\cup(v_{Af}i_Y)^*(b)\big)
=(i_X^*)^{-1}f_*^N\big(i_Y^*(m)\cup(v^N_{Af})^*i_R^*(b)\big)\\
=(i_X^*)^{-1}f_*^N\big(i_Y^*(m)\cup (f^N)^*(v^N_A)^*i_R^*(b)\big)
=(i_X^*)^{-1}\big(f_*^Ni_Y^*(m)\cup(i_Rv^N_A)^*(b)\big)\\
=(i_X^*)^{-1}\big(f_*^Ni_Y^*(m)\cup i_X^*v_A^*(b)\big)
=(i_X^*)^{-1}f_*^Ni_Y^*(m)\cup v_A^*(b)=f_\star(m)b.
\end{multline*}

So we have proved the following counterpart of Lemma~\ref{lemma:pull-back}.

\begin{lemma}\label{lemma:push-forward} The equivariant push-forward $f_\star:H^\bullet_L(Y,\k)\to H^{\bullet+\dim X-\dim Y}_L(X,\k)$
is a homomorphism of right $H_R^\bullet(\pt,\k)$-modules, where the right actions on the domain and codomain
are twisted by $Af$ and $A$, respectively.

If, moreover, the cohomology $H_L^\bullet(\pt,\k)$ vanishes in odd degrees, then $f_\star$ is
a homomorphism of left $H_L^\bullet(\pt,\k)$-modules, where both left actions of $H_L^\bullet(\pt,\k)$ are canonical.
\end{lemma}

\subsection{Semisimple groups}\label{Semisimple_groups} Suppose that $G$ is semisimple. Then $G$ can be described externally
as the set of points fixed by some analytic automorphism of
the corresponding Chevalley group~\cite[Theorem 16]{Steinberg}.

We consider a maximal torus $K<G$ and the Weyl group $W=N_G(K)/K$.
For any $w\in W$, we arbitrarily choose its lifting $\dot w\in N_G(K)$.
We will use the abbreviation $wK=\dot{w}K$.

The group $G$ has the Bruhat decomposition $G=\bigsqcup_{w\in W}K_w$, see~\cite[Corollary~5 of Lemma~45]{Steinberg}.
Taking the quotient, we get the decomposition $G/K=\bigsqcup_{w\in W}K_w/K$ into {\it Schubert cells}.
We can compute their closures, called {\it Schubert varieties},
as follows $\overline{K_w/K}=\bigsqcup_{w'\le w}K_{w'}$, where $\le$ is the Bruhat order.

We have $K\cong(S^1)^d$ for some $d$, where $S^1$ is the circle.
%If $E\to K\lq E$ is any universal $K$-bundle, then
We get the map $\rho_w:E\to E$ given by $\rho_w(e)=\dot we$. This map factors through the action
of $K$ and we get the map $K\lq\rho_w:K\lq E\to K\lq E$ and hence its pull-back
$(K\lq\rho_w)^*:H^\bullet_K(\pt,\k)\to H^\bullet_K(\pt,\k)$. %if $E$ is chosen to be universal.
We will think of this map as a left action of $W$ on $H^\bullet_K(\pt,\k)$:
\begin{equation}\label{eq:tw5:5}
w^{-1}(\mu)=(K\lq\rho_w)^*(\mu).
\end{equation}
This formula agrees with the standard choice $(S^\infty)^d\to{\mathbf P}^\infty({\mathbb C})^d$ of the universal principal $K$-bundle
and the ensuing representation of $H^\bullet_K(\pt,\k)$ as a polynomial ring, see~\cite[(20)]{bt}.

For any root $\alpha$, we denote by $G_\alpha$ the subgroup of $G$ generated
by the image of the root homomorphism $\phi_\alpha:\SU_2\to G$
and the torus $K$,
see~\cite[Chapter III (3.1)]{Bott_Samelson}. An alternative (external) description of this subgroup is given
by~\cite[Lemma~45]{Steinberg}. As $G_\alpha=G_{-\alpha}$, we can denote $G_{\omega_\alpha}=G_\alpha$, where $\omega_\alpha\in W$
is the reflection through the plane perpendicular to $\alpha$.
We will call elements $\omega_\alpha$ {\it reflections}.

\section{The isomorphism}\label{The isomorphism} We will prove that the map $\theta$ introduced in Section~\ref{The_tensor_product_morphism}
is an isomorphism under certain restrictions,
which we are going to formulate. First, we assume the following two conditions:

{
\renewcommand{\labelenumi}{{\rm \theenumi}}
\renewcommand{\theenumi}{{\rm(\Roman{enumi})}}
\begin{enumerate}
\itemsep=6pt
\item\label{restr:-1} the ring $\k$ has finite global dimension $\gld(\k)$;
\item\label{restr:0}  $G$ is a compact Lie group and $L$, $R$, $P$, $Q$, $M$ are its closed subgroups;
\end{enumerate}}
\noindent
These conditions will be supposed to hold for the rest of the paper.

\subsection{K\"unneth formula}
Let $X_1,\ldots,X_m$ be topological spaces such
that $H^n(X_j,\k)$ are free of finite rank for all $n\le N$ and $j=2,\ldots,m$. Then for any $n\le N-\gld(\k)$ we have the isomorphism
\begin{equation}\label{eq:coh:12}
\bigoplus_{i_1+\cdots+i_m=n}H^{i_1}(X_1,\k)\otimes\cdots\otimes H^{i_m}(X_m,\k)\stackrel\sim\to H^n(X_1\times\cdots\times X_m,\k)
\end{equation}
that is given by the cross product $a_1\otimes\cdots\otimes a_m\mapsto p_1^*(a_1)\cup\cdots\cup p_m^*(a_m)$.
%where $p_i:X_1\times\cdots\times X_m\to X_i$ is the projection to the $i$th coordinate.

This isomorphism and Proposition~\ref{proposition:3} prove that $H^n((E^N)^m,\k)=0$
for any $0<n\le 2(N-\n)$ and natural number $m$. Moreover, the space $E^m$ is contractible.
Hence we get the following result similar to Proposition~\ref{lemma:tw5:2}.

\begin{proposition}\label{lemma:tw5:2'} Suppose that for each $i=1,\ldots,m$, 
there is a left $L^{(i)}$-space $X^{(i)}$ for a closed subgroup $L^{(i)}$ of $G$.
%Let $L^{(1)},\ldots,L^{(m)}$ be closed subgroups of $G$ and $X^{(1)},\ldots,X^{(m)}$ be left $L$-spaces. 
Then the restriction map
$$
H^n((X^{(1)}\,{}_{L^{(1)}}\!\times E)\times\cdots\times(X^{(m)}\,{}_{L^{(m)}}\!\times E),\k)\to H^n((X^{(1)}\,{}_{L^{(1)}}\!\times E^N)\times\cdots\times(X^{(m)}\,{}_{L^{(m)}}\!\times E^N),\k)
$$
is an isomorphism for $n\le 2(N-\n)$.
\end{proposition}

\subsection{An auxiliary lemma} We are going to use the following result proved in~\cite[Lemma~18]{bt}.

\begin{proposition}\label{proposition:bt}
Let $\A$ be a locally compact Hausdorff space, $p:\A\to \B$ be a fibre bundle with fibre $\C$,
$\b\in \B$ be a point and $\k$ be a commutative ring.
Suppose that $\B$ is compact, Hausdorff, connected and simply connected,
all $H_c^n(\C,\k)$ are free of finite rank,
$H_c^n(\C,\k)=0$ for odd $n$ and $H_c^n(\B,\k)=0$ for odd $n\le N$.
\medskip

Then the following is true:

\medskip

\begin{enumerate}
\item\label{lemma:even_bundles:p:1} The restriction map
$H_c^n(\A,\k)\to H_c^n(p^{-1}(\b),\k)$ is surjective for all
%$n\le N-1$.\\[-3pt]
$n<N$.\\[-3pt]
\item\label{lemma:even_bundles:p:2} $H_c^n(\A-p^{-1}(\b),\k)=0$ for odd
%$n\le N-1$.\\[-3pt]
$n<N$.\\[-3pt]
\item\label{lemma:even_bundles:p:3} If $H_c^n(\B,\k)$ are free of finite rank for $n\le N$, then
the $\k$-modules $H_c^n(\A-p^{-1}(\b),\k)$ are also free of finite rank for
%$n\le N-1$.
$n<N$.
\end{enumerate}
\end{proposition}

\subsection{Restrictions on groups and spaces}
%In the rest of the paper,
%addtionally to conditions~\ref{restr:-1} and~\ref{restr:0},
%We suppose that the following conditions hold:
We consider the following additional conditions:

\medskip

{
\renewcommand{\labelenumi}{{\rm \theenumi}}
\renewcommand{\theenumi}{{\rm(\Roman{enumi})}}

\begin{enumerate}\setcounter{enumi}{2}
\itemsep=6pt
%\item the group $P$ is connected;
\item\label{restr:1} $P$ and $L$ are connected;
\item\label{restr:2} $H^n_P(\pt,\k)$ and $H^n_L(\pt,\k)$ are free of finite rank and vanish for odd $n$;
\item\label{restr:3} $X$ and $Y$ are compact and Hausdorff;
\item\label{restr:4} the cohomologies of $X/R$, $P/Q$, and $Y$ with coefficients in $\k$ are free of finite rank
      and vanish in odd degrees;
\item\label{restr:5} $X/R$ and $P/Q$ are simply connected;
\item\label{restr:6} the quotient map $X\to X/R$ is a principal $R$-bundle.
\end{enumerate}}

\medskip

\noindent
By Proposition~\ref{proposition:pb:1} and~\ref{restr:6}, these conditions imply that the natural
projections\footnote{These projection are given by $[x:p:y]\mapsto xR$ and $[p:y]\mapsto pQ$.
To apply Proposition~\ref{proposition:pb:1}, we can consider the isomorphisms
$X\mathop{\times}\limits_R P\mathop{\times}\limits_Q Y\cong(P\mathop{\times}\limits_Q Y)\,{}_R\!\!\times X$ and
$P\mathop{\times}\limits_Q Y\cong Y\,{}_Q\!\!\times P$ given by $[x:p:y]\mapsto R([p:y],x)$ and $[p:y]\mapsto Q(y,p^{-1})$,
respectively.
Here we consider $X$ as a left $R$-space under the following action $rx=xr^{-1}$. We also have the isomorphisms
$R\lq X\cong X/R$ and $Q\lq P\cong P/Q$ given by $Rx\mapsto xR$ and $Qp\mapsto p^{-1}Q$, respectively.
Our projections become canonical in the sense of Section~\ref{Group_actions_and_quotients} if we apply the identifications described just above.}
$$
X\mathop{\times}\limits_R P\mathop{\times}\limits_Q Y\to X/R,\qquad P\mathop{\times}\limits_Q Y\to P/Q
$$
are fibre bundles with fibres $P\mathop{\times}\limits_Q Y$ and $Y$, respectively.
By~\ref{restr:0} and~\ref{restr:3},
%the compactness condition,
we get that all four spaces in the left- and right-hand sides
of the above formulas are compact.

Applying successively the Leray spectral sequence for the cohomologies with compact support to these bundles
and using~\ref{restr:4} and~\ref{restr:5}, we get
\begin{equation}\label{eq:tw5:1}
%\begin{array}{c}
%\hspace{-150pt}H^\bullet\(X\mathop{\times}\limits_R P\mathop{\times}\limits_Q Y,\k\)\cong
%H^\bullet(X/R,\k)\otimes_\k H^\bullet\(P\mathop{\times}\limits_Q Y,\k\)\\[15pt]
%\hspace{200pt}\cong
%H^\bullet(X/R,\k)\otimes_\k H^\bullet(P/Q,\k)\otimes_\k H^\bullet(Y,\k).
%\end{array}
\begin{array}{c}
\hspace{-150pt}
H^\bullet\(X\mathop{\times}\limits_R P\mathop{\times}\limits_Q Y,\k\)\cong
H^\bullet(X/R,\k)\otimes_\k H^\bullet\(P\mathop{\times}\limits_Q Y,\k\)\\[15pt]
\hspace{200pt}\cong
H^\bullet(X/R,\k)\otimes_\k H^\bullet(P/Q,\k)\otimes_\k H^\bullet(Y,\k).
\end{array}
\end{equation}

\begin{lemma}\label{lemma:tw5:3}
The space $(X/R)\,{}_L\!\!\times\,E^N$ is simply connected.
\end{lemma}
\begin{proof} By Proposition~\ref{proposition:pb:1}, $(X/R)\,{}_L\!\!\times\,E^N\to L\lq E^N$
is a fibre bundle with fibre $X/R$.
Therefore, the result follows from~\ref{restr:5} and the long exact sequence of homotopy groups
$$
\{1\}=\pi_1(X/R)\to \pi_1((X/R)\,{}_L\!\!\times\,E^N)\to \pi_1(L\lq E^N)=\{1\}.
$$
The last equality follows from Corollary~\ref{corollary:sc}.
\end{proof}

\subsection{Surjectivity} We will use here the abbreviations
$$
Z=X\mathop{\times}\limits_R P\mathop{\times}\limits_Q Y,\quad C^N=((X/R)\,{}_L\!\!\times\,E^N)\times(Y\,{}_Q\!\!\times\,E^N)-\im\phi^N.
$$

\begin{lemma}\label{lemma:tw5:6}
$H^n((X/R)\;{}_L\!\times E^N,\k)=H^n(Z\;{}_L\!\times E^N,\k)=0$ for odd $n\le 2(N-\n)$.
\end{lemma}
\begin{proof} Due to similarity, we will prove only the last equality.
Let us consider the Leray spectral sequence for cohomologies with compact support for the canonical
projection $Z\;{}_L\!\times E^N\to L\lq E^N$. As this map is a fibre bundle with fibre $Z$,
it has the following second page
$$
E^{p,q}_2=H^p(L\lq E^N,\k)\otimes_\k H^q(Z,\k)
$$
in view of Corollary~\ref{corollary:sc}.
It also follows from~\ref{restr:2} and Proposition~\ref{lemma:tw5:2}
and also from~(\ref{eq:tw5:1}) and~\ref{restr:4} that $E_r^{p,q}=0$ for any finite $r\ge2$ if $p+q\le 2(N-\n)$
and either of $p$ or $q$ is odd. Thus this fact is also true for $r=\infty$.
%Therefore, the differentials coming to and starting from
%$E^{p,q}_r$ are zero if $r\ge2$ and $p+q+1\le 2(N-\n)$.
%It follows that $E^{p,q}_\infty=E^{p,q}_2$ if $p+q+1\le2(N-\n)$.
\end{proof}

We are going to study the cohomologies with compact support of the spaces $C^N$.
First, we study the cohomology of the fibre $F^N_{\bar e}$,
see Lemma~\ref{lemma:3} and Remark~\ref{remark:2.5}.

\begin{lemma}\label{lemma:tw5:4}
For any $n<2(N-\n)$ and $\bar e\in P\lq E^N$, the cohomology $H_c^n(F^N_{\bar e},\k)$
is free of finite rank and vanish if $n$ is odd.
\end{lemma}
\begin{proof}
It suffices to apply parts~\ref{lemma:even_bundles:p:2} and~\ref{lemma:even_bundles:p:3} of
Proposition~\ref{proposition:bt} to the fibre bundle
$\pi_Q\can:Y\,{}_Q\!\times\,E^N\to P\lq E^N$ with fibre $Y\,{}_Q\times P$,
see Lemma~\ref{lemma:4}.
Most of the hypotheses of the former lemma either can be checked routinely.
The fact that $H^n(Y\,{}_Q\times P)$ are free of finite rank and vanish in odd degrees follows from~\ref{restr:4}
and the Leray spectral sequence applied to the fibre bundle $Y\,{}_Q\times P\to Q\lq P$.
Finally, for $n\le 2(N-\n)$, we get $H^n(P\lq E^N)=H_P^n(\pt,\k)$. The last module is free of finite rank
and vanish for odd $n$ by~\ref{restr:2}.
\end{proof}

\begin{lemma}\label{lemma:tw5:5}
$H_c^n(C^N,\k)=0$ for odd $n<2(N-\n)$.
\end{lemma}
\begin{proof}
It follows from Lemmas~\ref{lemma:tw5:3} and~\ref{lemma:tw5:4} that the second page of the Leray spectral
sequence for cohomologies with compact support for the fibre bundle as in Lemma~\ref{lemma:3}
is equal to
$$
E_2^{p,q}=H^p((X/R)\,{}_L\!\!\times\,E^N,\k)\otimes_\k H_c^q(F^N_{\bar e},\k)
$$
if $q<2(N-\n)$. It follows from Lemmas~\ref{lemma:tw5:6} and~\ref{lemma:tw5:4} that $E_r^{p,q}=0$
for any finite $r\ge2$ if $p+q<2(N-\n)$ and either of $p$ or $q$ is odd.
%Therefore, the differentials coming to and starting from
%$E^{p,q}_r$ are zero if $r\ge2$ and $p+q+1<2(N-\n)$.
%It follows that $E^{p,q}_\infty=E^{p,q}_2$ if $p+q+1<2(N-\n)$.
Thus this fact is also true for $r=\infty$.
\end{proof}

\begin{corollary}\label{corollary:tw5:1}
For any $n<2(N-\n)-1$, the map
$$
\begin{tikzcd}
H^n\(((X/R)\,{}_L\!\!\times\,E^N)\times(Y\,{}_Q\!\!\times\,E^N),\k\)
\arrow{r}{(\phi^N)^*}& H^n\(Z\;{}_L\!\times E^N,\k\)
\end{tikzcd}
$$
is surjective.
\end{corollary}
\begin{proof}
%The result follows from the exact sequence
%\begin{multline*}
%H_c^n\(((X/R)\,{}_L\!\!\times\,E^N)\times(Y\,{}_Q\!\!\times\,E^N)-\im\phi^N,\k\)\to
%H^n\(((X/R)\,{}_L\!\!\times\,E^N)\times(Y\,{}_Q\!\!\times\,E^N),\k\)\\
%\to H^n\((X\mathop{\times}\limits_R P\mathop{\times}\limits_Q Y)\;{}_L\!\times E^N,\k\)
%\end{multline*}
By Theorem~\ref{theorem:2}, we get the exact sequence
$$
\begin{tikzcd}
H^n\(((X/R)\,{}_L\!\!\times\,E^N)\times(Y\,{}_Q\!\!\times\,E^N),\k\)
\arrow{r}{(\phi^N)^*}& H^n\(Z\;{}_L\!\times E^N,\k\)\arrow{r}&H_c^{n+1}\(C^N,\k\).
\end{tikzcd}
$$
Now the result follows from Lemmas~\ref{lemma:tw5:6} and~\ref{lemma:tw5:5}.
\end{proof}

\begin{corollary}
The map $\theta$ is surjective.
\end{corollary}
\begin{proof} Under our assumption~\ref{restr:4}, the left map (cross product) in~(\ref{eq:9}) is an isomorphism
by the K\"unneth formula. Therefore, it suffices to prove that $\phi^*$ is surjective.
In each degree, this facts follows from Corollary~\ref{corollary:tw5:1}
and Proposition~\ref{lemma:tw5:2'} for $N$ big enough.
\end{proof}

\begin{lemma}\label{lemma:7} In each degree, the homomorphism $\theta$ induces a map between free $\k$-modules
of the same finite rank.
\end{lemma}
\begin{proof} The result follows from the following computation:
\begin{multline*}
H^\bullet_L(X/R,\k)\otimes_{H_P^\bullet(\pt,\k)}H^\bullet_Q(Y,\k)\cong
H_L^\bullet(X/R,\k)\otimes_{H_P^\bullet(\pt,\k)}H_Q^\bullet(\pt,\k)\otimes_\k H^\bullet(Y,\k)\\[6pt]
\cong
H_L^\bullet(X/R,\k)\otimes_{H_P^\bullet(\pt,\k)}H^\bullet_P(\pt,\k)\otimes_\k H^\bullet(P/Q,\k)\otimes_\k H^\bullet(Y,\k)\\[6pt]
\cong H_L^\bullet(X/R,\k)\otimes_\k H^\bullet(P/Q,\k)\otimes_\k H^\bullet(Y,\k)\\[6pt]
\cong H_L^\bullet(\pt,\k)\otimes_\k H^\bullet(X/R,\k)\otimes_\k H^\bullet(P/Q,\k)\otimes_\k H^\bullet(Y,\k)\\
\cong H_L^\bullet(\pt,\k)\otimes_\k H^\bullet\(X\mathop{\times}\limits_R P\mathop{\times}\limits_Q Y,\k\)\cong
H_L^\bullet\(X\mathop{\times}\limits_R P\mathop{\times}\limits_Q Y,\k\).
\end{multline*}

%Warning!
%$$
%H_L^\bullet(X/R,\k)\otimes_{H_P^\bullet(\pt,\k)}H^\bullet_P(\pt,\k)\cong H_L^\bullet(X/R,\k)
%$$
%as $\k$-$\k$-bimodules, as $R\lq E$ and $P\lq E$ are connected as is $E$.
\end{proof}

\begin{theorem}\label{theorem:1}
Suppose that
%$G$ is a compact Lie group, $L,P,Q,R$ are its closed subgroups, and
conditions~\ref{restr:-1}--\ref{restr:6}
are satisfied. Then $\theta$ is an isomorphism of rings and left $H^\bullet_L(\pt,\k)$-modules.
If $M$ is another closed subgroup of $G$ and $B:Y\to G/M$
is a $Q$-equivariant continuous map, then $\theta$ is an isomorphism of rings and
$H^\bullet_L(\pt,\k)$-$H^\bullet_M(\pt,\k)$-bimodules.
\end{theorem}
\begin{proof}
Composing $\theta$ with the isomorphism of Lemma~\ref{lemma:7}, we get a surjective
homomorphism from a finitely generated $\k$-module to itself. Now it suffices
to apply~\cite[Theorem~3.6]{Rotman}, to prove that $\theta$ is bijective.
The remaining claims follow from Lemma~\ref{lemma:6}.
\end{proof}

\section{Basic examples}\label{Basic_examples}

\subsection{Equivariant cohomology of the flag variety}\label{Equivariant_cohomology_of_the_flag_variety} Suppose that the order $|W|$ of the Weyl group
is invertible in $\k$. %Then by ..., we get that
The map $\pi_K:K\lq E\to G\lq E$ given by $\pi_K(Ke)=Ge$
induces the isomorphism
\begin{equation}\label{eq:tw5:4}
\pi_K^*:H^\bullet_G(\pt,\k)\ito H_K^\bullet(\pt,\k)^W,
\end{equation}
where the right-hand side is the subring of $W$-invariants (see, the proof of~\cite[Proposition~1]{Brion}).
%In particular, the $H^n_G(pt)=0$ for odd $n$.
%This isomorphism guaranties~\ref{restr:2}.
%Note that this map is actually induced by the map.

We are going to apply the results of the previous sections to the following data
$X=G$, $Y=G/K$, $P=R=Q=G$, $L=M=K$, and $\alpha:X\to G$ and $B:Y\to G/M$ be the identity maps on $G$ and $G/K$, respectively.
%Then we get the map
%$$
%(G/K)\,{}_K\!\!\times\,E=\(G\mathop{\times}\limits_G G/K\)\,{}_K\!\!\times\,E\stackrel{\phi}\to((G/G)\,{}_K\!\!\times\,E)\times(G/K\,{}_G\!\!\times\,E)\cong (K\lq E)^2
%$$
%(see~\cite[Section 1.5]{Jantzen} for the last homeomorphism).
The induced map $A:X/R\to G/R$ is then the identity map on $G/G$.
By Theorem~\ref{theorem:1}, condition~\ref{restr:2} being guaranteed by~(\ref{eq:tw5:4}), the homomorphism $\theta$ is an isomorphism and we get %itw takes the form
\begin{equation}\label{eq:tw5:3}
H^\bullet_K\(G\mathop{\times}\limits_GG\mathop{\times}\limits_G G/K,\k\)\cong H_K^\bullet(G/G,\k)\otimes_{H^\bullet_G(\pt,\k)}H_G^\bullet(G/K,\k).
\end{equation}

Let us look at the left-hand side. We have the following commutative diagram\footnote{Recall that $\alpha B:G\mathop{\times}\limits_GG\mathop{\times}\limits_G G/K\to G/K$ is the map $[g_1:g_2:g_3\rr\mapsto g_1g_2g_3K$ according to~\ref{eq:tw:7}.}:
$$
%\begin{tikzcd}[column sep=50pt]
%\(G\mathop{\times}\limits_GG\mathop{\times}\limits_G G/K\)\,{}_K\!\!\times\,E\arrow{r}{\can}\arrow{d}[swap]{\wr}& K\lq E\\
%G/K\,{}_K\!\!\times\,E\arrow{ur}[swap]{\can}
%\end{tikzcd}
%\qquad
\begin{tikzcd}
\(G\mathop{\times}\limits_GG\mathop{\times}\limits_G G/K\)\,{}_K\!\!\times\,E\arrow{r}{v_{\alpha B}}\arrow{d}[swap]{\wr}& K\lq E\\
G/K\,{}_K\!\!\times\,E\arrow{ur}[swap]{v_B}
\end{tikzcd}
$$
where the vertical isomorphism is given by $K([g_1:g_2:g_3\rr,e)\mapsto K(g_1g_2g_3K,e)$.
Hence by Lemma~\ref{lemma:pull-back}, we get the isomorphism of $H^\bullet_K(\pt,\k)$-$H^\bullet_K(\pt,\k)$-bimodules
$$
H^\bullet_K\(G\mathop{\times}\limits_GG\mathop{\times}\limits_G G/K,\k\)\cong H^\bullet_K(G/K,\k)
$$
where both left actions are canonical and the right actions on the left-hand side and the right-hand side
are the actions twisted by $\alpha B$ and $B$, respectively.

Let us look at the first factor of the tensor product in the right-hand side of~(\ref{eq:tw5:3}).
Considering the commutative diagram
$$
\begin{tikzcd}
G/G\,{}_K\!\!\times\,E\arrow{r}{v_A}\arrow{d}[swap]{\can}{\wr}& G\lq E\\
K\lq E\arrow{ur}[swap]{\pi_K}
\end{tikzcd}
$$
we get by Lemma~\ref{lemma:pull-back} the isomorphism of $H_K^\bullet(\pt,\k)$-$H_G^\bullet(\pt,\k)$-bimodules
$$
H_K^\bullet(G/G,\k)\cong H_K^\bullet(\pt,\k),
$$
where the left actions are canonical, the right action on the left-hand side is the action twisted by $A$,
and the right action on the right-hand side is as follows:
$$
mb=m\cup\pi_K^*(b)
$$
for $m\in H_K^\bullet(\pt,\k)$ and $b\in H_G^\bullet(\pt,\k)$.

Finally, let us look at the second factor of the tensor product in the right-hand side of~(\ref{eq:tw5:3}).
It is easy to note that the map $v_B:(G/K)\,{}_G\!\!\times\,E\to K\lq E$ is a homeomorphism.
This fact was already mentioned in~\cite[(1.5)]{Jantzen}.
Considering the commutative diagram
$$
\begin{tikzcd}[column sep=50pt]
G/K\,{}_G\!\!\times\,E\arrow{r}{\can}\arrow{d}{\wr}[swap]{v_B}& G\lq E\\
K\lq E\arrow{ur}[swap]{\pi_K}
\end{tikzcd}
%\qquad\qquad
%\begin{tikzcd}[column sep=50pt]
%(G/K)\,{}_G\!\!\times\,E\arrow{r}{v_B}\arrow{d}[swap]{\wr}& K\lq E\\
%K\lq E\arrow[equal]{ur}
%\end{tikzcd}
$$
%where the vertical isomorphism is given by $G(cK,e)\mapsto Kc^{-1}e$.
%Hence
we get the isomorphism $(v_B^*)^{-1}$ of $H_G^\bullet(\pt,\k)$-$H_K^\bullet(\pt,\k)$-bimodules
$$
H_G^\bullet(G/K,\k)\cong H_K^\bullet(\pt,\k),
$$
where the left and the right actions on the left-hand side are the canonical action and
the action twisted by $B$, respectively, the right action on the right-hand side
is canonical and the left action on the right-hand side is as follows
$$
am=\pi_K^*(a)\cup m
$$
for any $a\in H_G^\bullet(\pt,\k)$ and $m\in H_K^\bullet(\pt,\k)$.

Applying all these isomorphisms, we obtain from~(\ref{eq:tw5:3}), the isomorphism of
$H_K^\bullet(\pt,\k)$-$H_K^\bullet(\pt,\k)$-bimodules (and rings)
$$
H^\bullet_K\(G/K,\k\)\cong H_K^\bullet(\pt,\k)\otimes_{H^\bullet_G(\pt,\k)}H_K^\bullet(\pt,\k)
$$
under the left and right actions described above.
Applying~(\ref{eq:tw5:4}), we get the isomorphism of $H_K^\bullet(\pt,\k)$-$H_K^\bullet(\pt,\k)$-bimodules
$$
H^\bullet_K\(G/K,\k\)\cong H_K^\bullet(\pt,\k)\otimes_{H^\bullet_K(\pt,\k)^W}H_K^\bullet(\pt,\k),
$$
where $H_K^\bullet(\pt,\k)$ and $H^\bullet_K(\pt,\k)^W$ act
%on both factors
canonically.

\subsection{Standard bimodules}\label{Standard_bimodules} Let us denote
\begin{equation}\label{eq:12}
\R=H^\bullet_K(\pt,\k).
\end{equation}
This is the $\R$-$\R$-bimodule with respect to the cup product. %, which we omit for the simplicity of notation.
For any $w\in W$, we denote by $\R_w$ the $\R$-$\R$-bimodule equal to $\R$ as an abelian group
whose left action $\cdot$ is untwisted and right action $\cdot_w$ is twisted by $w$:
$$
r'\cdot r=r'r,\quad r\cdot_w r''=rw(r'').
$$
%Here $\cdot$ denotes the left (untwisted) action.

The construction of this bimodule naturally arises as a special case of the twisted action described in Section~\ref{Twisted action}.
Indeed, let $A_w:\pt\to  G/K$ be the map with the value $wK$. Then we get
the map $v_{A_w}:K\lq E\cong\pt\,{}_K\!\!\times\,E\to K\lq E$, which coincides with the map $K\lq \rho_{w^{-1}}$
described in Section~\ref{Semisimple_groups}. By~(\ref{eq:tw5:5}), we get that the right action on $H_K^\bullet(\pt,\k)$ on itself twisted by $A_w$
coincodes with $\cdot_w$.

Now let us compute the tensor product of these bimodules with the help of the isomorphism~$\theta$.
Let $X=K$, $Y=\pt$, $L=R=P=Q=N=K$ and $\alpha:X\to G$ be the map defined by $\alpha(k)=\dot{w}k$.
The last map becomes $K$-$K$-equivarinat if we consider the right action of $K$ on $X$ by multiplication
and define the left action by $k'\ast k=\dot{w}^{-1}k'\dot{w}k$. Using this map $\alpha$,
we define the map $A:\pt\cong K/K\to G/K$ by~(\ref{eq:tw5:6}). Obviously, $A=A_w$.

Let us choose any $w'\in W$ and set $B=A_{w'}$. By~(\ref{eq:tw:7}), we get
$$
\alpha B([k:1:\pt])=\dot wk\dot w'K=\dot w\dot w'(\dot w'^{-1}k\dot w')K=\dot w\dot w'K=ww'K.
$$
Hence $\alpha B=A_{ww'}$ under the identification $K\mathop{\times}\limits_KK\mathop{\times}\limits_K\pt\cong\pt$.
By Theorem~\ref{theorem:1}, the homeomorphism $\theta$ is an isomorphism and it
reads as the classical $\R$-$\R$-bimodule isomorphism
$$
\R_w\otimes_\R\R_{w'}\cong\R_{ww'}.
$$

\section{Bott-Samelson varieties}\label{Bott-Samelson varieties}

\subsection{Computation of the equivariant cohomology}\label{Computation_of_the_equivariant_cohomology}
Let $G$ be a semisimple compact Lie group. We will use
%Let us assume the notation of Section~\ref{Basic_examples} including
abbreviation~(\ref{eq:12}).
Additionally, we assume that $2$
is invertible in $\k$. For any reflection $s$, we have two rings
$$
\R(s)=H^\bullet_{G_s}(\pt,\k),\quad \R^s=\{r\in\R\suchthat s(r)=r\}.
$$

\begin{lemma}\label{lemma:8} For any reflection $s$, the quotient map $\pi_s:K\lq E\to G_s\lq E$
induces the isomorphism $\pi_s^*:\R(s)\ito \R^s$.
\end{lemma}
\begin{proof}
The result follows from the proof of%
%~\cite[Chapter III, Proposition~1]{Hsiang} (see also
~\cite[Proposition~1]{Brion} applied to the group $G_s$, as its Weyl group consists of $2$ elements and $2$ is invertible in~$\k$.
\end{proof}

For a sequence of reflections $s=(s_1,\ldots,s_n)$, using the notation of Section~\ref{More quotient products},
we define the following space
$$
\BS(s)=G_{s_1}\mathop{\times}\limits_K{G_{s_2}}\mathop{\times}\limits_K\cdots\mathop{\times}\limits_KG_{s_n}/K
$$
called the {\it Bott-Samelson variety} for sequence $s$.
The sequence $s$ can be empty, in which case $\BS(s)$ is just the singleton.
In what follows, we use the natural identification $\pt\,{}_K\!\times E\cong K\lq E$ given by $K(pt,e)=Ke$.

The torus $K$ acts on the left on $\BS(s)$ (via the first factor for nonempty $s$).
Therefore, we can consider the equivariant cohomology denoted as follows:
$$
H(s)=H_K^\bullet(\BS(s),\k).
$$
It is a left $\R$-module with respect to the canonical action. On the other hand, there is the $K$-equivariant map
\begin{equation}\label{eq:30}
A_s:\BS(s)\to G/K
\end{equation}
given by $[g_1:\cdots:g_{n-1}:g_n\rr\mapsto g_1\cdots g_{n-1}g_nK$.
Therefore, $H(s)$ is also a right $\R$-module
with the action twisted by $A_s$ according to Section~\ref{Twisted action}.

We represent $H(s)$ graphically by

\begin{center}
\begin{tikzpicture}
\draw[dashed] (0,0) -- (3.5,0);
\draw[red,fill=red] (0.5,0) circle(0.04) node[anchor=south]{\footnotesize$s_1$};
\draw[blue,fill=blue] (1.25,0) circle(0.04) node[anchor=south]{\footnotesize$s_2$};
\draw[green,fill=green] (3,0) circle(0.04) node[anchor=south]{\footnotesize$s_n$};
\draw (2.15,0) node[anchor=south]{\footnotesize$\cdots$};
\end{tikzpicture}
\end{center}

\smallskip

\noindent
where each reflection is marked by a separate color. % (chosen arbitrarily).

We consider the following map $\phi_s:\BS(s)\,{}_K\!\!\times E\to(K\lq E)^{n+1}$:
$$
K([g_1:\cdots:g_{n-1}:g_n\rr,e)\mapsto(Ke,Kg_1^{-1}e,\ldots,K(g_1\cdots g_n)^{-1}e).
$$
%where $n$ is the length of $s$.
This map can be constructed inductively with the help of the embedding $\phi$ constructed
in Section~\ref{The_embedding} as follows.
If $s=\emptyset$, then $\phi_s$ is the identity map. %natural
%homeomorphism
%$$
%\nu:\pt\,{}_K\!\times E\ito K\lq E
%$$
%given by $K(pt,e)\mapsto Ke$.

Now suppose that $s$ is not empty. %Let us consider the truncated sequence $s'=(s_1,\ldots,s_{n-1})$.
%Let also
%$$
%\iota_s:\BS(s)\ito G_{s_1}\mathop{\times}\limits_K{G_{s_2}}\mathop{\times}\limits_K\cdots\mathop{\times}\limits_KG_{s_n}\mathop{\times}\limits_K\pt
%$$
%be the homeomorphism given by $[g_1:\cdots:g_{n-1}:g_nK]\mapsto [g_1:\cdots:g_{n-1}:g_n:\pt]$.
Then we consider the following diagram:
\begin{equation}\label{eq:10}
\begin{tikzcd}
\BS(s)\,{}_K\!\times E\arrow{d}[swap]{\phi_s}\arrow[equal]{r}& \(G_{s_1}\mathop{\times}\limits_K{G_{s_2}}\mathop{\times}\limits_K\cdots\mathop{\times}\limits_KG_{s_{n-1}}\mathop{\times}\limits_KG_{s_n}\mathop{\times}\limits_K\pt\){}\,{}_K\!\times E\arrow{d}{\phi}\\
%(K\lq E)^{n+1}&\(\(G_{s_1}\mathop{\times}\limits_K{G_{s_2}}\mathop{\times}\limits_K\cdots\mathop{\times}\limits_KG_{s_{n-1}}/K\)\,{}_K\!\times E\)\times(\pt\,{}_K\!\times E)\arrow{l}[swap]{\phi_{s'}\times\id}
(K\lq E)^{n+1}&\(\BS(s')\,{}_K\!\times E\)\times(\pt\,{}_K\!\times E)\arrow{l}[swap]{\phi_{s'}\times\id}
\end{tikzcd}
\end{equation}
Here $s'=(s_1,\ldots,s_{n-1})$ is the truncated sequence and the map $\phi$
is the map defined in Section~\ref{The_embedding} for the following
set of data%
%\footnote{Here the case $n=1$ deserves a separate treatment: we should define $X=K$ and $\iota_s$ should be the homeomorphism from $\BS(s)$
%to $K\mathop{\times}\limits_KG_{s_1}\mathop{\times}\limits_K\pt$ given by $[gK]\mapsto[1:g:\pt]$. We leave the corresponding
%changes in the diagrams to the reader.}
:
$$
X=G_{s_1}\mathop{\times}\limits_K{G_{s_2}}\mathop{\times}\limits_K\cdots\mathop{\times}\limits_KG_{s_{n-1}},\quad Y=\pt,\quad P=G_{s_n},\quad L=R=Q=K
$$
and $\alpha:X\to G$ given by $[g_1:\cdots:g_{n-1}]\mapsto g_1\cdots g_{n-1}$.
We claim that this diagram is commutative. Indeed, we get
$$
\hspace{-110pt}
\begin{tikzcd}
K([g_1:g_2:\cdots:g_n\rr,e)=K([\underbrace{g_1:g_2:\cdots:g_{n-1}}_x:\underbrace{g_n}_p:\underbrace{pt}_y],e)\arrow[mapsto]{r}{\phi}&[-10pt]{}
\end{tikzcd}
$$

\vspace{-8pt}

$$\hspace{-167pt}
\Big(K([g_1:g_2:\cdots:g_{n-1}\rr,e),K(pt,g_n^{-1}(g_1\cdots g_{n-1})^{-1}e)\Big)=
$$
$$
\begin{tikzcd}
\Big(K([g_1:g_2:\cdots:g_{n-1}\rr,e),K(g_1\cdots g_{n-1}g_n)^{-1}e\Big)\arrow[mapsto]{r}{\phi_{s'}\times\id}&[8pt](Ke,Kg_1^{-1}e,\ldots,K(g_1\cdots g_n)^{-1}e)
\end{tikzcd}
$$

\vspace{-8pt}

$$\hspace{300pt}
=\phi_s\Big(K([g_1:g_2:\cdots:g_n\rr,e)\Big).
$$
Here we underbraced the variables $x$, $p$, and $y$ as in~(\ref{eq:5}).

\begin{lemma}
The map $\phi_s$ is a topological embedding.
\end{lemma}
\begin{proof}
The result follows inductively from diagram~(\ref{eq:10}) and Theorem~\ref{theorem:2}.
\end{proof}

Using this embedding, we can define the map $\Theta_s$ as the following composition:
$$
\begin{tikzcd}
\R^{\otimes_\k n+1}\arrow{r}{\times}&H^\bullet((K\lq E)^{n+1},\k)\arrow{r}{\phi_s^*}&H(s),
\end{tikzcd}
$$
where the first map is the cross product, see Section~\ref{The_tensor_product}.
We consider the following commutative diagram:
\begin{equation}\label{eq:11}
\begin{tikzcd}
\R^{\otimes_\k n+1}\arrow{d}[swap]{\Theta_s}\arrow{r}{\Theta_{s'}\otimes\id}& H(s')\otimes_\k \R\arrow{d}{\Theta}\\
H(s)&H^\bullet_K\(G_{s_1}\mathop{\times}\limits_K\cdots\mathop{\times}\limits_K G_{s_n}\mathop{\times}\limits_K\pt,\k\)\arrow[equal]{l}
\end{tikzcd}
\end{equation}
where $\Theta$ is map~(\ref{eq:9}) corresponding to $\phi$. %for the above choice of the spaces, the groups, and the map $\alpha$.
This diagram is commutative as diagram~(\ref{eq:10}) is so. Indeed, we get
$$\hspace{-100pt}
\begin{tikzcd}
a_1\otimes\cdots\otimes a_{n+1}\arrow[mapsto]{r}{\Theta_{s'}\otimes\id}&[10pt]\phi_{s'}^*(\pr_1^*(a_1)\cup\cdots\cup\pr_{n+1}^*(a_1))\otimes a_{n+1}\arrow[mapsto]{r}{\Theta}&{}
\end{tikzcd}
$$
$$\hspace{200pt}
\phi^*\big(\pr_1^*\phi_{s'}^*(\pr_1^*(a_1)\cup\cdots\cup\pr_{n+1}^*(a_1))\cup\pr_2^*(a_{n+1})\big),
$$
$$\hspace{-175pt}
\begin{tikzcd}
a_1\otimes\cdots\otimes a_{n+1}\arrow[mapsto]{r}{\Theta_s}&[-4pt]\phi_s^*(\pr_1^*(a_1)\cup\cdots\cup\pr_{n+1}^*(a_{n+1})).
\end{tikzcd}
$$
Hence we need to prove that
$$
%\phi^*\pr_1^*\phi_{s'}^*\pr_i^* \phi_s^*\pr_i^*
\pr_i\phi_{s'}\pr_1\phi=\pr_i\phi_s\text{ for }1\le i\le n,\quad
%\phi^*\pr_2^*
\pr_2\phi=\pr_{n+1}\phi_s.
$$
Both equalities follow if we apply $\pr_i$ to diagram~(\ref{eq:10}).
As $\Theta_\emptyset=\id$, we can use this diagram to compute $\Theta_s$ inductively.
This computation and Theorem~\ref{lemma:6} prove that $\Theta$ is a homomorphism 
of rings and $\R$-$\R$-bimodules.

Let us consider now the tensor product:
%$$
%\R\otimes_{H_{G_{s_1}}^\bullet(\pt,\k)}\cdots\otimes_{H_{G_{s_{n-1}}}^\bullet(\pt,\k)}\R\otimes_\k \R
%$$
$$
\R^{\otimes s}=\R\otimes_{\R(s_1)}\cdots\otimes_{\R(s_n)}\R,
$$
where each $\R(s_i)$ %$=H_{G_{s_i}}^\bullet(\pt,\k)$
acts on the left and on the right on $\R$
through $\pi_{s_i}$, that is, $bm=\pi_{s_i}^*(b)\cup m$ and $mb=m\cup \pi_{s_i}^*(b)$.
for any $m\in \R$ and $b\in \R(s_i)$. %$=H_{G_{s_i}}^\bullet(\pt,\k)$.
%In this notation, we will omit brackets for sequences of length 1, that is
%$\R^{\otimes s}=\R^{\otimes(s)}=\R\otimes_{\R(s)}\R$ for any reflection $s$.
In view of Lemma~\ref{lemma:8}, we get
$$
\R^{\otimes s}=\R\otimes_{\R^{s_1}}\cdots\otimes_{\R^{s_n}}\R,
$$
We will use the first equality, when we apply the (iso)morphisms~$\theta$ introduced in Section~\ref{The_tensor_product}
and the second equality to perform algebraic computations.

\begin{theorem}\label{theorem:3}
The map $\Theta_s$ factors through the natural projection
$\R^{\otimes_\k n+1}\to \R^{\otimes s}$
to an isomorphism of rings and $\R$-$\R$-bimodules
$\theta_s:\R^{\otimes s}\ito H(s)$.
\end{theorem}
\begin{proof}
%$$
%\begin{tikzcd}
%\R^{\otimes s'}\otimes_\k \R\\
%\R^{\otimes_\k n+1}\arrow{u}\arrow{r}\arrow{d}&\R^{\otimes s}\arrow{r}{\theta_{s'}\otimes\id}\arrow[dashed]{dl}{\theta_s}&H_K^\bullet(\BS(s'),\k)\otimes_{H_{G_{s_n}}(\pt,\k)}\R\\
%H_K(\BS(s),\k)
%\end{tikzcd}
%$$
Let us apply the induction on the length of $s$.
%We will denote the resulting isomorphism by $\theta_s$.
There is nothing to prove if $s=\emptyset$,
as $\Theta_\emptyset=\id$. Now suppose that $s$ is not empty.
We consider the following diagram:
$$
\def\sp{50pt}
\begin{tikzcd}[column sep=-20pt]
H(s')\otimes_\k \R\arrow{dr}\arrow{rr}{\Theta}&[\sp]&H^\bullet_K\(G_{s_1}\mathop{\times}\limits_K\cdots\mathop{\times}\limits_KG_{s_n}\mathop{\times}\limits_K\pt,\k\)\arrow[equal]{dd}\\
{}&[\sp]H(s')\otimes_{\R(s_n)}\R\arrow{ur}{\theta}\\
\R^{\otimes s'}\otimes_\k \R\arrow{uu}{\theta_{s'}\otimes\id}\arrow{r}&[\sp]\R^{\otimes s}\arrow{u}{\theta_{s'}\otimes\id}\arrow[dashed]{r}{\theta_s}&H(s)\\
{}&[\sp]\R^{\otimes_\k n+1}\arrow{u}\arrow{ul}\arrow{ur}[swap]{\Theta_s}&\\
\end{tikzcd}
$$

\vspace{-20pt}

\noindent
where $\theta$ is the map corresponding to $\Theta$ from diagram~(\ref{eq:11}) as described in Section~\ref{The_tensor_product} and
the dashed arrow is defined so that the trapezoid containing it is commutative.
We need to prove that the triangle containing this arrow is also commutative.
This is however so, as all other triangles, the trapezoid as well as the outer perimeter
are commutative. The last claim follows from the commutativity of diagram~(\ref{eq:11}) and the inductive definition of $\theta_{s'}$.
Note that $\theta$ is an isomorphism by Theorem~\ref{theorem:1},
the validity of~\ref{restr:6} being guaranteed inductively by Lemma~\ref{lemma:pb:1}.
\end{proof}

\subsection{Concatenation}\label{Concatenation} Our next aim is to prove that
the isomorphisms $\theta_s$ of Theorem~\ref{theorem:3} behave well under the concatenation
of sequences. Let $s=(s_1,\ldots,n)$ and $t=(t_1,\ldots,t_m)$ be sequences of reflections.
As in Section~\ref{The_embedding}, we consider the embedding
$$
\phi_{s,t}:\BS(st)\,{}_K\!\!\times E\to(\BS(s)\,{}_K\!\!\times E)\times(\BS(t)\,{}_K\!\!\times E)
$$
for the following set of data:
%following case of Theorem~\ref{theorem:1}:
$$
X=G_{s_1}\mathop{\times}\limits_K{G_{s_2}}\mathop{\times}\limits_K\cdots\mathop{\times}\limits_KG_{s_n},\quad Y=\BS(t),\quad L=R=P=Q=K,
$$
and $\alpha:X\to G$ defined by $[g_1:\cdots:g_n]\mapsto g_1\cdots g_n$.
%Note that $X/R\cong\BS(s)$ and we assume that
Here we use the representation
$$
\BS(st)=X\mathop{\times}\limits_KK\mathop{\times}\limits_KY,
$$
thus assuming
\begin{equation}\label{eq:id:conc}
[g_1:\cdots:g_{n+m}\rr=[g_1:\cdots:g_n:1:g_{n+1}:\cdots:g_{n+m}\rr.
\end{equation}
Hence by Theorem~\ref{theorem:1}, we get the following isomorphisms of $\R$-$\R$-bimodules:
$$
\begin{tikzcd}
H(s)\otimes_{\R}H(t)\arrow{r}{\theta_{s,t}}[swap]{\sim}&H(st),
\end{tikzcd}
$$
where $\theta_{s,t}$ is induced by $\phi_{s,t}$ as in Section~\ref{The_tensor_product}.

\begin{theorem}\label{theorem:4}
There is the following commutative diagram:
$$
\begin{tikzcd}
H(s)\otimes_{\R}H(t)\arrow{r}{\theta_{s,t}}[swap]{\sim}&H(st)\\
\R^{\otimes s}\otimes_{\R}\R^{\otimes t}\arrow{u}{\theta_s\otimes\theta_t}\arrow{r}[swap]{\sim}&\R^{\otimes st}\arrow{u}{\theta_{st}}
\end{tikzcd}
$$
where the isomorphism of the bottom arrow is
%$
%\R^{\otimes s}\otimes_{\R}\R^{\otimes t}\ito \R^{\otimes st}
%$
given by
\begin{equation}\label{eq:13}
(a_1\otimes\cdots\otimes a_{n+1})\otimes(b_1\otimes\cdots\otimes b_{m+1})\mapsto a_1\otimes\cdots\otimes a_n \otimes a_{n+1}\cup b_1\otimes b_2\otimes\cdots\otimes b_{m+1}.
\end{equation}
\end{theorem}
\begin{proof}
Following the upper path, we get
$$\hspace{-210pt}
\begin{tikzcd}
(a_1\otimes\cdots\otimes a_{n+1})\otimes(b_1\otimes\cdots\otimes b_{m+1})\arrow[mapsto]{r}{\theta_s\otimes\theta_t}&{}
\end{tikzcd}
$$

\vspace{-15pt}

$$
\begin{tikzcd}
\phi_s^*(\pr_1^*(a_1)\cup\cdots\cup\pr_{n+1}^*(a_{n+1}))\otimes\phi_t^*(\pr_1^*(b_1)\cup\cdots\cup\pr_{m+1}^*(b_{m+1}))\arrow[mapsto]{r}{\theta_{s,t}}&{}
\end{tikzcd}
$$
$$\hspace{50pt}
\phi_{s,t}^*\big(\pr_1^*\phi_s^*(\pr_1^*(a_1)\cup\cdots\cup\pr_{n+1}^*(a_{n+1}))\cup\pr_2^*\phi_t^*(\pr_1^*(b_1)\cup\cdots\cup\pr_{m+1}^*(b_{m+1}))\big).
$$
Following the lower path, we get
$$%\hspace{-210pt}
\begin{tikzcd}
(a_1\otimes\cdots\otimes a_{n+1})\otimes(b_1\otimes\cdots\otimes b_{m+1})\arrow[mapsto]{r}&[-10pt]a_1\otimes\cdots\otimes a_n \otimes a_{n+1}\cup b_1\otimes b_2\otimes\cdots\otimes b_{m+1}
\end{tikzcd}
$$

\vspace{-18pt}

$$
\begin{tikzcd}
{}\arrow[mapsto]{r}&[-10pt]\phi_{st}^*\big(\pr_1^*(a_1)\cup\cdots\cup\pr_n^*(a_n) \cup\pr_{n+1}^*(a_{n+1}\cup b_1)\cup\pr_{n+2}^*(b_2)\cup\cdots\cup\pr_{n+m+1}^*(b_{m+1})\big).
\end{tikzcd}
$$

Comparing the right-hand sides, we see that we need to prove that
$$
\pr_i\phi_s\pr_1\phi_{s,t}=\pr_i\phi_{st}\text{ for }1\le i\le n+1,\quad
\pr_j\phi_t\pr_2\phi_{s,t}=\pr_{j+n}\phi_{st}\text{ for }1\le j\le m+1.
$$
Let us prove the first formula. Using identification~(\ref{eq:id:conc}), we get
$$\hspace{-50pt}
K\big([g_1:\cdots:g_n:g'_1:\cdots:g'_m\rr,e\big)=K([\underbrace{g_1:\cdots:g_n}_x:\underbrace{1}_p:\underbrace{g'_1:\cdots:g'_m}_y\rr,e)
$$

\vspace{-10pt}

$$\hspace{170pt}
\begin{tikzcd}
\arrow[mapsto]{r}{\pr_1\phi_{s,t}}&[10pt]K\big([g_1:\cdots:g_n\rr,e\big)\arrow[mapsto]{r}{\pr_i\phi_s}&K(g_1\cdots g_{i-1})^{-1}e.
\end{tikzcd}
$$
One can see that the right-hand side is exactly $\pr_i\phi_{s,t}$ applied to the left-hand side. As usually, we underbraced
the variables $x$, $y$, and $p$ as in~(\ref{eq:5}).

Let us prove the second formula. Using identification~(\ref{eq:id:conc}), we get
$$\hspace{-50pt}
K\big([g_1:\cdots:g_n:g'_1:\cdots:g'_m\rr,e\big)=K([\underbrace{g_1:\cdots:g_n}_x:\underbrace{1}_p:\underbrace{g'_1:\cdots:g'_m}_y\rr,e)
$$

\vspace{-10pt}

$$\hspace{20pt}
\begin{tikzcd}
\arrow[mapsto]{r}{\pr_2\phi_{s,t}}&[10pt]K\big([g'_1:\cdots:g'_m\rr,(g_1\cdots g_n)^{-1}e\big)\arrow[mapsto]{r}{\pr_j\phi_t}&K(g'_1\cdots g'_{j-1})^{-1}(g_1\cdots g_n)^{-1}e.
\end{tikzcd}
$$

\vspace{-9pt}

$$\hspace{280pt}
=K(g_1\cdots g_ng'_1\cdots g'_{j-1})^{-1}e.
$$
One can see that the right-hand side is exactly $\pr_{n+j}\phi_{s,t}$ applied to the left-hand side.
\end{proof}

\begin{remark}\label{remark:3}\rm
The maps $\phi_{s,t}$ were introduced for the universal principal bundle $E$. We will also use their
finite dimensional versions $\phi_{s,t}^N$ with $E$ replaced by $E^N$.
\end{remark}

\subsection{Fixed points}\label{Fixed_points} Let $t$ be a reflection. In that case, the lifting $\dot t$
(see Section~\ref{Semisimple_groups}) can be chosen within $G_t$. We will assume this choice
for the rest of the paper. Moreover, we will assume that the neutral element of $W$
is lifted to the neutral element of $G$.

Let $s=(s_1,\ldots,s_n)$ be a sequence of reflections. We denote by $\Gamma_s$
the set of {\it generalized combinatorial galleries} whose elements are sequences
$\gamma=(\gamma_1,\ldots,\gamma_n)$, where $\gamma_i=s_i$ or $\gamma_i=1$ for each $i$.
We will identify $\gamma$ with the point $[\dot\gamma_1:\cdots:\dot\gamma_{n-1}:\dot\gamma_n\rr$
of $\BS(s)$. This point does not depend
on the choice of liftings and is fixed by $K$. Moreover, any point of ${}^K\BS(s)$
is of this form. For each $x\in W$, we denote by $\Gamma_{s,x}$
the subset of $\Gamma_s$ consisting of $\gamma$ such that $\gamma_1\cdots\gamma_n=x$.

According to Lemma~\ref{lemma:pull-back}, the %natural
embedding
$i_\gamma:\pt\hookrightarrow\BS(s)$ taking value $\gamma\in\Gamma_{s,x}$
induces the homomorphism of
$\R$-$\R$-bimodules
$$
i_\gamma^\star:H(s)\to \R_x,
$$
Note that in this formula, the left actions are canonical and the right actions
on the domain and codomain are twisted by $A_s$ and $A_si_\gamma$, respectively (see, Lemma~\ref{lemma:pull-back}).
The lats map takes the only value $xK$. We can compute $i_\gamma^\star$ in coordinates.
\begin{theorem}\label{theorem:5} For each $\gamma\in\Gamma_{s,x}$, the composition
$$
\begin{tikzcd}
\R^{\otimes s}\arrow{r}{\theta_s}&H(s)\arrow{r}{i^\star_\gamma}&\R_x
\end{tikzcd}
$$
is given by
$
a_1\otimes\cdots\otimes a_{n+1}\mapsto a_1\cup\gamma_1(a_2)\cup\cdots\cup\gamma_1\cdots\gamma_n(a_{n+1})
$.
\end{theorem}
\begin{proof}
Replacing $\theta_s$ by $\Theta_s$, it remains to compute the following composition:
$$
\begin{tikzcd}
\R^{\otimes_\k n+1}\arrow{r}{\times}&H^\bullet((K\lq E)^{n+1},\k)\arrow{r}{\phi_s^*}&H(s)\arrow{r}{\iota_\gamma^\star}&\R_x.
\end{tikzcd}
$$
Following these maps, we get
\begin{equation}\label{eq:19}
\begin{array}{c}
\hspace{-160pt}\begin{tikzcd}%[column sep=small]
a_1\otimes\cdots\otimes a_{n+1}\arrow[mapsto]{r}{\times}&\pr_1^*(a_1)\cup\cdots\cup\pr_{n+1}^*(a_{n+1})
\end{tikzcd}\\[6pt]
\hspace{80pt}\begin{tikzcd}
\arrow[mapsto]{r}{\iota_\gamma^*\phi_s^*}&(\pr_1\phi_s(i_\gamma\times\id))^*(a_1)\cup\cdots\cup(\pr_{n+1}\phi_s(i_\gamma\times\id))^*(a_{n+1}).
\end{tikzcd}
\end{array}
\end{equation}
It is easy to note that $\pr_j\phi_s(i_\gamma\times\id)=K\lq\rho_{(\gamma_1\cdots\gamma_{j-1})^{-1}}$,
where the last map is defined in Section~\ref{Semisimple_groups}.
Therefore, by~(\ref{eq:tw5:5}), we get
$$
(\pr_i\phi_s(i_\gamma\times\id))^*(a_i)=(K\lq\rho_{(\gamma_1\cdots\gamma_{i-1})^{-1}})^*
=\gamma_1\cdots\gamma_{i-1}(a_i).
$$
The result follows from this formula and~(\ref{eq:19}).
\end{proof}
Using %the notation of Section~\ref{Standard_bimodules} and
the Mayer-Vietoris isomorphism,
we obtain that the restriction to ${}^{K\!}\BS(s)$ is the following homomorphism of $\R$-$\R$-bimodules:
$$
H^\bullet_K(\BS(s),\k)\to H^\bullet_K({}^K\BS(s),\k)\ito\bigoplus_{x\in W}\bigoplus_{\gamma\in\Gamma_{s,x}}\R_x.
$$
%where the $\Gamma_{s,x}$ is the subset of $\Gamma_s$ consisting of $\gamma$ such that $\gamma_1\cdots\gamma_n=x$.
We do not claim that this homomorphism
is a monomorphism. However, this is true under some restrictions on $\k$,
in which case we say that {\it the localization theorem holds}.

\section{Morphisms}\label{Morphisms}

\subsection{The setup}
In this section, we will construct morphisms between cohomologies of Bott-Samelson varieties $H(s)$.
These morphisms are either equivariant pull-backs
or push-forwards or compositions of both. To treat push-forwards correctly, we will assume that $2$
is invertible in $\k$. To the same end, we will only consider the varieties $\BS(s)$
for sequences of simple reflections $s$. In this case,
\begin{equation}\label{eq:BSalg}
\BS(s)\cong P_{s_1}\mathop{\times}\limits_B{P_{s_2}}\mathop{\times}\limits_B\cdots\mathop{\times}\limits_BP_{s_n}/B,
\end{equation}
where $B$ and $P_{s_i}$ are the Borel and the minimal parabolic groups of the complexification $G^c$ of $G$, respectively.
In this way, $\BS(s)$ receives a complex structure and thus an orientation.
We also get $G/K\cong G^c/B$ and thus the quotient $G/K$ also receives a complex structure.

In what follows, we will replace double brackets with simple ones. For example,
$$
\BS(s_1,\ldots,s_n)=\BS((s_1,\ldots,s_n)),\quad H(s_1,\ldots,s_n)=H((s_1,\ldots,s_n)).
$$
Moreover, we will replace any sequence of length 1 by its unique element. For example
$\phi_{t,s}=\phi_{t,(s)}$, $\theta_{t,s}=\theta_{t,(s)}$, $\R^{\otimes s}=\R^{\otimes(s)}$, etc.
We also will no longer write the symbol of the cup product between elements of $\R$.

%For the sake of the push-forward we will assume until the rest of the paper that $\k$ is a field.
%We also will no longer write the symbol of the cup product.
We draw morphisms between bimodules $H(s)$ as usual from bottom to top. The easiest case is represented
by trivial morphisms,
which are drawn as vertical lines in the colors corresponding to the colors of the reflections:

\medskip

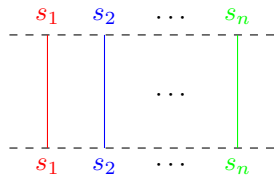
\begin{center}
\begin{tikzpicture}
\draw[dashed] (0,1.5) -- (3.5,1.5);
\draw[dashed] (0,0) -- (3.5,0);
\draw[red] (0.5,0)  node[anchor=north]{\footnotesize$s_1$};
\draw[blue] (1.25,0)  node[anchor=north]{\footnotesize$s_2$};
\draw[green] (3,0)  node[anchor=north]{\footnotesize$s_n$};
\draw (2.15,0) node[anchor=north]{\footnotesize$\cdots$};
\draw[red] (0.5,1.5)  node[anchor=south]{\footnotesize$s_1$};
\draw[blue] (1.25,1.5)  node[anchor=south]{\footnotesize$s_2$};
\draw[green] (3,1.5)  node[anchor=south]{\footnotesize$s_n$};
\draw (2.15,1.5) node[anchor=south]{\footnotesize$\cdots$};
\draw[red] (0.5,0) -- (0.5,1.5);
\draw[blue] (1.25,0) -- (1.25,1.5);
\draw[green] (3,0) -- (3,1.5);
\draw (2.15,0.7) node{\footnotesize$\cdots$};
\end{tikzpicture}
\end{center}

%We dealing with sequences of length 1, we will omit some obvious brackets:
%$$
%\BS(s)=\BS((s)),\; H(s)=H((s)),\; \phi_s=\phi_{(s)},\; \theta_s=\theta_{(s)}
%$$
%for any reflection $s$. We will also abbreviate
%$$
%\BS(s_1,\ldots,s_n)=\BS((s_1,\ldots,s_n)),\quad H(s_1,\ldots,s_n)=
%H((s_1,\ldots,s_n)),\quad \theta_{s_1,\ldots,s_n}
%$$

We denote by $\alpha_s$ the simple root %corresponding to a simple reflection $s$,
such that $s=\omega_{\alpha_s}$, see Section~\ref{Semisimple_groups}, and set
$$
P_s(f)=\frac{f+s(f)}2,\qquad \partial_s(f)=\frac{f-s(f)}{\alpha_s}.
$$
The operator $\partial_s$ in the last formula is called the {\it Demazure operator}.
We get $P_s(f)\in\R^s$, $\partial_s(f)\in\R^s$ and
\begin{equation}\label{eq:32}
f=P_s(f)+\partial_s(f)\,\frac{\alpha_s}2.
\end{equation}

\subsection{One-color morphisms}\label{One-color morphisms} Let us fix a reflection $s$. Consider the natural embedding
$\iota_s:\pt\to\BS(s)$, which maps $pt$ to $[1\rr=1K$.
We get $\iota_s=i_{(1)}$ in the notation of Section~\ref{Fixed_points}.
We get the equivariant pull-back
$$
\iota_s^\star:H(s)\to \R,
$$
which we represent as the diagram
%For any reflection $s$, we consider the startdot diagram
\begin{equation}\tag{$\iota_s^\star$}
\begin{tikzpicture}[baseline=19pt]
\draw[dashed] (0,0) -- (3,0);
\draw[dashed] (0,1.5) -- (3,1.5);
\draw[red] (1.5,0) node[anchor=north]{\footnotesize$s$} -- (1.5,0.75);
\draw[red,fill=red] (1.5,0.75) circle(0.06);
%\draw[red] (1,1) node{A} -- (2,2) node{B};
%\draw (1.5,-0.2) \node{$\small s$};
\end{tikzpicture}
\end{equation}
%\vspace{1pt}
%
%According to our stipulation, we need to construct a morphism of bimodules $H(s)\to\R$.
%Let us do it as follows: considers the natural embedding $\iota_s:\pt\to\BS(s)$, which maps $pt$ to $[1\rr$.
%We get $\iota_s=i_{(1)}$ in the notation of Section~\ref{Fixed_points}.
%Then we take the pull-back
%$$
%\iota_s^\star:H(s)\to \R.
%$$
This map is a homomorphism of $\R$-$\R$-bimodules by Lemma~\ref{lemma:pull-back}.
By Theorem~\ref{theorem:5}, we get that the composition map
$$
\begin{tikzcd}
%i_{(1)}^*\theta_{(s)}:
%\R\otimes_{\R(s)}\R\arrow{r}{\theta_s}& H(s)\arrow{r}{\iota_s^*}& \R
\R\otimes_{\R(s)}\R\arrow{r}{\theta_{s}}[swap]{\sim}& H(s)\arrow{r}{\iota_s^\star}& \R
\end{tikzcd}
$$
is given by
\begin{equation}\label{eq:31}
\iota_s^\star\theta_s(a\otimes b)=ab.
\end{equation}

Now we consider the negated equivariant push-forward
$$
-\iota_{s\star}:\R\to H(s)(2).
$$
Here and in what follows a number in round brackets means the degree shift.
We will represent the above morphism as follows:
%Now consider the enddot diagram
%\begin{equation}\label{eq:28}
%\begin{tikzcd}[row sep=small]
%{}\arrow[dash,dashed]{rr}{s}&\arrow[dash,end anchor={[yshift=-5pt]north},start anchor={[yshift=-1pt]north}]{d}&{}\\
%{}&\bullet&{}\\
%{}\arrow[dash,dashed]{rr}&&{}\\
%%{}&s&{}
%\end{tikzcd}
%\end{equation}
\vspace{-5pt}
\begin{equation}\label{eq:28}\tag{$-\iota_{s\star}$}
%\begin{center}
\begin{tikzpicture}[baseline=19pt]
\draw[dashed] (0,0) -- (3,0);
\draw[dashed] (0,1.5) -- (3,1.5);
\draw[red] (1.5,1.5) node[anchor=south]{\footnotesize$s$} -- (1.5,0.75);
\draw[red,fill=red] (1.5,0.75) circle(0.06);
%\draw[red] (1,1) node{A} -- (2,2) node{B};
%\draw (1.5,-0.2) \node{$\small s$};
\end{tikzpicture}
%\end{center}
\end{equation}
\vspace{3pt}

\noindent
%Using the same map $\iota_s$ as above, we get the push-forward map
%$$
%\iota_{s\star}:\R\to H(s)(2).
%$$

By Lemma~\ref{lemma:push-forward}, this map is a homomorphism of
$\R$-$\R$-bimodules.
We want to compute the composition:
\begin{equation}\label{eq:21}
\begin{tikzcd}
\R\arrow{r}{-i_{s\star}}&H(s)(2)\arrow{r}{\theta_s^{-1}}[swap]{\sim}&\R\otimes_{\R(s)}\R(2).
\end{tikzcd}
\end{equation}
It suffices to compute the image of $1$. As the element $-\theta_s^{-1}i_{s\star}(1)$
is central and of degree $2$, the standard argument shows that
$$
-\theta_s^{-1}i_{s\star}(1)=c\alpha_s\otimes 1+c\otimes\alpha_s
$$
for some $c\in\k$. In order to compute this constant, let
apply first $-\theta_s$ and then $\iota_s^\star$ to the formula above.
By Theorem~\ref{theorem:5} and~(\ref{eq:31}), we get
$$
-\alpha_s=\Eu_K(1K,G_s/K)=\iota_s^\star\iota_{s\star}(1)=-\iota_s^\star\theta_s(c\alpha_s\otimes 1+c\otimes\alpha_s)=-2c\alpha_s,
$$
where the second (from the left) term denotes the equivariant Euler class of the normal bundle of $1K$ as a submanifold of $G_s/K$.
Hence $c=1/2$. As all maps in~(\ref{eq:21}) are homomorphisms of bimodules, we get %from~\ref
$$
\theta_s^{-1}(-i_{s\star})(d)=d\frac{\alpha_s}2\otimes 1+d\otimes\frac{\alpha_s}2.
$$

Let $\mu_s:\BS(s,s)\to\BS(s)$ be the map given by $\mu_s([g_1:g_2\rr)=[g_1g_2\rr$.
Taking the pull-back, we get the map
$$
\mu_s^\star:H(s)\to H(s,s).
$$
We will draw this diagram as follows:

\vspace{-5pt}
\begin{equation}\tag{$\mu_s^\star$}
\begin{tikzpicture}[baseline=19pt]
\draw[dashed] (0,0) -- (3,0);
\draw[dashed] (0,1.5) -- (3,1.5);
\draw[red] (1.5,0) node[anchor=north]{\footnotesize$s$} -- (1.5,0.75);
\draw[red] (1.5,0.75) -- (0.75,1.5) node[anchor=south]{\footnotesize$s$};
\draw[red] (1.5,0.75) -- (2.25,1.5) node[anchor=south]{\footnotesize$s$};
\end{tikzpicture}
\end{equation}

This map is a homomorphism of $\R$-$\R$-bimodules by
Lemma~\ref{lemma:pull-back}. We are going to compute the composition
\begin{equation}\label{eq:22}
\begin{tikzcd}
\R\otimes_{\R(s)}\R\arrow{r}{\theta_s}&
H(s)\arrow{r}{\mu^*}&H(s,s)\arrow{r}{\theta_{(s,s)}^{-1}}&\R\otimes_{\R(s)}\R\otimes_{\R(s)}\R.
\end{tikzcd}
\!\!\!\!
\end{equation}
Computing in degree zero, we get
$$
\mu^\star\theta_s(1\otimes 1)=1=\theta_{(s,s)}(1\otimes 1\otimes 1),
$$
where the central unit is the constant unit function on $\BS(s,s){}_K\times E$.
As all maps in~(\ref{eq:22}) are homomorphisms of bimodules, we get
$$
\theta_{(s,s)}^{(-1)}\mu^\star\theta_s(a\otimes b)=a\otimes1\otimes b.
$$

As expected, the merge diagram
%\begin{equation}\label{eq:27}
%\begin{tikzcd}[row sep=small,column sep=small]
%{}\arrow[dash,dashed]{rrrr}{s}&&\arrow[dash,end anchor={[yshift=-2pt]north},start anchor={[yshift=-1pt]north}]{d}&&{}\\
%{}&{}&{}\\
%{}\arrow[dash,dashed]{rrrr}[near start,swap]{\!\!s}[near end,swap]{\;\;s}&\arrow[dash,end anchor={[yshift=-2pt]north},start anchor={[yshift=-1pt]north}]{ur}&&\arrow[dash,end anchor={[yshift=-2pt]north},start anchor={[yshift=-1pt]north}]{ul}&{}\\
%%{}&s&{}
%\end{tikzcd}
%\end{equation}
\vspace{-5pt}
\begin{equation}\label{eq:27:m}\tag{$-\mu_{s\star}$}
\begin{tikzpicture}[baseline=19pt]
\draw[dashed] (0,0) -- (3,0);
\draw[dashed] (0,1.5) -- (3,1.5);
\draw[red] (1.5,1.5) node[anchor=south]{\footnotesize$s$} -- (1.5,0.75);
\draw[red] (1.5,0.75) -- (0.75,0) node[anchor=north]{\footnotesize$s$};
\draw[red] (1.5,0.75) -- (2.25,0) node[anchor=north]{\footnotesize$s$};
\end{tikzpicture}
\end{equation}
\vspace{1pt}

\noindent
is given by the negated push-forward map $-\mu_{s\star}:H(s,s)\to H(s)(-2)$. We postpone the computation
of its coordinate form until we manage to extend diagram~(\ref{eq:28}), as well as all other diagrams of this section,
horizontally.

\subsection{Horizontal extensions}\label{Horizontal_extensions} We are going to apply the concatenation properties proved in Section~\ref{Concatenation}.
First, let us consider the planar diagram
%$$
%\begin{tikzcd}[row sep=small,column sep=tiny]
%\arrow[dash,dashed]{rrrrr}[near start]{{}\;\;\;\;\;\;\;\;\;\;\;\;\;\;\quad t_1\qquad\;\;\;\;\;\; t_n\;\;\;\;\;\;\;}{}&{}&{}&{}&&{}\\
%{}&{}&\cdots&{}&\bullet&{}\\
%\arrow[dash,dashed]{rrrrr}[swap, near start]{{}\;\;\;\;\;\;\;\;\;\;\;\;\;\;\quad t_1\qquad\;\;\;\;\;\; t_n \;\;\;\;\hspace{3pt} s}&\arrow[dash,start anchor={[yshift=-1pt]north},end anchor={[yshift=-1pt]north}]{uu}&{}&\arrow[dash,start anchor={[yshift=-1pt]north},end anchor={[yshift=-1pt]north}]{uu}&\arrow[dash,end anchor={[yshift=-5pt]north},start anchor={[yshift=-1pt]north}]{u}&{}\\
%%{}&s&{}
%\end{tikzcd}
%$$

\vspace{-5pt}
\begin{equation}\label{eq:27:l}\tag{$\iota_{t,s}^\star$}
\begin{tikzpicture}[baseline=19pt]
\draw[dashed] (-0.5,0) -- (3.5,0);
\draw[dashed] (-0.5,1.5) -- (3.5,1.5);
\draw[red] (3,0) node[anchor=north]{\footnotesize$s$}-- (3,0.75);
\draw[red,fill=red] (3,0.75) circle(0.06);
\draw (2.25,0) node[anchor=north]{\footnotesize$t_n$} -- (2.25,1.5) node[anchor=south]{\footnotesize$t_n$};
\draw (0,0) node[anchor=north]{\footnotesize$t_1$} -- (0,1.5) node[anchor=south]{\footnotesize$t_1$};
\draw (0.75,0) node[anchor=north]{\footnotesize$t_2$} -- (0.75,1.5) node[anchor=south]{\footnotesize$t_2$};
\draw  (1.5,0.7) node{\footnotesize$\cdots$};
\end{tikzpicture}
\end{equation}

\noindent
for a simple reflection $s$ and a sequence of simple reflections $t=(t_1,\ldots,t_n)$ (the black strings can be of any color).
Let $ts=(t_1,\ldots,t_n,s)$ be the concatenated sequence and
$\iota_{t,s}:\BS(t)\hookrightarrow\BS(ts)$ be the embedding given by
$\iota_{t,s}([g_1:\ldots:g_n\rr)=[g_1:\ldots:g_n:1\rr$. We claim that the diagram
$$
\begin{tikzcd}
H(t)\otimes_\R H(s)\arrow{r}{\id\otimes\iota_s^\star}\arrow{d}[swap]{\theta_{t,s}}{\wr}&H(t)\otimes_\R\R\arrow{d}{\theta_{t,\emptyset}}[swap]{\wr}\\
H(ts)\arrow{r}{\iota_{t,s}^\star}&H(t)
\end{tikzcd}
$$
is commutative. Indeed, for any $a\in H(t)$ and $b\in H(s)$, we get
$$
\theta_{t,\emptyset}(\id\otimes\iota_s^\star)(a\otimes b)=\theta_{t,\emptyset}(a\otimes\iota_s^\star(b))
=\phi_{t,\emptyset}^*(\pr_1^*(a)\cup\pr_2^*\iota_s^\star(b)).
$$
$$
\iota_{t,s}^\star\theta_{t,s}(a\otimes b)=\iota_{t,s}^\star\phi_{t,s}^*(\pr_1^*(a)\cup\pr_2^*(b)).
$$
Comaring the results, we see that it suffices to prove the equalities
\begin{equation}\label{eq:23}
\pr_1\phi_{t,\emptyset}=\pr_1\phi_{t,s}(\iota_{t,s}\ltimes{K}\id),\quad(\iota_s\ltimes{K}\id)\pr_2\phi_{t,\emptyset}=\pr_2\phi_{t,s}(\iota_{t,s}\ltimes{K}\id).
\end{equation}
They follow directly from identifiaction rule~(\ref{eq:id:conc}) for concatenation operators and definition~(\ref{eq:5}).
%, we get
%$$
%\phi_{t,\emptyset}(K([g_1:\cdots:g_n\rr,e))=\phi_{t,\emptyset}(K([g_1:\cdots:g_n:1\rr,e))=\(K([g_1:\cdots:g_n\rr,e),K(g_1\cdots g_n)^{-1}e\).
%$$

Now let us extend diagram~(\ref{eq:27:l}) to the right
%$$
%\begin{tikzcd}[row sep=small,column sep=tiny]
%\arrow[dash,dashed]{rrrrrrrr}[near start]{{}\;\;\;\;\;\;\;\;\;\;\;\;\;\;\hspace{42pt}\quad t_1\qquad\;\;\;\;\;\;\;\; t_n\;\;\;\;\;\;\;\hspace{16pt}r_1\qquad\;\;\;\;\;\;r_m}{}&{}&{}&{}&&{}&{}&{}&{}\\
%{}&{}&\cdots&{}&\bullet&{}&\cdots&\\
%\arrow[dash,dashed]{rrrrrrrr}[swap, near start]{{}\;\;\;\;\;\;\;\;\;\;\;\;\;\;\quad\hspace{44pt} t_1\qquad\;\;\;\;\;\;\;\; t_n \;\;\;\;\hspace{3pt} s\hspace{16pt}r_1\qquad\;\;\;\;\;\;r_m}&\arrow[dash,start anchor={[yshift=-1pt]north},end anchor={[yshift=-1pt]north}]{uu}&{}&\arrow[dash,start anchor={[yshift=-1pt]north},end anchor={[yshift=-1pt]north}]{uu}&\arrow[dash,end anchor={[yshift=-5pt]north},start anchor={[yshift=-1pt]north}]{u}&\arrow[dash,start anchor={[yshift=-1pt]north},end anchor={[yshift=-1pt]north}]{uu}&{}&\arrow[dash,start anchor={[yshift=-1pt]north},end anchor={[yshift=-1pt]north}]{uu}&{}\\
%%{}&s&{}
%\end{tikzcd}
%$$
\vspace{-5pt}
\begin{equation}\label{eq:27:lr}\tag{$\iota_{t,s,r}^\star$}
\begin{tikzpicture}[baseline=19pt]
\draw[dashed] (-0.5,0) -- (6.5,0);
\draw[dashed] (-0.5,1.5) -- (6.5,1.5);
\draw[red] (3,0) node[anchor=north]{\footnotesize$s$}-- (3,0.75);
\draw[red,fill=red] (3,0.75) circle(0.06);
\draw (2.25,0) node[anchor=north]{\footnotesize$t_n$} -- (2.25,1.5) node[anchor=south]{\footnotesize$t_n$};
\draw (0,0) node[anchor=north]{\footnotesize$t_1$} -- (0,1.5) node[anchor=south]{\footnotesize$t_1$};
\draw (0.75,0) node[anchor=north]{\footnotesize$t_2$} -- (0.75,1.5) node[anchor=south]{\footnotesize$t_2$};
\draw  (1.5,0.7) node{\footnotesize$\cdots$};
\draw (3.75,0) node[anchor=north]{\footnotesize$r_1$} -- (3.75,1.5) node[anchor=south]{\footnotesize$r_1$};
\draw (4.5,0) node[anchor=north]{\footnotesize$r_2$} -- (4.5,1.5) node[anchor=south]{\footnotesize$r_2$};
\draw  (5.25,0.7) node{\footnotesize$\cdots$};
\draw (6,0) node[anchor=north]{\footnotesize$r_m$} -- (6,1.5) node[anchor=south]{\footnotesize$r_m$};
\end{tikzpicture}
\end{equation}

\noindent
for another sequence of reflections $r=(r_1,\ldots,r_m)$. Let $tsr=(t_1,\ldots,t_n,s,r_1,\ldots,r_m)$ be the concatenated
sequence and $\iota_{t,s,r}:\BS(tr)\hookrightarrow\BS(tsr)$ be the map defined by
$\iota_{t,s,r}([g_1:\ldots:g_n:g'_1:\cdots:g'_m\rr)=[g_1:\ldots:g_n:1:g'_1:\cdots:g'_m\rr$.
We get the diagram
$$
\begin{tikzcd}
H(ts)\otimes_\R H(r)\arrow{r}{\iota_{t,s}^\star\otimes\id}\arrow{d}[swap]{\theta_{ts,r}}{\wr}&H(t)\otimes_\R H(r)\arrow{d}{\theta_{t,r}}[swap]{\wr}\\
H(tsr)\arrow{r}{\iota_{t,s,r}^\star}&H(tr)
\end{tikzcd}
$$
The proof of its commutativity is similar to the similar proof for the previos diagram.

A routine (but tedious) check proves the commutativity\footnote{which boils down to checking the corresponding equalities
involving the coordinatization maps $\phi_s,\phi_t,\ldots$ introduced in Section~\ref{Computation_of_the_equivariant_cohomology},
the concatenation maps $\phi_{t,s},\phi_{ts,r},\ldots$ introduced in Section~\ref{Concatenation}, and
various projections to factors and applying the appropriate identifications described in Sections~\ref{Quotient_products} and Section~\ref{More quotient products}.} of the following there-dimensional diagram:
$$
\begin{tikzcd}[column sep=-8pt]
&\R^{\otimes t}\otimes_{\R}\R^{\otimes s}\otimes_{\R}\R^{\otimes r}\arrow{ld}[swap]{\theta_t\otimes\theta_s\otimes\theta_t}\arrow{rr}{\id\otimes\iota_s^\star\theta_s\otimes\id}\arrow{dd}&&\R^{\otimes t}\otimes_{\R}\R\otimes_{\R}\R^{\otimes r}\arrow{dl}[swap]{\theta_t\otimes\id\otimes\theta_r}\arrow{dd}\\
H(t)\otimes_\R H(s)\otimes_\R H(r)\arrow[crossing over]{rr}[near end]{\id\otimes\iota_s^\star\otimes\id}\arrow{dd}[swap]{\theta_{t,s}\otimes\id}&&H(t)\otimes_\R \R\otimes_\R H(r)\\
&\R^{\otimes ts}\otimes_{\R}\R^{\otimes r}\arrow{rr}\arrow{dl}[swap]{\theta_{ts}\otimes\theta_t}\arrow{dd}&&\R^{\otimes t}\otimes_{\R}\R^{\otimes r}\arrow{dl}[swap]{\theta_t\otimes\theta_r}\arrow{dd}\\
H(ts)\otimes_\R H(r)\arrow[crossing over]{rr}[near end]{\iota_{t,s}^\star\otimes\id}\arrow[swap]{dd}{\theta_{ts,r}}&&H(t)\otimes_\R H(r)\arrow[from=uu,crossing over]{}[near start]{\theta_{t,\emptyset}\otimes\id}\\
&\R^{\otimes tsr}\arrow{rr}\arrow{dl}[swap]{\theta_{tsr}}&&\R^{\otimes tr}\arrow{dl}[swap]{\theta_{tr}}\\
H(tsr)\arrow[crossing over]{rr}[swap]{\iota_{t,s,r}^\star}&&H(tr)\arrow[from=uu,crossing over]{}[near start]{\theta_{t,r}}
\end{tikzcd}
$$
Hence
$$
\theta_{tr}^{-1}\iota_{t,s,r}^\star\theta_{tsr}(a_1\otimes\cdots\otimes a_{n+1}\otimes b_1\otimes\cdots\otimes b_{m+1})=a_1\otimes\cdots\otimes a_n\otimes a_{n+1}b_1\otimes b_2\otimes\cdots\otimes b_{m+1}.
$$

Let us extend the diagram~(\ref{eq:28}) horizontally. To this end, we first prove that the diagram
\begin{equation}\label{eq:26}
\begin{tikzcd}[column sep=35pt]
H(t)\otimes_\R\R\arrow{r}{\id\otimes\iota_{s\star}}\arrow{d}[swap]{\theta_{t,\emptyset}}&H(t)\otimes_\R H(s)(2)\arrow{d}{\theta_{t,s}}\\
H(t)\arrow{r}{\iota_{t,s\star}}&H(ts)(2)
\end{tikzcd}
\end{equation}
is commutative. It is more difficult, as it involves push-forward operators.
For any $a\in H(t)$ and $b\in\R$, we get by the first equality of~(\ref{eq:23}) and the projection formula
%$$
%\theta_{t,s}(\id\otimes\iota_{s\star})(a\otimes b)=\theta_{t,s}(a\otimes \iota_{s\star}(b))=\phi_{t,s}^*(\pr_1^*(a)\cup\pr_2^*\iota_{s\star}(b)),
%$$
\begin{equation}\label{eq:25}
\begin{array}{c}
\!\!\!\!\!\!\!\!\!\!\!
\iota_{t,s\star}\theta_{t,\emptyset}(a\otimes b)=\iota_{t,s\star}\phi_{t,\emptyset}^*(\pr_1^*(a)\cup\pr_2^*(b))
=\iota_{t,s\star}((\pr_1\phi_{t,\emptyset})^*(a)\cup (\pr_2\phi_{t,\emptyset})^*(b))\\[8pt]
\;\;\;\;\;\;\;\;\;\;\;\;=\iota_{t,s\star}\big(\iota_{t,s}^\star(\pr_1\phi_{t,s})^*(a)\cup (\pr_2\phi_{t,\emptyset})^*(b)\big)
=(\pr_1\phi_{t,s})^*(a)\cup\iota_{t,s\star}(\pr_2\phi_{t,\emptyset})^*(b).
\end{array}
\end{equation}
Let us consider %the finite dimensional vertions $\phi_{t,\emptyset}^N$ and $\phi_{t,s}^N$ of $\phi_{t,\emptyset}$ and $\phi_{t,s}$
the compatibly oriented Cartesian square
$$
\begin{tikzcd}[column sep=42pt]
\BS(t)\ltimes{K}E^N\arrow{r}{\pr_2\phi_{t,\emptyset}^N}\arrow{d}[swap]{\iota_{t,s}\ltimes{K}\id}&K\lq E^N\arrow{d}{\iota_s\ltimes{K}\id}\\
\BS(ts)\ltimes{K}E^N\arrow{r}{\pr_2\phi_{t,s}^N}&\BS(s)\ltimes{K}E^N
\end{tikzcd}
$$
By~(\ref{eq:24}), we get $(\iota_{t,s}\ltimes{K}\id)_*(\pr_2\phi_{t,\emptyset}^N)^*=(\pr_2\phi_{t,s}^N)^*(\iota_s\ltimes{K}\id)_*$. Taking the limit $N\to\infty$,
we get $\iota_{t,s\star}(\pr_2\phi_{t,\emptyset})^*=(\pr_2\phi_{t,s})^*\iota_{s\star}$. Applying this substitution to
the right-hand side of~(\ref{eq:25}), we get
\begin{multline*}
\iota_{t,s\star}\theta_{t,\emptyset}(a\otimes b)=(\pr_1\phi_{t,s})^*(a)\cup(\pr_2\phi_{t,s})^*\iota_{s\star}(b)\\[3pt]
=\phi_{t,s}^*(\pr_1^*(a)\cup\pr_2^*\iota_{s\star}(b))=\theta_{t,s}(a\otimes \iota_{s\star}(b))=\theta_{t,s}(\id\otimes\iota_{s\star})(a\otimes b).
\end{multline*}
This equality proves the commutativity of diagram~(\ref{eq:26}). We prove similarly that the diagram
$$
\begin{tikzcd}[column sep=35pt]
H(t)\otimes_\R H(r)\arrow{r}{\id\otimes\iota_{s\star}}\arrow{d}[swap]{\theta_{t,r}}&H(ts)\otimes_\R H(r)(2)\arrow{d}{\theta_{ts,r}}\\
H(tr)\arrow{r}{\iota_{t,s,r\star}}&H(tsr)(2)
\end{tikzcd}
$$
is commutative. Computing in coordiates as in the three-dimensional diagram above, we get
\begin{equation}\label{eq:29}
\begin{array}{c}\hspace{-160pt}
\theta_{tsr}^{-1}(-\iota_{t,s,r\star})\theta_{tr}(a_1\otimes\cdots\otimes a_n\otimes d\otimes b_2\otimes\cdots\otimes b_{m+1})\\[6pt]
\hspace{130pt}\displaystyle=a_1\otimes\cdots\otimes a_n\otimes\( d\frac{\alpha_s}2\otimes 1+d\otimes\frac{\alpha_s}2\)\otimes b_2\otimes\cdots\otimes b_{m+1}.
\end{array}
\end{equation}
This morphism is represented by the diagram
\begin{equation}\label{eq:27}\tag{$-\iota_{t,s,r\star}$}
\begin{tikzpicture}[baseline=19pt]
\draw[dashed] (-0.5,0) -- (6.5,0);
\draw[dashed] (-0.5,1.5) -- (6.5,1.5);
\draw[red] (3,1.5) node[anchor=south]{\footnotesize$s$}-- (3,0.75);
\draw[red,fill=red] (3,0.75) circle(0.06);
\draw (2.25,0) node[anchor=north]{\footnotesize$t_n$} -- (2.25,1.5) node[anchor=south]{\footnotesize$t_n$};
\draw (0,0) node[anchor=north]{\footnotesize$t_1$} -- (0,1.5) node[anchor=south]{\footnotesize$t_1$};
\draw (0.75,0) node[anchor=north]{\footnotesize$t_2$} -- (0.75,1.5) node[anchor=south]{\footnotesize$t_2$};
\draw  (1.5,0.7) node{\footnotesize$\cdots$};
\draw (3.75,0) node[anchor=north]{\footnotesize$r_1$} -- (3.75,1.5) node[anchor=south]{\footnotesize$r_1$};
\draw (4.5,0) node[anchor=north]{\footnotesize$r_2$} -- (4.5,1.5) node[anchor=south]{\footnotesize$r_2$};
\draw  (5.25,0.7) node{\footnotesize$\cdots$};
\draw (6,0) node[anchor=north]{\footnotesize$r_m$} -- (6,1.5) node[anchor=south]{\footnotesize$r_m$};
\end{tikzpicture}
\end{equation}

The split map $\mu_s^\star$ can be extended horizontally in the same way. Consider the diagram
\begin{equation}\tag{$\mu_{t,s,r}^\star$}
\begin{tikzpicture}[baseline=19pt]
\draw[dashed] (-0.5,0) -- (8,0);
\draw[dashed] (-0.5,1.5) -- (8,1.5);
\draw[red] (3,1.5) node[anchor=south]{\footnotesize$s$}-- (3.75,0.75);
\draw[red] (4.5,1.5) node[anchor=south]{\footnotesize$s$}-- (3.75,0.75);
\draw[red] (3.75,0.75)-- (3.75,0)node[anchor=north]{\footnotesize$s$};
%\draw[red,fill=red] (3,0.75) circle(0.06);
\draw (2.25,0) node[anchor=north]{\footnotesize$t_n$} -- (2.25,1.5) node[anchor=south]{\footnotesize$t_n$};
\draw (0,0) node[anchor=north]{\footnotesize$t_1$} -- (0,1.5) node[anchor=south]{\footnotesize$t_1$};
\draw (0.75,0) node[anchor=north]{\footnotesize$t_2$} -- (0.75,1.5) node[anchor=south]{\footnotesize$t_2$};
\draw  (1.5,0.7) node{\footnotesize$\cdots$};
\draw (5.25,0) node[anchor=north]{\footnotesize$r_1$} -- (5.25,1.5) node[anchor=south]{\footnotesize$r_1$};
\draw (6,0) node[anchor=north]{\footnotesize$r_2$} -- (6,1.5) node[anchor=south]{\footnotesize$r_2$};
\draw  (6.75,0.7) node{\footnotesize$\cdots$};
\draw (7.5,0) node[anchor=north]{\footnotesize$r_m$} -- (7.5,1.5) node[anchor=south]{\footnotesize$r_m$};
\end{tikzpicture}
\end{equation}
for sequences of simple reflections $t=(t_1,\ldots,t_n)$ and $r=(r_1,\ldots,r_m)$. Let $\mu_{t,s,r}:\BS(tssr)\to\BS(tsr)$
be the map\footnote{here $tssr=(t_1,\ldots,t_n,s,s,r_1,\ldots,r_m)$.} given by $\mu_{t,s,r}([g_1:\cdots:g_n:g'_1:g_2':g''_1:\cdots:g''_m\rr)=[g_1:\cdots:g_n:g'_1g_2':g''_1:\cdots:g''_m\rr$.
Arguing as for $\iota_{t,s,r}^\star$, we get
$$
\theta_{tssr}^{(-1)}\mu_{t,s,r}^\star\theta_{tsr}(a_1\otimes\cdots\otimes a_{n+1}\otimes b_1\otimes\cdots\otimes b_{m+1})
=a_1\otimes\cdots\otimes a_{n+1}\otimes1\otimes b_1\otimes\cdots\otimes b_{m+1}.
$$

At this point, we can compute the coordinate form of merge diagram~(\ref{eq:27:m}).
We clearly have $\mu_s\iota_{s,s,\emptyset}=\id$. Taking the push-forwards, we get $(-\mu_{s\star})(-\iota_{s,s,\emptyset\star})=\id$.
Expressed diagrammatically, we get our first relation:
%\begin{equation}\label{eq:27}
%\begin{tikzcd}[row sep=small,column sep=small]
%{}\arrow[dash,dashed]{rrrr}{s}&&\arrow[dash,end anchor={[yshift=-2pt]north},start anchor={[yshift=-1pt]north}]{d}&&{}\\
%{}&{}&{}\\
%{}&\arrow[dash,end anchor={[yshift=-2pt]north},start anchor={[yshift=-1pt]north}]{ur}&&\arrow[dash,end anchor={[yshift=-2pt]north},start anchor={[yshift=-1pt]north}]{ul}&{}\\
%%{}&s&{}
%\end{tikzcd}
%\end{equation}
\begin{center}
\begin{tikzpicture}[baseline=25pt]
\draw[dashed] (0,0) -- (3,0);
\draw[dashed] (0,2) -- (3,2);
\draw[red] (1.5,2) node[color=black,anchor=south]{\footnotesize$s$} -- (1.5,1.5);
\draw[red] (1.5,1.5) -- (1,1);
\draw[red] (1.5,1.5) -- (2,1);
\draw[red] (1,1) -- (1,0) node[color=black,anchor=north]{\footnotesize$s$};
\draw[red] (2,1) -- (2,0.5);
\draw[red,fill=red] (2,0.5) circle(0.06);
%\draw (3.7,1) node{$=$};
\end{tikzpicture}
\qquad
=
\qquad
\begin{tikzpicture}[baseline=25pt]
\draw[dashed] (0,0) -- (3,0);
\draw[dashed] (0,2) -- (3,2);
\draw[red] (1.5,2) node[color=black,anchor=south]{\footnotesize$s$} -- (1.5,0) node[color=black,anchor=north]{\footnotesize$s$};
\end{tikzpicture}
\end{center}
\noindent
Note that we have proved this relation without resorting to coordinatization. By~(\ref{eq:29}), we get
\begin{multline*}
1\otimes1=\theta_s^{-1}\mu_{s\star}\iota_{s,s,\emptyset\star}\theta_s(1\otimes 1)=\theta_s^{-1}(-\mu_{s\star})\theta_{(s,s)}\(1\otimes\frac{\alpha_s}2\otimes 1+1\otimes 1\otimes\frac{\alpha_s}2\)\\[8pt]
=\theta_s^{-1}(-\mu_{s\star})\theta_{(s,s)}\(1\otimes\frac{\alpha_s}2\otimes 1\)-\theta_s^{-1}\mu_{s\star}\theta_{(s,s)}\(1\otimes 1\otimes 1\)\frac{\alpha_s}2.
\end{multline*}
As $\theta_s^{-1}\mu_{s\star}\theta_{(s,s)}\(1\otimes 1\otimes 1\)=0$ for the degree reason, we get
$$
\theta_s^{-1}(-\mu_{s\star})\theta_{(s,s)}\(1\otimes\frac{\alpha_s}2\otimes 1\)=1\otimes1.
$$
Let us take any $b\in H(s)$. Applying~(\ref{eq:32}) and the above formula, we get
%%>Написать о нём.
\begin{multline*}
\theta_s^{-1}(-\mu_{s\star})\theta_{(s,s)}(1\otimes b\otimes 1)=\\
-P_s(b)\,\theta_s^{-1}\mu_{s\star}\theta_{(s,s)}(1\otimes 1\otimes 1)
+\partial_s(b)\,\theta_s^{-1}(-\mu_{s\star})\theta_{(s,s)}\(1\otimes\frac{\alpha_s}2\otimes 1\)
=\partial_s(b)\otimes 1.
\end{multline*}
As this map is a homomorphism of bimodules, we get
$$
\theta_s^{-1}(-\mu_{s\star})\theta_{(s,s)}(a\otimes b\otimes c)=a\partial_s(b)\otimes c.
$$
The arguments described in this section, allow us to extend $-\mu_{s\star}$ to the morphism $-\mu_{t,s,r\star}$
represented by the diagram
\begin{equation}\tag{$-\mu_{t,s,r\star}$}
\begin{tikzpicture}[baseline=19pt]
\draw[dashed] (-0.5,0) -- (8,0);
\draw[dashed] (-0.5,1.5) -- (8,1.5);
\draw[red] (3,0) node[anchor=north]{\footnotesize$s$}-- (3.75,0.75);
\draw[red] (4.5,0) node[anchor=north]{\footnotesize$s$}-- (3.75,0.75);
\draw[red] (3.75,0.75)-- (3.75,1.5)node[anchor=south]{\footnotesize$s$};
%\draw[red,fill=red] (3,0.75) circle(0.06);
\draw (2.25,0) node[anchor=north]{\footnotesize$t_n$} -- (2.25,1.5) node[anchor=south]{\footnotesize$t_n$};
\draw (0,0) node[anchor=north]{\footnotesize$t_1$} -- (0,1.5) node[anchor=south]{\footnotesize$t_1$};
\draw (0.75,0) node[anchor=north]{\footnotesize$t_2$} -- (0.75,1.5) node[anchor=south]{\footnotesize$t_2$};
\draw  (1.5,0.7) node{\footnotesize$\cdots$};
\draw (5.25,0) node[anchor=north]{\footnotesize$r_1$} -- (5.25,1.5) node[anchor=south]{\footnotesize$r_1$};
\draw (6,0) node[anchor=north]{\footnotesize$r_2$} -- (6,1.5) node[anchor=south]{\footnotesize$r_2$};
\draw  (6.75,0.7) node{\footnotesize$\cdots$};
\draw (7.5,0) node[anchor=north]{\footnotesize$r_m$} -- (7.5,1.5) node[anchor=south]{\footnotesize$r_m$};
\end{tikzpicture}
\end{equation}
and prove the coordinate formula
\begin{multline*}
\theta_{tsr}^{-1}(-\mu_{t,s,r\star})\theta_{tssr}(a_1\otimes\cdots\otimes a_n\otimes a\otimes b\otimes c\otimes b_2\otimes\cdots\otimes b_{m+1})\\[3pt]
=a_1\otimes\cdots\otimes a_n\otimes a\partial_s(b)\otimes c\otimes b_2\otimes\cdots\otimes b_{m+1}.
\end{multline*}

\subsection{Generalized Bott-Samelson varieties}\label{Generalized Bott-Samelson varieties} For any nonempty set $S$ of simple reflections,
let $G_S$ be the subgroup of $G$ generated by all groups $G_s$, where $s\in S$.
%%>>Замкнутая?
Let $S=(S_1,\ldots,S_n)$ be a sequence of nonempty sets of reflections. Then we define
$$
\BS(S)=G_{S_1}\mathop{\times}\limits_K{G_{S_2}}\mathop{\times}\limits_K\cdots\mathop{\times}\limits_KG_{S_n}/K
$$
Clearly, if each $S_i$ is a singleton, then $\BS(S)=\BS(s)$, where $S_i=\{s_i\}$ and $s=(s_1,\ldots,s_n)$.
Bearing in mind this fact, we will omit brackets for singletons. For example, we abbreviate
$\BS(\{s\},\{t,r\},\{p\},\{q\})$ to $\BS(s,\{t,r\},p,q)$. We denote the equivariant cohomologies by
$H(S)=H_K^\bullet(\BS(S),\k)$, use the similar abbreviation and draw $H(S)$ and the identical morphisms
between them similarly to the cohomologies and morphisms for the usual Bott-Samelson varieties. For example,
the identity map $H(s,\{t,r\},p,q)\to H(s,\{t,r\},p,q)$ is depicted as
%\medskip

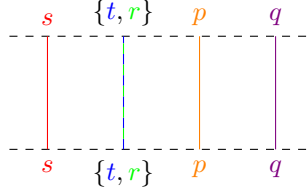
\vspace{-6pt}

\begin{center}
\begin{tikzpicture}
\draw[dashed] (0,1.5) -- (4,1.5);
\draw[dashed] (0,0) -- (4,0);
\draw[red] (0.5,0)  node[anchor=north]{\footnotesize$s$};
\draw(1.5,0)  node[anchor=north]{\footnotesize$\color{black}\{\color{blue}t\color{black},\color{green}r\color{black}\}$};
\draw[orange] (2.5,0)  node[anchor=north]{\footnotesize$p$};
\draw[violet] (3.5,0)  node[anchor=north]{\footnotesize$q$};

%\draw (2.15,0) node[anchor=north]{\footnotesize$\cdots$};
\draw[red] (0.5,1.5)  node[anchor=south]{\footnotesize$s$};
\draw(1.5,1.5)  node[anchor=south]{\footnotesize$\color{black}\{\color{blue}t\color{black},\color{green}r\color{black}\}$};
%\draw[green] (3,1.5)  node[anchor=south]{\footnotesize$s_n$};
\draw[orange] (2.5,1.5)  node[anchor=south]{\footnotesize$p$};
\draw[violet] (3.5,1.5)  node[anchor=south]{\footnotesize$q$};
%\draw (2.15,1.5) node[anchor=south]{\footnotesize$\cdots$};
\draw[red] (0.5,0) -- (0.5,1.5);
\draw[blue,dash pattern= on 4pt off 4pt] (1.5,0) -- (1.5,1.5);
\draw[green,dash pattern= on 4pt off 4pt,dash phase=4pt] (1.5,0) -- (1.5,1.5);

\draw[orange] (2.5,0) -- (2.5,1.5);
\draw[violet] (3.5,0) -- (3.5,1.5);
%\draw (2.15,0.7) node{\footnotesize$\cdots$};
\end{tikzpicture}
\end{center}

\vspace{-4pt}

We will consider here only the case where each set $S_i$ is a singleton with exactly one exception
for which this set consists of two distinct simple reflections (as in the example above).
These varieties also have complex structures coming from algebraic realization similar to~(\ref{eq:BSalg}).

\subsection{Two-color morphisms}\label{Two-color morphisms}
%First we restrict our attention to the Bott-Samelson varieties $\BS(s,t)$
%for sequences $s$ of simple reflections.
For the simplicity of exposition, we will omit the coordinatization maps $\theta_s$,
thus thinking of them as identity maps. We denote
$$
x_s=\frac{\alpha_s}2,\quad c_s=x_s\otimes 1+1\otimes x_s
$$
for any simple reflection $s$.

If we consider $H(s)$ as a left $\R$-module, then it has the following $\R$-basis:
$$
\{1\otimes x_{s_1}^{\epsilon_1}\otimes\cdots\otimes x_{s_n}^{\epsilon_n}\suchthat \epsilon_1,\ldots,\epsilon_n\in\{0,1\}\}.
$$
It consists of $2^n$ elements, where $n$ is the length of $s$.
We denote by $H(s)^{<}$ the left $\R$-sub\-module of $H(s)$ generated by all
above basis elements such that $\epsilon_i=0$ for some $i$.
The remaining basis element $1\otimes x_{s_1}\otimes\cdots\otimes x_{s_n}$
will be called the {\it normal element} of $H(s)$, see~\cite[D\'efinition 4.6]{Libedinsky}.

We leave the proof of the following result to the reader.

\begin{proposition}\label{proposition:4}
$c_{s_1}\cdots c_{s_n}=1\otimes x_{s_1}\otimes\cdots\otimes x_{s_n}+h$ for some $h\in H(s)^{<}$.
\end{proposition}

For any distinct simple reflections $s$ and $t$, we denote by $m_{s,t}$
the order of the product $st$. We have $m_{s,t}=m_{t,s}\in\{2,3,4,6\}$.
% and this number can take one of the following values:
We define the following sequence
$$
[s,t]=\left\{\!\!
\begin{array}{ll}
(s,t),&\text{ if }m_{s,t}=2,\\
(s,t,s),&\text{ if }m_{s,t}=3,\\
(s,t,s,t)&\text{ if }m_{s,t}=4,\\
(s,t,s,t,s,t)&\text{ if }m_{s,t}=6.
\end{array}
\right.
$$
%We have
%$$
%st=ts,\quad sts=tst,\quad stst=tsts,\quad ststst=tststs,
%$$
%respectively.
The product of the reflections of $[s,t]$ is denoted by $w_{s,t}$. The braid relation for this pair reads as $w_{s,t}=w_{t,s}$.
We will use the notation $\BS(s,t,\ldots)=\BS([s,t])$ and $H(s,t,\ldots)=H([s,t])$
and denote by $1\otimes x_s\otimes x_t\otimes\cdots$ the normal element of $H(s,t,\ldots)$.
Similarly $c_sc_t\cdots$ denotes one of the product $c_sc_t$, $c_sc_tc_s$, $c_sc_tc_sc_t$, $c_sc_tc_sc_tc_sc_t$
depending on $m_{s,t}$.

The image of the twisting map $A_{[s,t]}$, see~(\ref{eq:30}),
is actually contained in the Schubert variety
%$X_{s,t}=\overline{K_{w_{s,t}}/K}$.
$X_{s,t}=\overline{K_{w_{s,t}}/K}=\BS(\{s,t\})$. Thus we get the map
$\eta_{s,t}:\BS(s,t,\ldots)\to X_{s,t}$. The twisting maps for both sides are $A_{[s,t]}$
and the natural inclusion to $G/K$, respectively. Note also that both these spaces
are smooth manifolds having real dimension $2m_{s,t}$. Therefore, we get the following two morphisms of $\R$-$\R$-bimodules:
$$
\eta_{s,t}^\star:H_K^\bullet(X_{s,t},\k)\to H(s,t,\ldots),\quad
\eta_{s,t\star}:H(s,t,\ldots)\to H_K^\bullet(X_{s,t},\k)
$$
presented by the diagrams
%\begin{center}
%\begin{tikzpicture}
%\draw[dashed] (0,0) -- (3,0);
%\draw[dashed] (0,1.5) -- (3,1.5);
%\draw[red] (1.5,1.5) node[anchor=south]{\footnotesize$s$} -- (1.5,0.75);
%\draw[red] (1.5,0.75) -- (0.75,1.5) node[anchor=south]{\footnotesize$s$};
%\draw[red] (1.5,0.75) -- (2.25,1.5) node[anchor=south]{\footnotesize$s$};
%\end{tikzpicture}
%\end{center}
%Composing these maps, we will get the desired maps.

\vspace{-5pt}
%\begin{equation}\tag{$\eta_{s,t}^\star,\eta_{s,t\star}$}
$\displaystyle
(\eta_{s,t}^\star)\hfill
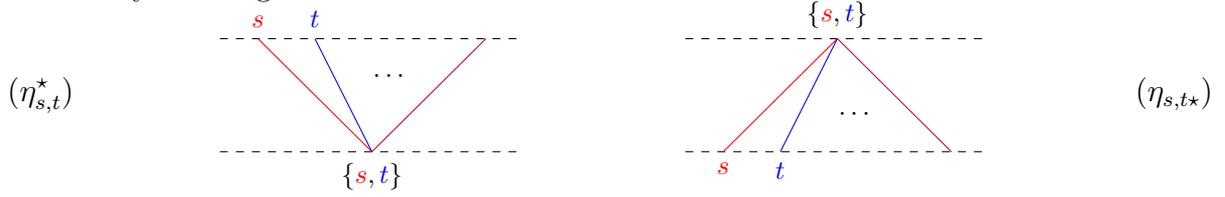
\begin{tikzpicture}[baseline=19pt]
\draw[dashed] (-0.5,0) -- (3.5,0);
\draw[dashed] (-0.5,1.5) -- (3.5,1.5);
\draw[red] (1.5,0) node[anchor=north]{\footnotesize$\color{black}\{\color{red}s\color{black},\color{blue}t\color{black}\}$} -- (0,1.5) node[anchor=south]{\footnotesize$s$};
\draw[blue] (1.5,0) -- (0.75,1.5) node[anchor=south]{\footnotesize$t$};
\draw[purple] (1.5,0) -- (3,1.5);
\draw  (1.75,1) node{\footnotesize$\cdots$};
\end{tikzpicture}
\hspace{60pt}
\begin{tikzpicture}[baseline=19pt]
\draw[dashed] (-0.5,0) -- (3.5,0);
\draw[dashed] (-0.5,1.5) -- (3.5,1.5);
\draw[red] (1.5,1.5) node[anchor=south]{\footnotesize$\color{black}\{\color{red}s\color{black},\color{blue}t\color{black}\}$} -- (0,0) node[anchor=north]{\footnotesize$s$};
\draw[blue] (1.5,1.5) -- (0.75,0) node[anchor=north]{\footnotesize$t$};
\draw[purple] (1.5,1.5) -- (3,0);
\draw  (1.75,0.5) node{\footnotesize$\cdots$};
\end{tikzpicture}
\hfill(\eta_{s,t\star})
$
%\end{equation}
%$
\vspace{1pt}

\noindent
Here the purple line can be either red or blue, depending on the parity of $m_{s,t}$.

First, we calculate the following composition.
\begin{lemma}\label{lemma:10}
$\eta_{s,t\star}\eta_{s,t}^\star=\id$.
\end{lemma}
\begin{proof}
Applying the projection formula, we get
$$
\eta_{s,t\star}\eta_{s,t}^\star(h)=\eta_{s,t\star}(1\cup\eta_{s,t}^\star(h))=\eta_{s,t\star}(1)\cup h
$$
for any $h\in H_K^\bullet(X_{s,t},\k)$. Therefore, it suffices to prove that $\eta_{s,t\star}(1)=1$.
There is an open subset $V\subset X_{s,t}$ such that
$\eta_{s,t}$ induces the orientation preserving diffeomorphism $\eta'_{s,t}:U\ito V$, where $U=\eta_{s,t}^{-1}(V)$.
Consider the following Cartesian square
$$
\begin{tikzcd}
U\arrow[hook]{r}{j_U}\arrow{d}{\wr}[swap]{\eta'_{s,t}}&\BS(s,t,\ldots)\arrow{d}{\eta_{s,t}}\\
V\arrow[hook]{r}{j_V}&X_{s,t}
\end{tikzcd}
$$
By (the equariant version of)~(\ref{eq:24}), we get
$$
j_V^\star\eta_{s,t\star}(1)=\eta'_{s,t\star}j_U^\star(1)=\eta'_{s,t\star}(1)=1.
$$
As $X_{s,t}$ is connected, the above calculation proves that $\eta_{s,t\star}(1)=1$.
\end{proof}
Diagramatically expresed this lemma looks as follows:
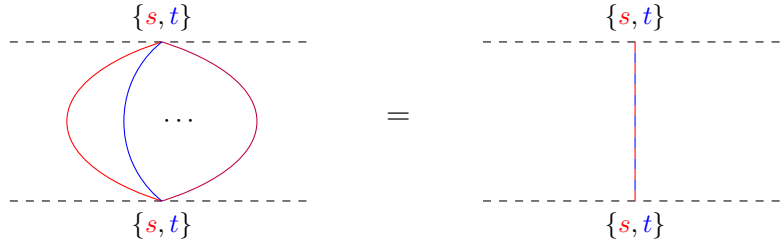
\vspace{-5pt}
\begin{equation*}
\begin{tikzpicture}[baseline=29pt,y=40pt]
\draw[dashed] (-0.5,0) -- (3.5,0);
\draw[dashed] (-0.5,1.5) -- (3.5,1.5);
\draw  (1.75,0.75) node{\footnotesize$\cdots$};
\draw [red] plot [smooth,tension=1.2] coordinates {(1.5,0) (0.25,0.75) (1.5,1.5)};
\draw [blue] plot [smooth,tension=1.2] coordinates {(1.5,0) (1,0.75) (1.5,1.5)};
\draw [purple] plot [smooth,tension=1.2] coordinates {(1.5,0) (2.75,0.75) (1.5,1.5)};
\draw (1.5,0) node[anchor=north]{\footnotesize$\color{black}\{\color{red}s\color{black},\color{blue}t\color{black}\}$};
\draw (1.5,1.5) node[anchor=south]{\footnotesize$\color{black}\{\color{red}s\color{black},\color{blue}t\color{black}\}$};
\end{tikzpicture}
\qquad
=
\qquad
%\begin{tikzpicture}[baseline=-3.1pt]
%\draw[dashed] (-0.5,0) -- (3.5,0);
%\end{tikzpicture}
\begin{tikzpicture}[baseline=29pt,y=40pt]
\draw[dashed] (-0.5,0) -- (3.5,0);
\draw[dashed] (-0.5,1.5) -- (3.5,1.5);
\draw[red,dash pattern= on 4pt off 4pt] (1.5,0) -- (1.5,1.5);
\draw[blue,dash pattern= on 4pt off 4pt,dash phase=4pt] (1.5,0) node[anchor=north]{\footnotesize$\color{black}\{\color{red}s\color{black},\color{blue}t\color{black}\}$}-- (1.5,1.5)node[anchor=south]{\footnotesize$\color{black}\{\color{red}s\color{black},\color{blue}t\color{black}\}$};
\end{tikzpicture}
\end{equation*}
\vspace{1pt}

\subsection{The $2m_{s,t}$-valent vertex} Here, we are going to consider
the composition $\eta_{t,s}^\star\eta_{s,t\star}:H(s,t,\ldots)\to H(t,s,\ldots)$,
which is a morphism of $\R$-$\R$-bimodules. It is depicted as follows:

\vspace{-5pt}
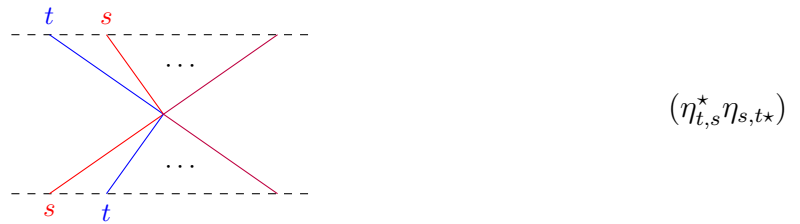
\begin{equation}\tag{$\eta_{t,s}^\star\eta_{s,t\star}$}
\begin{tikzpicture}[baseline=29pt,y=40pt]
\draw[dashed] (-0.5,0) -- (3.5,0);
\draw[dashed] (-0.5,1.5) -- (3.5,1.5);
\draw[red] (1.5,0.75) -- (0,0) node[anchor=north]{\footnotesize$s$};
\draw[blue] (1.5,0.75) -- (0.75,0) node[anchor=north]{\footnotesize$t$};
\draw[purple] (1.5,0.75) -- (3,0);
\draw  (1.75,0.25) node{\footnotesize$\cdots$};
\draw[blue] (1.5,0.75) -- (0,1.5) node[anchor=south]{\footnotesize$t$};
\draw[red] (1.5,0.75) -- (0.75,1.5) node[anchor=south]{\footnotesize$s$};
\draw[purple] (1.5,0.75) -- (3,1.5);
\draw  (1.75,1.2) node{\footnotesize$\cdots$};
\end{tikzpicture}
\end{equation}

Note that all such morphisms
are proportional~\cite[Proposition 4.3]{Libedinsky}.

%\begin{lemma}\label{lemma:11}
%$\eta_{t,s}^\star\eta_{s,t\star}\ne0$.
%\end{lemma}
%\begin{proof}
%Supposing the converse $\eta_{t,s}^\star\eta_{s,t\star}=0$, we get by Lemma~\ref{lemma:10} the contradiction
%$$
%0=\eta_{t,s\star}(\eta_{t,s}^\star\eta_{s,t\star})\eta_{s,t}^\star=(\eta_{t,s\star}\eta_{t,s}^\star)(\eta_{s,t\star}\eta_{s,t}^\star)=\id.
%$$
%\end{proof}

\begin{lemma}\label{lemma:12}
$\ker\mu_{s,t\star}\subset H(s,t,\ldots)^{<}$.
\end{lemma}
\begin{proof}
Let $h\in\ker\mu_{s,t\star}$. We have $h=\alpha\otimes x_s\otimes x_t\otimes\cdots+h'$ for some $h'\in H(s,t,\ldots)^{<}$.
For the degree reason, we get $\eta_{t,s}^\star\eta_{s,t\star}(h')\in H(t,s,\ldots)^{<}$.
Therefore, it follows from $\eta_{t,s}^\star\eta_{s,t\star}(h)=0$ that
$$
\alpha\,\eta_{t,s}^\star\eta_{s,t\star}(1\otimes x_s\otimes x_t\otimes\cdots)\in H(t,s,\ldots)^{<}.
$$
By~\cite[Proposition 4.3 and Lemme 4.7]{Libedinsky}, this is only possible if $\alpha=0$.
\end{proof}

Our next aim is to prove the following normalization condition.

\begin{lemma}\label{lemma:13}
$\eta_{t,s}^\star\eta_{s,t\star}(1\otimes x_s\otimes x_t\otimes\cdots)=1\otimes x_t\otimes x_s\otimes\cdots+h$
for some $h\in H(t,s,\ldots)^{<}$.
\end{lemma}
\begin{proof}
Let us consider the following chain of inclusions:
$$
\begin{tikzcd}
\pt\arrow{r}{\iota_{\emptyset,s,\emptyset}}&\BS(s)\arrow{r}{\iota_{(s),t,\emptyset}}&[9pt]\BS(s,t)\arrow{r}{\iota_{(s,t),s,\emptyset}}&[12pt]\cdots\arrow{r}{\iota_{(s,t,\ldots),r,\emptyset}}&[20pt]\BS(s,t,\ldots),
\end{tikzcd}
$$
where $r=s$ or $r=t$ depending on the parity of $m_{s,t}$.
We denote the resulting composition by $\iota$. Its image is $[1:1:\cdots:1\rr$.
Taking the equivariant push-forward

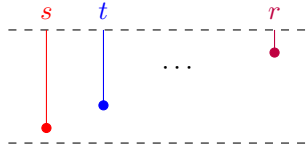
\vspace{-10pt}
\begin{equation*}
\begin{tikzpicture}[baseline=29pt]
\draw[dashed] (-0.5,0) -- (3.5,0);
\draw[dashed] (-0.5,1.5) -- (3.5,1.5);
%\draw[red] (1.5,0.75) -- (0,0) node[anchor=north]{\footnotesize$s$};
%\draw[blue] (1.5,0.75) -- (0.75,0) node[anchor=north]{\footnotesize$t$};
%\draw[purple] (1.5,0.75) -- (3,0);
%\draw  (1.75,0.25) node{\footnotesize$\cdots$};
%\draw[blue] (1.5,0.75) -- (0,1.5) node[anchor=south]{\footnotesize$t$};
%\draw[red] (1.5,0.75) -- (0.75,1.5) node[anchor=south]{\footnotesize$s$};
%\draw[purple] (1.5,0.75) -- (3,1.5);
\draw  (1.75,1) node{\footnotesize$\cdots$};
\draw[red] (0,0.2) -- (0,1.5) node[anchor=south]{\footnotesize$s$};
\draw[blue] (0.75,0.5) -- (0.75,1.5) node[anchor=south]{\footnotesize$t$};
\draw[purple] (3,1.2) -- (3,1.5) node[anchor=south]{\footnotesize$r$};
\draw[red,fill=red] (0,0.2) circle(0.06);
\draw[blue,fill=blue] (0.75,0.5) circle(0.06);
\draw[purple,fill=purple] (3,1.2) circle(0.06);
\end{tikzpicture}
\end{equation*}

\bigskip

\noindent
we get by~(\ref{eq:29}) that the image of $(-1)^{m_{s,t}}$
is equal to $c_sc_t\cdots$. Hence $\eta_{s,t\star}(c_sc_t\cdots)=(\eta_{s,t}\iota)_\star((-1)^{m_{s,t}})$.
Arguing similarly, we get $\eta_{t,s\star}(c_tc_s\cdots)=(\eta_{t,s}\iota)_\star((-1)^{m_{s,t}})$.
As $\eta_{s,t}\iota=\eta_{t,s}\iota$, we get by Lemmas~\ref{lemma:10} and~\ref{lemma:12} that
$$
\eta_{t,s}^\star\eta_{s,t\star}(c_sc_t\cdots)-c_tc_s\cdots\in\ker\eta_{t,s\star}\subset H(t,s,\ldots)^{<}.
$$
To conclude the proof, it suffices to apply Proposition~\ref{proposition:4} and the fact
that $\eta_{t,s}^\star\eta_{s,t\star}$, being a degree preseving homomorphism $H(s,t,\ldots)\to H(t,s,\ldots)$
of left $\R$-modules, maps $H(s,t,\ldots)^{<}$ to $H(t,s,\ldots)^{<}$.
\end{proof}

Thus the composition $\eta_{t,s}^\star\eta_{s,t\star}$ is just the map $f_{s,t}$ as in~\cite[Lemme 4.7]{Libedinsky}.

\subsection{The Jones-Wenzl projector}\label{The Jones-Wenzl projector} This map is given by the composition $\eta_{s,t}^\star\eta_{s,t\star}$
and is represented by the diagram:

\vspace{-5pt}
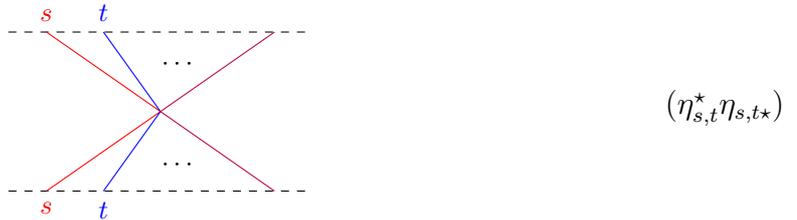
\begin{equation}\tag{$\eta_{s,t}^\star\eta_{s,t\star}$}
\begin{tikzpicture}[baseline=29pt,y=40pt]
\draw[dashed] (-0.5,0) -- (3.5,0);
\draw[dashed] (-0.5,1.5) -- (3.5,1.5);
\draw[red] (1.5,0.75) -- (0,0) node[anchor=north]{\footnotesize$s$};
\draw[blue] (1.5,0.75) -- (0.75,0) node[anchor=north]{\footnotesize$t$};
\draw[purple] (1.5,0.75) -- (3,0);
\draw  (1.75,0.25) node{\footnotesize$\cdots$};
\draw[red] (1.5,0.75) -- (0,1.5) node[anchor=south]{\footnotesize$s$};
\draw[blue] (1.5,0.75) -- (0.75,1.5) node[anchor=south]{\footnotesize$t$};
\draw[purple] (1.5,0.75) -- (3,1.5);
\draw  (1.75,1.2) node{\footnotesize$\cdots$};
\end{tikzpicture}
\end{equation}
\vspace{1pt}

Form Lemma~\ref{lemma:10}, we immediately get the relations

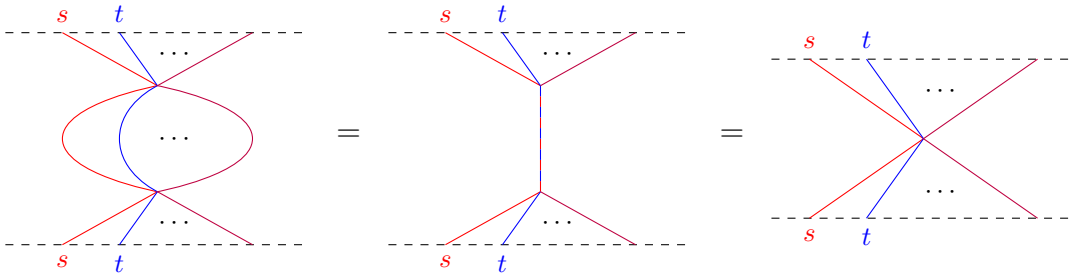
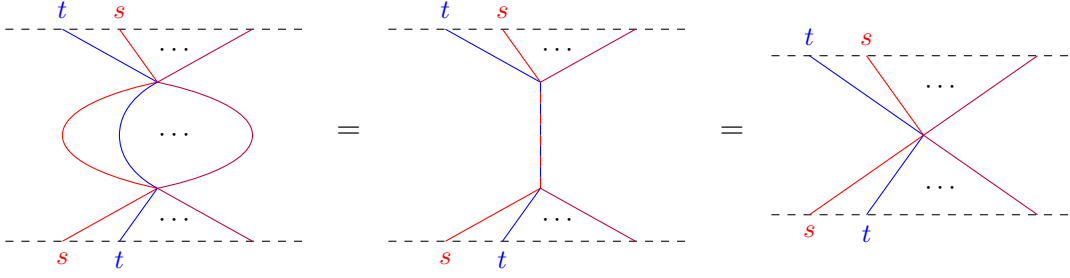
\vspace{-5pt}
\begin{equation*}
\begin{tikzpicture}[baseline=19pt,y=40pt]
\draw[dashed] (-0.5,-0.5) -- (3.5,-0.5);
\draw[dashed] (-0.5,1.5) -- (3.5,1.5);
\draw  (1.75,0.5) node{\footnotesize$\cdots$};
\draw [red] plot [smooth,tension=1.2] coordinates {(1.5,0) (0.25,0.5) (1.5,1)};
\draw [blue] plot [smooth,tension=1.2] coordinates {(1.5,0) (1,0.5) (1.5,1)};
\draw [purple] plot [smooth,tension=1.2] coordinates {(1.5,0) (2.75,0.5) (1.5,1)};
%\draw (1.5,0) node[anchor=north]{\footnotesize$\color{black}\{\color{red}s\color{black},\color{blue}t\color{black}\}$};
%\draw (1.5,1) node[anchor=south]{\footnotesize$\color{black}\{\color{red}s\color{black},\color{blue}t\color{black}\}$};
\draw [red]  (1.5,1) -- (0.25,1.5) node[anchor=south]{\footnotesize$s$};
\draw [blue] (1.5,1) -- (1,1.5) node[anchor=south]{\footnotesize$t$};
\draw [purple] (1.5,1) -- (2.75,1.5);
\draw [red]  (1.5,0) -- (0.25,-0.5) node[anchor=north]{\footnotesize$s$};
\draw [blue] (1.5,0) -- (1,-0.5) node[anchor=north]{\footnotesize$t$};
\draw [purple] (1.5,0) -- (2.75,-0.5);
\draw  (1.75,1.3) node{\footnotesize$\cdots$};
\draw  (1.75,-0.3) node{\footnotesize$\cdots$};
\end{tikzpicture}
\;\;=\;\;
\begin{tikzpicture}[baseline=19pt,y=40pt]
\draw[dashed] (-0.5,-0.5) -- (3.5,-0.5);
\draw[dashed] (-0.5,1.5) -- (3.5,1.5);
%\draw  (1.75,0.5) node{\footnotesize$\cdots$};
%\draw [red] plot [smooth,tension=1.2] coordinates {(1.5,0) (0.25,0.5) (1.5,1)};
%\draw [blue] plot [smooth,tension=1.2] coordinates {(1.5,0) (1,0.5) (1.5,1)};
%\draw [purple] plot [smooth,tension=1.2] coordinates {(1.5,0) (2.75,0.5) (1.5,1)};
%\draw (1.5,0) node[anchor=north]{\footnotesize$\color{black}\{\color{red}s\color{black},\color{blue}t\color{black}\}$};
%\draw (1.5,1) node[anchor=south]{\footnotesize$\color{black}\{\color{red}s\color{black},\color{blue}t\color{black}\}$};
\draw[red,dash pattern= on 4pt off 4pt] (1.5,0) -- (1.5,1);
\draw[blue,dash pattern= on 4pt off 4pt,dash phase=4pt] (1.5,0) -- (1.5,1);
\draw [red]  (1.5,1) -- (0.25,1.5) node[anchor=south]{\footnotesize$s$};
\draw [blue] (1.5,1) -- (1,1.5) node[anchor=south]{\footnotesize$t$};
\draw [purple] (1.5,1) -- (2.75,1.5);
\draw [red]  (1.5,0) -- (0.25,-0.5) node[anchor=north]{\footnotesize$s$};
\draw [blue] (1.5,0) -- (1,-0.5) node[anchor=north]{\footnotesize$t$};
\draw [purple] (1.5,0) -- (2.75,-0.5);
\draw  (1.75,1.3) node{\footnotesize$\cdots$};
\draw  (1.75,-0.3) node{\footnotesize$\cdots$};
\end{tikzpicture}
\;\;=\;\;
\begin{tikzpicture}[baseline=29pt,y=40pt]
\draw[dashed] (-0.5,0) -- (3.5,0);
\draw[dashed] (-0.5,1.5) -- (3.5,1.5);
\draw[red] (1.5,0.75) -- (0,0) node[anchor=north]{\footnotesize$s$};
\draw[blue] (1.5,0.75) -- (0.75,0) node[anchor=north]{\footnotesize$t$};
\draw[purple] (1.5,0.75) -- (3,0);
\draw  (1.75,0.25) node{\footnotesize$\cdots$};
\draw[red] (1.5,0.75) -- (0,1.5) node[anchor=south]{\footnotesize$s$};
\draw[blue] (1.5,0.75) -- (0.75,1.5) node[anchor=south]{\footnotesize$t$};
\draw[purple] (1.5,0.75) -- (3,1.5);
\draw  (1.75,1.2) node{\footnotesize$\cdots$};
\end{tikzpicture}
\end{equation*}

\vspace{-5pt}
\begin{equation*}
\begin{tikzpicture}[baseline=19pt,y=40pt]
\draw[dashed] (-0.5,-0.5) -- (3.5,-0.5);
\draw[dashed] (-0.5,1.5) -- (3.5,1.5);
\draw  (1.75,0.5) node{\footnotesize$\cdots$};
\draw [red] plot [smooth,tension=1.2] coordinates {(1.5,0) (0.25,0.5) (1.5,1)};
\draw [blue] plot [smooth,tension=1.2] coordinates {(1.5,0) (1,0.5) (1.5,1)};
\draw [purple] plot [smooth,tension=1.2] coordinates {(1.5,0) (2.75,0.5) (1.5,1)};
%\draw (1.5,0) node[anchor=north]{\footnotesize$\color{black}\{\color{red}s\color{black},\color{blue}t\color{black}\}$};
%\draw (1.5,1) node[anchor=south]{\footnotesize$\color{black}\{\color{red}s\color{black},\color{blue}t\color{black}\}$};
\draw [blue] (1.5,1) -- (0.25,1.5) node[anchor=south]{\footnotesize$t$};
\draw [red]  (1.5,1) -- (1,1.5) node[anchor=south]{\footnotesize$s$};
\draw [purple] (1.5,1) -- (2.75,1.5);
\draw [red]  (1.5,0) -- (0.25,-0.5) node[anchor=north]{\footnotesize$s$};
\draw [blue] (1.5,0) -- (1,-0.5) node[anchor=north]{\footnotesize$t$};
\draw [purple] (1.5,0) -- (2.75,-0.5);
\draw  (1.75,1.3) node{\footnotesize$\cdots$};
\draw  (1.75,-0.3) node{\footnotesize$\cdots$};
\end{tikzpicture}
\;\;=\;\;
\begin{tikzpicture}[baseline=19pt,y=40pt]
\draw[dashed] (-0.5,-0.5) -- (3.5,-0.5);
\draw[dashed] (-0.5,1.5) -- (3.5,1.5);
%\draw  (1.75,0.5) node{\footnotesize$\cdots$};
%\draw [red] plot [smooth,tension=1.2] coordinates {(1.5,0) (0.25,0.5) (1.5,1)};
%\draw [blue] plot [smooth,tension=1.2] coordinates {(1.5,0) (1,0.5) (1.5,1)};
%\draw [purple] plot [smooth,tension=1.2] coordinates {(1.5,0) (2.75,0.5) (1.5,1)};
%\draw (1.5,0) node[anchor=north]{\footnotesize$\color{black}\{\color{red}s\color{black},\color{blue}t\color{black}\}$};
%\draw (1.5,1) node[anchor=south]{\footnotesize$\color{black}\{\color{red}s\color{black},\color{blue}t\color{black}\}$};
\draw[red,dash pattern= on 4pt off 4pt] (1.5,0) -- (1.5,1);
\draw[blue,dash pattern= on 4pt off 4pt,dash phase=4pt] (1.5,0) -- (1.5,1);
\draw [blue] (1.5,1) -- (0.25,1.5) node[anchor=south]{\footnotesize$t$};
\draw [red]  (1.5,1) -- (1,1.5) node[anchor=south]{\footnotesize$s$};
\draw [purple] (1.5,1) -- (2.75,1.5);
\draw [red]  (1.5,0) -- (0.25,-0.5) node[anchor=north]{\footnotesize$s$};
\draw [blue] (1.5,0) -- (1,-0.5) node[anchor=north]{\footnotesize$t$};
\draw [purple] (1.5,0) -- (2.75,-0.5);
\draw  (1.75,1.3) node{\footnotesize$\cdots$};
\draw  (1.75,-0.3) node{\footnotesize$\cdots$};
\end{tikzpicture}
\;\;=\;\;
\begin{tikzpicture}[baseline=29pt,y=40pt]
\draw[dashed] (-0.5,0) -- (3.5,0);
\draw[dashed] (-0.5,1.5) -- (3.5,1.5);
\draw[red] (1.5,0.75) -- (0,0) node[anchor=north]{\footnotesize$s$};
\draw[blue] (1.5,0.75) -- (0.75,0) node[anchor=north]{\footnotesize$t$};
\draw[purple] (1.5,0.75) -- (3,0);
\draw  (1.75,0.25) node{\footnotesize$\cdots$};
\draw[blue] (1.5,0.75) -- (0 ,1.5) node[anchor=south]{\footnotesize$t$};
\draw[red]  (1.5,0.75) -- (0.75,1.5) node[anchor=south]{\footnotesize$s$};
\draw[purple] (1.5,0.75) -- (3,1.5);
\draw  (1.75,1.2) node{\footnotesize$\cdots$};
\end{tikzpicture}
\end{equation*}

\subsection{Two-color dot contraction}\label{Two-color dot contraction} As an example for $m_{s,t}=3$, let us prove the following relation

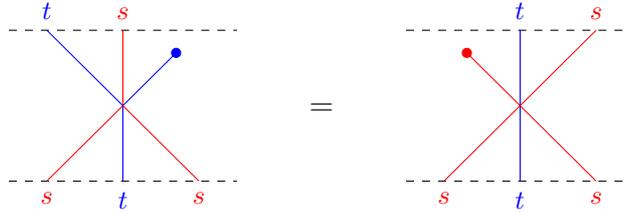
\vspace{-5pt}
\begin{equation*}
\begin{tikzpicture}[baseline=25pt]
\draw[dashed] (0,0) -- (3,0);
\draw[dashed] (0,2) -- (3,2);
%\draw[red] (1.5,0.75) -- (0,0) node[anchor=north]{\footnotesize$s$};
\draw[red] (1.5,1) -- (0.5,0) node[anchor=north]{\footnotesize$s$};
\draw[blue] (1.5,1) -- (1.5,0) node[anchor=north]{\footnotesize$t$};
\draw[red] (1.5,1) -- (2.5,0) node[anchor=north]{\footnotesize$s$};
\draw[blue] (1.5,1) -- (0.5,2) node[anchor=south]{\footnotesize$t$};
\draw[red] (1.5,1) -- (1.5,2) node[anchor=south]{\footnotesize$s$};
\draw[blue] (1.5,1) -- (2.2,1.7);
\draw[blue,fill=blue] (2.2,1.7) circle(0.06);
\end{tikzpicture}
\qquad
=
\qquad
\begin{tikzpicture}[baseline=25pt]
\draw[dashed] (0,0) -- (3,0);
\draw[dashed] (0,2) -- (3,2);
%\draw[red] (1.5,0.75) -- (0,0) node[anchor=north]{\footnotesize$s$};
\draw[red] (1.5,1) -- (0.5,0) node[anchor=north]{\footnotesize$s$};
\draw[blue] (1.5,1) -- (1.5,0) node[anchor=north]{\footnotesize$t$};
\draw[red] (1.5,1) -- (2.5,0) node[anchor=north]{\footnotesize$s$};
\draw[red] (1.5,1) -- (0.8,1.7);
\draw[blue] (1.5,1) -- (1.5,2) node[anchor=south]{\footnotesize$t$};
\draw[red] (1.5,1) -- (2.5,2) node[anchor=south]{\footnotesize$s$};
\draw[red,fill=red] (0.8,1.7) circle(0.06);
\end{tikzpicture}
\end{equation*}
\vspace{1pt}
\noindent
Indeed, the left-hand side and the right-hand side are equal to
$$
i_{(t,s),t,\emptyset}^\star\eta_{t,s}^\star\eta_{s,t\star}=(\eta_{t,s}i_{(t,s),t,\emptyset})^\star\eta_{s,t\star},\qquad
i_{\emptyset,s,(t,s)}^\star\eta_{s,t}^\star\eta_{s,t\star}=(\eta_{s,t}i_{\emptyset,s,(t,s)})^\star\eta_{s,t\star},
$$
respectively. Hence it suffices to prove that $\eta_{t,s}i_{(t,s),t,\emptyset}=\eta_{s,t}i_{\emptyset,s,(t,s)}$.
This follows from the following calculation:
$$
\begin{tikzcd}
{}&{}[g:g':1\rr\arrow[mapsto]{dr}{\eta_{t,s}}&\\
{}[g:g'\rr\arrow[mapsto]{ur}{i_{(t,s),t,\emptyset}}\arrow[mapsto,swap]{dr}{i_{\emptyset,s,(t,s)}}&{}&gg'K\\
{}&{}[1:g:g'\rr\arrow[mapsto,swap]{ur}{\eta_{s,t}}& \end{tikzcd}
$$
For the other values of $m_{s,t}$ the proofs are similar.

\subsection{Horizontal extensions of two-color morphisms}\label{Horizontal extensions of two-color morphisms}
Finally, it remains to extend the morphisms $\eta_{s,t}^\star$ and $\eta_{s,t\star}$ horizontally.
It can be done similarly to Section~\ref{Horizontal_extensions}.
First, let $p=(p_1,\ldots,p_n)$ be a sequence of simple reflections. Then we define the map
$$
\eta_{p,s,t}:\BS(p_1,\ldots,p_n,s,t,\ldots)\to\BS(p_1,\ldots,p_n,\{s,t\})
$$
by $[g_1:\cdots:g_n:g'_1:\cdots:g'_{m_{s,t}}\rr\mapsto[g_1:\cdots:g_n:g'_1\cdots g'_{m_{s,t}}\rr$.
We have the following commutative diagram
$$
\begin{tikzcd}
H(p)\otimes_\R H_K^\bullet(X_{s,t},\k)\arrow{d}[swap]{\theta_{p,s,t}}{\wr}\arrow{r}{\id\otimes\eta_{s,t}^\star}&H(p)\otimes_\R H(s,t,\ldots)\arrow{d}{\theta_{p,[s,t]}}[swap]{\wr}\\
H(p_1,\ldots,p_n,\{s,t\})\arrow{r}{\eta_{p,s,t}^\star}&H(p_1,\ldots,p_n,s,t,\ldots)
\end{tikzcd}
$$
where $\theta_{p,s,t}$ corresponds to
$$
\phi_{p,s,t}:\BS(p_1,\ldots,p_n,\{s,t\})\,{}_K\!\!\times E\to(\BS(p)\,{}_K\!\!\times E)\times(X_{s,t}\,{}_K\!\!\times E)
$$
as in Section~\ref{The_embedding} for the following set of data:
%following case of Theorem~\ref{theorem:1}:
$$
X=G_{p_1}\mathop{\times}\limits_K{G_{p_2}}\mathop{\times}\limits_K\cdots\mathop{\times}\limits_KG_{p_n},\quad Y=X_{s,t},\quad L=R=P=Q=K,
$$
and $\alpha:X\to G$ defined by $[g_1:\cdots:g_n]\mapsto g_1\cdots g_n$.

Now let $q=(q_1,\ldots,q_m)$ be another sequence of simple reflections. We define the map
$$
\eta_{p,s,t,q}:\BS(p_1,\ldots,p_n,s,t,\ldots,q_1,\ldots,q_m)\to\BS(p_1,\ldots,p_n,\{s,t\},q_1,\ldots,q_m)
$$
by $[g_1:\cdots:g_n:g'_1:\cdots:g'_{m_{s,t}}:g''_1:\cdots:g''_m\rr\mapsto[g_1:\cdots:g_n:g'_1\cdots g'_{m_{s,t}}:g''_1:\cdots:g''_m\rr$.
We have the following commutative diagram
$$
\begin{tikzcd}[column sep=35pt]
H(p_1,\ldots,p_n,\{s,t\})\otimes_\R H(q)\arrow{d}[swap]{\theta_{p,s,t,q}}{\wr}\arrow{r}{\eta_{p,s,t}^\star\otimes\id}&H(p_1,\ldots,p_n,s,t,\ldots)\otimes H(q)\arrow{d}{\theta_{p[s,t],q}}[swap]{\wr}\\
H(p_1,\ldots,p_n,\{s,t\},q_1,\ldots,q_m)\arrow{r}{\eta_{p,s,t,q}^\star}&H(p_1,\ldots,p_n,s,t,\ldots,q_1,\ldots,q_m)
\end{tikzcd}
$$
where $\theta_{p,s,t,q}$ is defined similarly to $\theta_{p,s,t}$. We leave it to the reader
to write down similar diagrams for push-forwards, to draw the correspondind diagrams and to compose them
(see, the examples of diagrams in the introduction).

\def\sep{\\[-5pt]}

\bigskip

\noindent
Financial University under the Government of the Russian Federation, \\ Leningradsky Prospekt 49, Moscow, 125993, Russia,
shchigolev\_vladimir@yahoo.com

\end{document}